\documentclass[english]{amsproc} 
\usepackage{bm}
\usepackage{amscd}
\usepackage{amssymb,url,xspace}
\usepackage{paralist}
\usepackage{datetime} 
\usepackage{colortbl}
\usepackage[%colorlinks,linkcolor=BrickRed,
hypertexnames=true,linktocpage]{hyperref} 
\usepackage{mathtools}
\usepackage{euscript}
\usepackage[all]{xy}
\usepackage{graphicx} 

\usepackage{nameref}
 \usepackage{newtxmath}
\usepackage{comment}  
\usepackage{cite}  
\usepackage{thmtools}
\usepackage{enumitem}
\usepackage{letltxmacro} 
\usepackage{nameref}
\usepackage{cleveref}
\usepackage{tikz-cd}  
 \usepackage[utf8]{inputenc}
 \usepackage{newunicodechar} 
\usepackage{euscript} 
\usepackage{mathrsfs}
\usetikzlibrary{trees} %only for drawing trees

\usepackage{epigraph}

\usepackage[hyperpageref]{backref}

\theoremstyle{plain}

\newtheorem{thm}{Theorem}[section]  
\newtheorem{lem}[thm]{Lemma}
\newtheorem{claim}[thm]{Claim} 
\newtheorem{proposition}[thm]{Proposition}
\newtheorem{cor}[thm]{Corollary} 
\newtheorem{conjecture}[thm]{Conjecture}

\theoremstyle{definition}
\newtheorem{dfn}[thm]{Definition}

\theoremstyle{remark}
\newtheorem{rem}[thm]{Remark} 
\newtheorem{example}[thm]{Example}



\theoremstyle{plain}
\newlist{thmlist}{enumerate}{1}
\setlist[thmlist]{  font=\normalfont, label=(\roman*), ref=\thethm.(\roman{thmlisti})}

\addtotheorempostheadhook[thm]{\crefalias{thmlisti}{thm}}

\addtotheorempostheadhook[assumpsion]{\crefalias{thmlisti}{assumption}}

\addtotheorempostheadhook[cor]{\crefalias{thmlisti}{cor}}

\addtotheorempostheadhook[proposition]{\crefalias{thmlisti}{proposition}}

\addtotheorempostheadhook[dfn]{\crefalias{thmlisti}{dfn}}

\addtotheorempostheadhook[lem]{\crefalias{thmlisti}{lem}}
\addtotheorempostheadhook[main]{\crefalias{thmlisti}{main}}

\addtotheorempostheadhook[rem]{\crefalias{thmlisti}{rem}}

\newlist{corlist}{enumerate}{1} % also creates a counter called 'propenumi'
\setlist[corlist]{  font=\normalfont, label=(\roman*), ref=\thecorx.(\roman{corlisti})}%{label=\alph*), ref=\thethmx~(\alph*)}
\crefalias{corlisti}{corx} 

\crefname{lem}{Lemma}{Lemmas} 
\crefname{conjecture}{Conjecture}{Conjectures}
\crefname{thm}{Theorem}{Theorems}
\crefname{proposition}{Proposition}{Propositions}
\crefname{dfn}{Definition}{Definitions}
\crefname{rem}{Remark}{Remarks}
\crefname{cor}{Corollary}{Corollaries}
\crefname{corx}{Corollary}{Corollaries}
\crefname{problem}{Problem}{Problems}
\crefname{thmx}{Theorem}{Theorems}
\crefname{claim}{Claim}{Claims}
\crefname{assumption}{Assumption}{Assumptions}
\crefname{main}{Main Theorem}{Main Theorems}
\crefname{setting}{Setting}{Settings}

\numberwithin{equation}{thm} % Equation numbering follows theorem numbering

\def\ep{\varepsilon}

\def\vhs{\text{\tiny{VHS}}}

\makeatletter
\newcommand*{\rom}[1]{\expandafter\@slowromancap\romannumeral #1@}
\makeatother

\makeatletter     %Replace the Section ...  by the symbol \S ...  in Cref
\newcommand{\crefnames}[3]{%
	\@for\next:=#1\do{%
		\expandafter\crefname\expandafter{\next}{#2}{#3}%
	}%
}
\makeatother

\crefnames{part,chapter,section}{\S}{\S\S}

\newcommand{\sD}{\mathscr{D}}

\newcommand{\sX}{\mathscr{X}}

\newcommand{\cD}{\mathcal D}

\newcommand{\cL}{\mathcal L}

\newcommand{\cO}{\mathcal O}
\newcommand{\cP}{\mathcal P}

\newcommand{\bA}{\mathbb{A}}

\newcommand{\bC}{\mathbb{C}}
\newcommand{\bD}{\mathbb{D}}

\newcommand{\bF}{\mathbb{F}}
\newcommand{\bG}{\mathbb{G}}

\newcommand{\bN}{\mathbb{N}}

\newcommand{\bP}{\mathbb{P}}
\newcommand{\bQ}{\mathbb{Q}}

\newcommand{\bZ}{\mathbb{Z}}

\newcommand{\xsp}{X^{\! \rm sp}}

\def\db{\bar{\partial}}

  \def\spec{\textrm{Spec}\,}
 \def\d{\partial}

\def\sn{\sqrt{-1}}

\def\End{{\rm \small  End}}

\def\Sym{{\text{Sym}}}
\def\vol{\text{\small  Vol}}

\def\btau{{\bm{\tau}}}

\newcommand{\Hom}{{\rm Hom}}

\newcommand{\ord}{{\rm ord}\,}

\newcommand{\ram}{{\rm ram}\,}

\newcommand{\GL}{{\rm GL}}

\makeatletter
\renewcommand{\thmhead}[3]{%
	\thmname{#1}\thmnumber{\@ifnotempty{#1}{ }\@upn{#2}}%
	\thmnote{ {\the\thm@notefont%
			\begingroup
			\let\@origcite\cite
			\def\cite{\@tightcite}%
			\def\@tightcite{\@ifnextchar[{\@tightcite@opt}{\@tightcite@noopt}}%
			\def\@tightcite@opt[##1]##2{\kern-0.5em\@origcite[##1]{##2}\kern-0.05em}%
			\def\@tightcite@noopt##1{\kern-0.3em\@origcite{##1}\kern-0.02em}%
			(#3)%
			\endgroup}}}
\makeatother

\begin{document} \title[Topology \& Hyperbolicity of Algebraic Varieties]{Topology, Hyperbolicity, and the Shafarevich Conjecture for Complex Algebraic Varieties}

	\author[Ya Deng]{Ya Deng}

	\email{ya.deng@math.cnrs.fr, deng@imj-prg.fr}
	\address{CNRS,  
		Institut de Math\'ematiques de Jussieu-Paris Rive Gauche,
		Sorbonne Universit\'e, Campus Pierre et Marie Curie,
		4 place Jussieu, 75252 Paris Cedex 05, France}
	\urladdr{https://ydeng.perso.math.cnrs.fr} 
	\thanks{The  author was supported in part by  the ANR grant Karmapolis (ANR-21-CE40-0010).}
	
%\date{January 1, 1994 and, in revised form, June 22, 1994.}
%\subjclass
	\begin{abstract}
This survey presents recent developments concerning the Shafarevich conjecture, non-abelian Hodge theories, hyperbolicity, and the topology of complex algebraic varieties, as well as the interplay among these areas. 
More precisely, we present the main ideas and techniques involved in the linear versions of the following conjectures: the Shafarevich conjecture, the Chern–Hopf–Thurston conjecture, Kollár’s conjecture on the holomorphic Euler characteristic, the de Oliveira–Katzarkov–Ramachandran conjecture, and Campana’s nilpotency conjecture. In addition, we  discuss characterizations of the hyperbolicity of complex quasi-projective varieties via representations of their fundamental groups, together with  the generalized Green–Griffiths–Lang conjecture   in the presence of a big local system.
	\end{abstract}
	\maketitle
\epigraph{\textit{Parmi les combinaisons que l'on choisira, les plus fécondes seront souvent celles qui sont formées d'éléments empruntés à des domaines très éloignés.}}
{--- \textsc{Henri Poincaré}, \textit{Science et méthode} (1908)}
	\tableofcontents
	\section{Introduction}  
	\subsection{Overview of the Paper}
	The study of the fundamental groups of algebraic varieties lies at the intersection of 
	algebraic topology and algebraic geometry. 
	It traces back to the foundational work of Picard, Lefschetz, Hodge, Hirzebruch,  and Deligne. 
	Subsequent breakthroughs, notably the theory of variations of Hodge structures developed by 
	the school of Griffiths, and the \emph{non-abelian Hodge theory} due to Simpson, 
	Corlette, Gromov--Schoen, and T. Mochizuki, have provided powerful tools for understanding the 
	topology of complex algebraic varieties.
	
	In this survey, I focus on the following question: how do representations of the fundamental group  of the algebraic variety
	into \(\GL_N(K)\), where \(K\) is a field, influence the \emph{geometry} of such  variety? 
	More specifically, we seek to understand how such representations interact with the topology 
	of the variety, its  {hyperbolicity}, and various notions of  {positivity} in algebraic geometry. 
	This perspective is particularly appealing, as it naturally connects diverse areas of mathematics, 
	including harmonic map theory and Nevanlinna theory.
	
	In \cref{sec:topology}, we begin with a classical problem in geometric topology, namely the \emph{Chern--Hopf--Thurston conjecture}, which concerns the sign of the Euler characteristic of closed \emph{aspherical} even-dimensional real manifolds.   If we further assume that such an aspherical manifold carries the structure of a complex projective variety, we are naturally led to the study of algebraic varieties with \emph{large} (or \emph{big}) fundamental groups. We then present recent progress related to this conjecture, as well as a conjecture of Kollár regarding the positivity of Euler characteristic of the canonical bundle of projective varieties with big fundamental groups. 	Finally, we discuss our results on a conjecture of 
	De~Oliveira--Katzarkov--Ramachandran concerning the deformation openness of the property of having 
	a big fundamental group.

	In \cref{sec:hyper}, after discussing several basic examples that illustrate necessary conditions of fundamental groups for hyperbolicity, we present a positive result on the above question by showing that a smooth complex quasi-projective variety admitting a local system over an arbitrary field with semisimple algebraic monodromy group is of log general type and is pseudo Picard or Brody hyperbolic.
	We then turn to the \emph{generalized Green--Griffiths--Lang conjecture}, which relates pseudo 
	Picard or Brody hyperbolicity to  some positivity in algebraic geometry. 
	After reviewing the relatively few known results on this conjecture, we present our solution in 
	the case of varieties admitting a \emph{big local system}.
	
	A central theme of this survey, developed in \cref{sec:sha}, is the \emph{Shafarevich conjecture}, which predicts that the universal cover $\widetilde{X}$ of a smooth complex projective variety $X$ is \emph{holomorphically convex}. After an overview of the historical background, we outline the principal tools used in its study: non-abelian Hodge theory in the archimedean setting, developed by Corlette, Simpson and Mochizuki, and in the non-archimedean setting, initiated by Gromov--Schoen and further developed by Katzarkov, Zuo and Eyssidieux. We then describe our recent progress in extending the Gromov--Schoen theory to quasi-projective varieties. Building on these developments, we explain further recent progress on the Shafarevich conjecture, and sketch a proof in the case of complex projective surfaces with reductive fundamental group, in order to highlight the key ideas underlying the general theory. The lecture notes of Eyssidieux \cite{Eys11}, his seminal paper \cite{Eys04}, and his joint work with Katzarkov, Pantev and Ramachandran \cite{EKPR12} are among the most profound works on this problem. In this survey, we aim to provide new perspectives on the conjecture, and to emphasize that the techniques arising from its study play a crucial role in the proofs of the main results discussed in \cref{sec:topology,sec:hyper}. Since the author's work on the linear Shafarevich conjecture for normal compact K\"ahler varieties \cite{Den26} (see \cref{thm:linear}) was completed after this survey had been finalized, we do not discuss its ideas here, and refer the reader to \cite{Den26}.
	
	Finally, in \cref{sec:topology2,sec:hyperbolicity}, we outline the main ideas underlying the 
	proofs of the results presented in \cref{sec:topology,sec:hyper}, using methods developed in the 
	study of the linear Shafarevich conjecture. 
	In \cref{sec:application}, we present several applications of these techniques, including  Campana’s conjecture on fundamental groups of special varieties;
	a conjecture of Claudon--Höring--Kollár on the structure of universal covers of projective varieties;
	a structure theorem motivated by a conjecture of Kollár.

	\subsection{Algebraic varieties with big fundamental groups}\label{sec:topology}
	\subsubsection{Large fundamental groups: motivation, conjectures and results}
	A well-known conjecture of Hopf from the 1920s states that the Euler characteristic 
	\(\chi(M)\) of a closed \(2n\)-dimensional Riemannian manifold of non-positive sectional curvature
	must satisfy the inequality
	\[
	(-1)^n \chi(M) \geq 0.
	\]
	This conjecture is known in dimensions \(n=1,2\) by the Gauss--Bonnet formula, but it remains
	widely open for \(n\geq 3\).  Note that for any closed odd-dimensional real manifold, 
	its Euler characteristic is zero by the Poincar\'e duality theorem.
	
	By the Cartan--Hadamard theorem, the universal cover \(\widetilde{M}\) of such a manifold 
	is diffeomorphic to \(\mathbb{R}^{2n}\); in particular, it is contractible. 
	Therefore, \(M\) is aspherical, i.e.,  its universal covering is contractible.
	
	In the 1970s, Thurston proposed the following more general conjecture, 
	now referred to as the \emph{Chern--Hopf--Thurston} conjecture.
	\begin{conjecture}[Chern--Hopf--Thurston]\label{conj:CHT}
		Let \(M\) be a closed real \(2n\)-manifold.  
		If \(M\) is aspherical, then it satisfies the inequality
		\[
		(-1)^n \chi(M)\geq 0.
		\]
	\end{conjecture}
	This conjecture is particularly appealing, as both the hypothesis and the conclusion 
	are purely topological: no metric or curvature assumptions are involved.  
	It remains widely open for \(n \geq 2\). Moreover, note that,  if
	a space is aspherical, it becomes a $K(\pi,1)$-space, and thus its
	structure—specifically the Euler characteristic—is determined
	by its fundamental group.
	
	In this survey, we will focus on \cref{conj:CHT} in the setting where \(M\) 
	admits the structure of a complex projective variety. An easy observation is the following.
	\begin{lem}\label{lem:asp}
		Let $X$ be a smooth complex projective variety. If $X$ is aspherical or if its universal covering $\widetilde{X}$ is a Stein manifold, then $X$ has a large fundamental group.
	\end{lem}
	
	We recall the notion of large fundamental group from \cite{Kol95}.
	
	\begin{dfn}[Large fundamental group]\label{def:large}
		A complex quasi-projective normal variety $X$ is said to have \emph{large fundamental group} if, for every closed irreducible positive-dimensional subvariety $Z \subset X$, the image
		\[
		\operatorname{Im}\bigl[\pi_1(Z^{\mathrm{norm}}) \longrightarrow \pi_1(X)\bigr]
		\]
		is an infinite group.
	\end{dfn}
	\begin{rem}
		In alternative definitions of a \emph{large fundamental group}, the normalization of $Z$ is sometimes omitted. Although $\pi_1(Z^{\mathrm{norm}})$ and $\pi_1(Z)$ may differ with respect to finiteness---for example, one may be infinite while the other is finite, as in the case of a rational nodal curve---their respective images in $\pi_1(X)$,
		\[
		\operatorname{Im}\bigl[\pi_1(Z^{\mathrm{norm}})\longrightarrow \pi_1(X)\bigr]
		\quad\text{and}\quad
		\operatorname{Im}\bigl[\pi_1(Z)\longrightarrow \pi_1(X)\bigr],
		\]
		are nevertheless expected to be either both infinite or both finite. This behavior was predicted as a consequence of the \emph{Shafarevich conjecture} (cf.  \cref{conj:Sha}) and was first observed by Gurjar \cite{Gur85}. 
	\end{rem}
\begin{lem}\label{lem:Gurjar}
	Let $X$ be a smooth projective variety. Let $\widetilde{X}\to X$ be the universal cover, and assume that there exists a proper holomorphic fibration $f:\widetilde{X}\to S$ to a complex space $S$ that does not contain any positive-dimensional compact subvarieties. Then, for any closed subvariety $Z \subset X$, 
	$$
	\mathrm{Im}[\pi_1(Z)\to \pi_1(X)] \text{ is finite} \iff \mathrm{Im}[\pi_1(Z^{\mathrm{norm}})\to \pi_1(X)] \text{ is finite.}
	$$
\end{lem}

\begin{proof} 
	First, assume that $\mathrm{Im}[\pi_1(Z^{\mathrm{norm}})\to \pi_1(X)]$ is finite. Let $Y$ be a connected component of $Z^{\mathrm{norm}}\times_X\widetilde{X}$. Then the covering map $Y\to Z^{\mathrm{norm}}$ is finite \'etale. Since $Z$ is a closed subvariety of the projective variety $X$, $Z$ and its normalization $Z^{\mathrm{norm}}$ are compact. Consequently, $Y$ is also compact.
	
Since $f$ is proper, 	the image of  $Y$ under the composition $Y\to \widetilde{X}\xrightarrow{f} S$ is a compact analytic subvariety of $S$. By our assumption on $S$, this image must be a point. Hence, the image of   $Z^{\mathrm{norm}}\times_X\widetilde{X}$ in $S$ is a discrete set of points.
	
	Recall that the natural map $Z^{\mathrm{norm}}\times_X\widetilde{X}\to Z \times_X\widetilde{X}$ is the normalization.  Therefore, the image of $Z\times_X\widetilde{X}$ under  $f$ is exactly the same discrete set of points. 
	
	Now, let $W$ be a connected component of $Z\times_X\widetilde{X}$. By the above, its image in $S$ is a single point $s \in S$, meaning $W$ is contained in the fiber $f^{-1}(s)$. Because $W$ is a closed analytic subspace of $\widetilde{X}$ and the fibration $f$ is proper, the fiber $f^{-1}(s)$ is compact. As a closed subspace of a compact space, $W$ itself must be compact. This implies that the covering map $W \to Z$ has finite degree, and thus $\mathrm{Im}[\pi_1(Z)\to \pi_1(X)]$ is finite.
	
	The converse implication is straightforward. The lemma is proved.
\end{proof}
	\begin{proof}[Proof of \cref{lem:asp}]
		We proceed by contradiction. Assume that there exists a closed subvariety $Z \subset X$ of dimension $k > 0$ such that the image
		\[
		\operatorname{Im}\bigl[\pi_1(Z) \longrightarrow \pi_1(X)\bigr]
		\]
		is finite. Let $Z'$ be a connected component of $\pi_X^{-1}(Z)$, where $\pi_X : \widetilde{X} \to X$. Then the map $Z' \to Z$ is a covering of degree equal to the index of $\operatorname{Im}\bigl[\pi_1(Z) \to \pi_1(X)\bigr]$, which is finite by assumption.  Hence $Z'$ is a compact subvariety of $\widetilde{X}$. Let $f \colon Y \to Z'$ be a desingularization and let $g \colon Y \to X$ be the composite map. 
		
		If $\widetilde{X}$ is Stein, it cannot contain a compact subvariety of positive dimension, which yields an immediate contradiction. 
		
		If $\widetilde{X}$ is contractible (the aspherical case), the composite map
		\[
		H^{2}(X, \mathbb{C}) \longrightarrow H^2(\widetilde{X}, \mathbb{C}) \longrightarrow H^{2}(Y, \mathbb{C})
		\]
		is zero. Fix an ample line bundle $L$ on $X$. Since the map on cohomology is zero, the class $c_1(g^*L) \in H^2(Y, \mathbb{C})$ vanishes. Consequently,
		\[
		\int_{Y} c_1(g^*L)^{k} = 0.
		\]
		However, since $L$ is ample and $g$ is a generically finite map onto its image, this integral must be positive, which gives a contradiction.
	\end{proof}A partial converse to \cref{lem:asp} is suggested by the famous conjecture of Shafarevich.
	\begin{conjecture}[Shafarevich]\label{conj:Sha}
		Let $X$ be a smooth complex projective variety. If $X$ has a large fundamental group, then its universal covering $\widetilde{X}$ is Stein. 
	\end{conjecture} 
	While this conjecture remains open in full generality, a great achievement was made   by Eyssidieux, Katzarkov, Pantev, Ramachandran  in \cite{EKPR12}. 
	\begin{thm}[\cite{EKPR12}]\label{thm:EKPR}
		Let $X$ be a smooth complex projective variety. If there exists a faithful representation $\varrho:\pi_1(X)\to \GL_{N}(\bC)$, then $\widetilde{X}$ is holomorphically convex. In particular, if additionally $X$ has large fundamental group, then $\widetilde{X}$ is Stein. 
	\end{thm}  We will return to \cref{conj:Sha,thm:EKPR} later.
	
	A stronger converse to \cref{lem:asp} would ask if a large fundamental group implies asphericity. This is related to a question posed by Koll\'ar (cf. \cite[\S 0.3.1]{Kol95}): is the fundamental group of a projective variety commensurable (up to finite kernels) with a group $G'$ that admits a quasi-projective $K(G', 1)$?
	
	However, Dimca, Papadima, and Suciu provided a negative answer in \cite{DPS09}. They constructed a smooth projective variety $X$ whose universal covering is Stein (satisfying the condition of a large fundamental group), yet $\pi_1(X)$ is not commensurable to any group admitting a quasi-projective $K(\pi,1)$. In particular, although $X$ has a large fundamental group, its universal cover is not contractible, meaning $X$ is not aspherical. This demonstrates that the asphericity of projective varieties is a strictly stronger condition than having a large fundamental group.

	Motivated by this viewpoint, Arapura and Wang 
	\cite[Conjecture~1.2]{AW25} proposed the following conjecture.
	\begin{conjecture}[Arapura-Wang]\label{conj:AW}
		Let $X$ be a   complex projective $n$-fold with large fundamental group. 
		Then for any perverse sheaf $\cP$ on $X$, one has
		\[
		\chi(X,\cP)\geq 0. 
		\] In particular, $(-1)^n\chi(X)\geq 0$. 
	\end{conjecture}
	Observe that the shifted constant sheaf $\underline{\mathbb{C}}[n]$ is a perverse sheaf. Consequently, we have
	$$
	\chi(X, \underline{\mathbb{C}}[n]) = (-1)^n \chi(X).
	$$Therefore, \cref{conj:AW} is broader than \cref{conj:CHT} in another direction. In \cite{AW25}, Arapura and Wang proved their conjecture assuming the existence of a faithful and cohomologically rigid representation $\pi_1(X)\to \GL_N(\mathbb{C})$. In the author's joint work with Wang \cite{DW24}, we build upon the strategy in \cite{AW25}, and leverage non-abelian Hodge theory alongside techniques from \cite{Eys04, EKPR12, DY23} to extend this result to a more general setting:
	\begin{thm}[\cite{DW24}]\label{thm:CHT}
		Let $X$ be a smooth complex projective variety. If there exists a  large representation  $\varrho:\pi_1(X)\to \GL_N(K)$, where $K$ is any field,  then
		$$
		\chi(X,\cP)\geq 0
		$$
		for any perverse sheaf on $X$. 
	\end{thm}
	In particular, we proved \cref{conj:AW} when $\pi_1(X)$ is {linear}, i.e., there exists a faithful representation $\pi_1(X)\to \GL_{N}(K)$. Here the definition of  {large representation} is analoguous to \cref{def:large}. The representation $\varrho:\pi_1(X)\to \GL_N( {K})$, is called \emph{large} if, for every closed irreducible positive-dimensional subvariety $Z \subset X$, the image
	$$
	\varrho(\mathrm{Im} [\pi_1(Z^{\mathrm{norm}}) \longrightarrow \pi_1(X)])
	$$
	is an {infinite group}.
	We shall discuss the proof of \cref{thm:CHT} in the subsequent sections. 
	\subsubsection{Big fundamental groups}
	Note that a variety with  large fundamental group must be \emph{minimal}, i.e.,  it cannot contain any rational curves or $\mathbb{A}^1$-curves. Consequently, the blow-up of an algebraic variety never has large fundamental group. To address this limitation, Koll\'ar \cite{Kol93,Kol95} introduced a more natural and birationally invariant notion:
	
	\begin{dfn}[Big fundamental group]\label{def:big}
		A complex quasi-projective normal variety $X$ is said to have a \emph{big fundamental group} if, for every closed irreducible positive-dimensional subvariety $Z \subset X$ passing through a very general point (i.e., there exists a countable union of proper closed subvarieties $Z_i \subsetneq X$ such that $Z \not\subset \bigcup_{i=1}^\infty Z_i$), the image
		\[
		\operatorname{Im}\bigl[\pi_1(Z^{\mathrm{norm}}) \longrightarrow \pi_1(X)\bigr]
		\]
		is infinite.
	\end{dfn}
	
	This condition is birationally invariant: if $f : X \to Y$ is a proper birational morphism between quasi-projective normal varieties, then $X$ has a big fundamental group if and only if $Y$ does. In \cite{Cam94}, varieties with big fundamental groups are also referred to as being of \emph{$\pi_1$-general type}.
	
	We shall now provide some examples of varieties with large or big fundamental groups.
	\begin{example}\label{exa:big}
		Let $X$ be a complex quasi-projective normal variety. Then:
		\begin{itemize}
			\item If $X$ is a quasi-projective curve different from $\mathbb{C}$ and $\mathbb{P}^1$, then it has large fundamental group.
			
			\item If the universal cover of $X$ is contractible, or a bounded symmetric domain, or more generally a Stein manifold, then $X$ has large fundamental group (cf. \cref{lem:asp}).
			
			\item By the above result, if $X$ is an abelian variety, then it has large fundamental group.
			
			\item If $f : X \to Y$ is a finite \'etale morphism, then $X$ has a large (resp. big) fundamental group if and only if $Y$ does.

			\item Let $f: X \to Y$ be a morphism to another quasi-projective normal variety $Y$.  
			If $Y$ has large (resp. big) fundamental group and if $f$ is quasi-finite (resp.  generically finite onto its image), then $X$ has large (resp. big) fundamental group.
			
			\item As a consequence, if $X$ has maximal quasi-Albanese dimension, i.e., the quasi-Albanese map 
			\[
			X \longrightarrow {\rm Alb}(X)
			\]
			is generically finite onto its image, then $X$ has big fundamental group.
			
			\item Suppose that $X$ carries a $\mathbb{Z}$-VHS.  
			If the period map 
			\[
			p : X \longrightarrow \sD/\Gamma,
			\]
			where $\sD$ is the associated period domain and $\Gamma$ is the monodromy group, is quasi-finite (resp.\ generically finite onto its image), then $X$ has large (resp.\ big) fundamental group.
		\end{itemize}
	\end{example}  
	Note that \cref{conj:AW} concerns the topological aspect of varieties with large 
	fundamental groups. In general, the statement fails if we replace the condition 
	“large fundamental group’’ by “big fundamental group”, since the topological 
	Euler characteristic is not a birational invariant.  
	However, one basic observation is that if 
	$f:X\to Y$ is a birational morphism between two smooth projective varieties, then
	\[
	\chi(X,K_X)=\chi(Y,K_Y),
	\]
	where $K_X$ and $K_Y$ denote the canonical bundles of $X$ and $Y$.  
	Recall that
	\[
	\chi(X,K_X)=\sum_{i=0}^{\dim X}(-1)^i \dim H^i(X,K_X).
	\]
	Motivated by this observation, Koll\'ar proposed in \cite{Kol95} the following conjecture, 
	which may be viewed as a sheaf-theoretic analogue of \cref{conj:CHT}.
	
	\begin{conjecture}[Koll\'ar]\label{conj:kollar}
		Let $X$ be a smooth projective variety.  
		If $X$ has big fundamental group, then $\chi(X,K_X)\geq 0$.  
	\end{conjecture}
	
	Notice that \cref{conj:kollar} holds when $X$ has maximal Albanese dimension, 
	thanks to the generic vanishing theorem of Green--Lazarsfeld \cite{GL87}.  
	Indeed, in this case, \cite{GL87} implies that for a generic topologically trivial 
	line bundle $L\in \mathrm{Pic}^0(X)$, 
	\[
	H^i(X,K_X\otimes L)=0 \qquad \text{for all } i\geq 1.
	\]
	Hence,
	\[
	\chi(X,K_X)
	= \chi(X,K_X\otimes L)
	= \sum_{i=0}^{\dim X} (-1)^i \dim H^i(X,K_X\otimes L)
	= \dim H^0(X,K_X\otimes L)
	\geq 0.
	\] 
	This proof inspired us that some vanishing theorem for cohomology will be useful to prove \cref{conj:kollar}. However, it seems that we do not have the generic vanishing theorem for general varieties with big fundamental group. In \cite{DW24b}, for the case where $\pi_1(X)$ is linear, Wang and the author established a vanishing theorem for the $ {L^2}$-Dolbeault cohomology of $\widetilde{X}$, analogous to the Green--Lazarsfeld theorem.

	\begin{thm}[\cite{DW24b}]\label{main:kollar}
	Let $X$ be a smooth projective variety.	If there exists a big representation $\varrho:\pi_1(X)\to\GL_N(\bC)$, then the following properties hold:
		\begin{thmlist} 
			\item  \label{item:vanishing} for the universal covering $\widetilde{X}$ of $X$,  $H^{p,0}_{(2)}(\widetilde{X})=0$   for  $0\leq p\leq n-1$ and $H^{n,q}_{(2)}(\widetilde{X})=0$ for  $1\leq q\leq n$. 
			\item \label{item:Euler}The Euler characteristic of $K_X$, denoted by  $\chi(X, K_X)$, is non-negative.
			\item  \label{item:strict} If the strict inequality $\chi(X,K_X)>0$ holds, then 
			\begin{enumerate}[label*=(\alph*)]
				\item \label{item:converse2}there exists  {a nontrivial} $L^2$-holomorphic $n$-form on $\widetilde{X}$;
				\item  \label{item:converse} $X$ is of general type. 
			\end{enumerate}  
		\end{thmlist} 
	\end{thm}
	The representation $\varrho:\pi_1(X)\to \GL_N(\mathbb{K})$, is called \emph{big} if  for every closed irreducible positive-dimensional subvariety $Z \subset X$ passing through a very general point, the image
	\[
	\varrho(\mathrm{Im} [\pi_1(Z^{\mathrm{norm}}) \longrightarrow \pi_1(X)])
	\]
	is infinite.
	\begin{dfn}[$L^2$-cohomology]\label{defn:L2} 
		Let $(Y,\omega)$ be a complete K\"ahler manifold.  Let   $L_{(2)}^{p,q}(Y)$  be the space of $L^2$-integrable   $(p,q)$-forms with respect to the metric $\omega$. A section $u$ is said to be in  ${\rm Dom}\, \db$  if  $\db u$  calculated in the sense of distributions is still in $L^2$.  Then  the $L^2$-Dolbeault cohomology is defined as
		\[ H^{p,q}_{(2)}(Y)= {{\rm ker}\, \db}\big/\,{\overline{{\rm Im}\, \db \cap {\rm Dom}\, \db}}.
		\]  
If $(X,\omega)$ is a compact K\"ahler manifold, we write $H^{p,q}_{(2)}(\widetilde{X})$ for the $L^2$-cohomology computed with respect to the metric $\pi_X^*\omega$. 
Note that this space is independent of the choice of the K\"ahler metric $\omega$, since any two such pullback metrics on $\widetilde{X}$ are mutually bounded. 
	\end{dfn}  
As is well-known to experts, once we have the vanishing of $L^2$-cohomology in \cref{item:vanishing}, one can use  Atiyah's $ {L^2}$-index theorem to prove that $\chi(X, K_X)\geq 0$. This proves \cref{conj:kollar} under the assumptions in \cref{main:kollar}.  We explain the proof below. 
	\begin{proof}[Proof of $(i)\implies (ii)\& (iii)$ in \cref{main:kollar}]
		We denote by $\Gamma=\pi_1(X)$ and $\dim_{\Gamma}H^{n,q}_{(2)}(\widetilde{X})$ the Von Neumann dimension of $H^{n,q}_{(2)}(\widetilde{X})$ (cf. \cite{Ati76} for the definition). 
		By Atiyah's $L^2$-index theorem along with \cref{item:vanishing}, we have
		\begin{align} \label{eq:euler}	\chi(X,K_X)=\sum_{q=0}^{n}(-1)^q\dim_{\Gamma}H^{n,q}_{(2)}(\widetilde{X})= \dim_{\Gamma}H^{n,0}_{(2)}(\widetilde{X})\geq 0.
		\end{align} 
		\Cref{item:Euler} is proved. 
		
If the strict inequality~\eqref{eq:euler} holds, then $H^{n,0}_{(2)}(\widetilde{X}) \neq 0$. Thus, one can choose a non-trivial element $f \in H^{n,0}_{(2)}(\widetilde{X})$. In~\cite[Chapter~13]{Kol95}, Koll\'ar, building on ideas of Gromov \cite{Gro92}, introduced the \emph{Poincar\'e series}
\[
P_k(f) := \sum_{\gamma \in \Gamma} \gamma^*\!\big(f^{2k}\big)(\gamma x),
\]
and showed that for each $k \in \mathbb{N}$, the series $P_k(f)$ defines a $\Gamma$-invariant holomorphic section of $2kK_{\widetilde{X}}$. Hence it descends to a pluricanonical section
\[
f_k \in H^0(X, 2kK_X).
\]
 Koll\'ar then considers the linear series generated by products of these sections:
\[
R_{2m}
\;:=\;
\Big\{
\textstyle\prod_{i} f_{k_i}
\;\Big|\;
\sum_{i} k_i = m
\Big\}
\;\subset\; H^0(X, 2mK_X).
\]
In~\cite[Theorem~13.9]{Kol95}, he proves that for all sufficiently large $m$, the rational map associated with $R_{2m}$ is birational onto its image, assuming that $X$ has big fundamental group. This shows that $K_X$ is big, and thus proves~\Cref{item:strict}.

We refer the interested reader to~\cite[Chapter~13]{Kol95} for further details.
 \end{proof}
 \begin{rem}
 Indeed, Koll\'ar's argument above also shows that for any holomorphic line bundle $L$ on $X$ with big fundamental group, the existence of a non-zero $L^2$-section $H_{(2)}^0(\widetilde{X},L) \neq 0$ implies that $L$ is a big line bundle. 
 \end{rem}
	%It was also asked by Koll\'ar \cite{Kol95} whether a smooth projective variety $X$
	% of general type with a large fundamental group must satisfy the strict inequality
	% $\chi(X)>0$. However, counterexamples were given by Green--Lazarsfeld and later
	% refined by Chen--Debarre--Jiang.  We propose the following
	% conjecture. 
	%\begin{conjecture}
	%	Let $X$ be a smooth projective variety. Assume that there exists a representation
	% 	$\varrho:\pi_1(X)\to G(\bC)$, where $G$ is a semisimple algebraic group, such that
	%	$\varrho(\pi_1(X))$ is Zariski dense. Then
	%	\[
	%	\chi(K_X)>0.
	%	\]
	% \end{conjecture} 
% We will prove this conjecture in the case where $\varrho$ is rigid and integral.

%We will outline the proof strategy and the key ideas of results  in
% the subsequent sections.

\subsubsection{Deformation of big fundamental groups}
From \cref{def:big,def:large}, one can see that the properties of having big or large fundamental groups  depends essentially on the algebraic structures of the varieties, as one has to make the test for all subvarieties or all subvarieties passing through a general point.  It is natural to ask whether given two homeomorphic algebraic varieties, one has big fundamental group  if and only if the other has big fundamental group. This question holds trivially for the curves. For surfaces, it was  answered positively by  Claudon in  \cite{Cla10}, based on a theorem by Siu \cite{Siu87}.

We first define a topological invariant $g(X)$ for any compact K\"ahler manifold $X$ as follows.  
A vector subspace $V \subset H^1(X,\mathbb{R})$ is said to be \emph{isotropic} 
if it is annihilated by the exterior product
\[
\Lambda^2 H^1(X,\mathbb{R}) \longrightarrow H^2(X,\mathbb{R}),
\]
that is, if $\alpha \wedge \beta = 0$ for all $(\alpha,\beta) \in V \times V$.
We then set
\[
g(X) = \max \left\{ \dim(V) \mid V \subset H^1(X,\mathbb{R}),\ 
V \text{ is isotropic} \right\}.
\] 
We remark that the invariant $g(X)$ is an invariant for fundamental groups. Namely,     if $Y$ is another  compact K\"ahler manifold and $\pi_1(X)\simeq \pi_1(Y)$, then $g(X)=g(Y)$ (see \cite{ABCKT,Py25}). 
\begin{thm}[{\cite{Siu87}}] \label{thm:Siu}
	Let $X$ be a compact K\"ahler manifold and let $g \geq 2$ be an integer. 
	Then $\pi_1(X)$ admits the surface group $\pi_1(C_g)$ as a quotient 
	if and only if $X$ admits a fibration onto a  curve of genus 
	$g' \geq g$. 
	Moreover, if 
	\[
	\rho : \pi_1(X) \longrightarrow \pi_1(C_g)
	\]
	is a surjective homomorphism with $g = g(X)$, then there exists a 
	fibration of $X$ onto a curve of genus $g$ that induces $\rho$. 
\end{thm}

\begin{thm}[Claudon]\label{thm:Claudon}
	Let $X$ and $Y$ be two compact K\"ahler surfaces that are homeomorphic. Then $X$ has big fundamental groups if and only if $Y$  has big fundamental group. 
\end{thm}  
\begin{proof}
	Assume that  $X$ does not have big fundamental group. If $\pi_1(X)$ is finite, then $\pi_1(Y)$ is also finite as they are  homeomorphic. Therefore, we just consider the case $\pi_1(X)$ is infinite. In this case, a theorem by Koll\'ar \cite{Kol93} 
	and Campana \cite{Cam94} states that, after replacing $X$ by a finite \'etale cover,  there is a fibration $f:X\to C_g$ onto a smooth projective curve $C_g$ of genus  $g\geq 1$  such that for  a  smooth general  fiber $F$ of $f$,  ${\rm Im}[\pi_1(F)\to \pi_1(X)]$ is a finite group, and we have the following short exact sequence
	$$
	0\to \pi_1(F)\to \pi_1(X)\to \pi_1(C_g) \to 0.
	$$ 
	This means that $\pi_1(X)$ is commensurable to a surface group $\pi_1(C_g)$.   After replacing $X$ by a finite \'etale cover,  $\pi_1(X)\simeq \pi_1(C_{g'})$ for another  projective curve $C_{g'}$ of genus $g'\geq 1$.  Hence we can replace $Y$ by a finite \'etale cover such that $\pi_1(Y)\simeq \pi_1(C_{g'})$.  
	
	If $g'=1$, then $\pi_1(Y)$ is abelian,  and  one can see that   the Albanese map ${\rm alb}_Y:Y\to {\rm Alb}(Y)$ of $Y$  is surjective, and  ${\rm Alb}(Y)$ is an elliptic curve.  Moreover, ${\rm alb}_Y$ induces an isomorphism of fundamental groups. Therefore, $Y$ does not have big  fundamental groups, as for each fiber $F$ of ${\rm alb}_Y$, we have ${\rm Im}[\pi_1(F)\to \pi_1(Y)]=\{1\}$. 
	
	If $g'\geq 2$, by the above arguments together with \cref{thm:Siu}, there exists a fibration $h:Y\to C_{g'}'$ to a projective curve $C_{g'}'$ with genus $g'$ such that $h_*:\pi_1(Y)\to \pi_1(C_{g'}')$ is an isomorphism. Therefore, for each fiber $F$ of $h$, we have ${\rm Im}[\pi_1(F)\to \pi_1(Y)]=\{1\}$.  This implies that $Y$ does not have big funmental group.  The theorem is proved.  
\end{proof}
However, it remains unknown whether \cref{thm:Claudon} holds for smooth projective varieties of dimension greater than two. In \cite{DOKR}, De Oliveira, Katzarkov, and Ramachandran proposed the following conjecture; we also refer to the monograph by Koll\'ar \cite{Kol95} for a discussion of related problems.
\begin{conjecture}\label{conj:kat}
	Let $f:\sX\to\bD$ be a smooth projective family over the unit disk. If $X_0:=f^{-1}(0)$ has big fundamental group, then any $X_t:=f^{-1}(t)$ also has big fundamental group for small $t$. 
\end{conjecture}
In \cite{Cla10}, Claudon proved \cref{conj:kat} for threefolds, assuming suitable conjectures on the fundamental groups of orbifold surfaces.

In a joint work with Mese and Wang  \cite{DMW24}, we prove \cref{conj:kat} in the linear case. 
\begin{thm}[\cite{DMW24}]\label{conj:deformation}
	Let $f:\sX\to\bD$ be a smooth projective family over the unit disk. If there exists a big representation $\varrho:\pi_1(X_0)\to \GL_N(\bC)$, then there exists a representation $\tau:\pi_1(\sX)\to\GL_{N}(\bC)$ such that   for $|t| $ sufficiently small, the representation
	$$
	\tau_t: \pi_1(X_t)\stackrel{\simeq}{\to}\pi_1(\sX)\to \GL_N(\bC)
	$$ 
	is big. In particular, 
	\cref{conj:kat} holds if $\pi_1(X_0)$ is linear. 
\end{thm}
To prove \cref{conj:deformation}, one must establish the deformation continuity of equivariant harmonic mappings into symmetric spaces or Euclidean buildings (cf. \cref{harmonic2,lem:continuity} below).

\subsection{Hyperbolicity and fundamental groups: conjectures and main results}\label{sec:hyper}

In this subsection, we illustrate how the notions of  {big} or  {large} fundamental groups, introduced previously, play a pivotal role in characterizing the hyperbolicity of algebraic varieties. We then formulate a conjectural characterization (cf. \cref{conj:funhyp}) along these lines and present our results toward this conjecture. Furthermore, we discuss our main results concerning the generalized Green--Griffiths--Lang conjecture under suitable assumptions on the fundamental group. The strategies for the proofs of the main theorems are detailed in \cref{sec:hyperbolicity}. These proofs rely on the technical machinery developed for the reductive Shafarevich conjecture in \cref{sec:sha}. 
\subsubsection{Notions of hyperbolicity}
The notion of hyperbolicity  originates from Picard's great theorem and Picard's little theorem on  the range of an analytic function.  
\begin{thm}[Little Picard theorem]\label{thm:little}
	Any holomorphic map  $f:\bC\to \bP^1\backslash\{0,1,\infty\}$  must be constant.  
\end{thm}
This  theorem is a significant strengthening of Liouville's theorem which states that the image of an entire non-constant function must be unbounded. 
\begin{thm}[Great Picard theorem]\label{thm:great}
	Any holomorphic map  $f:\bD^*\to \bP^1\backslash\{0,1,\infty\}$   does not have essential singularity at the origin. 
\end{thm}
This is a substantial strengthening of the Casorati–Weierstrass theorem, which only guarantees that the range of  a holomorphic function defined over $\bD^*$ with essential singularity at the origin has image  dense in $\bC$.  One can see that \cref{thm:great} implies \cref{thm:little}. 

The complex algebraic varieties that have the similar  properties as described in  \cref{thm:little,thm:great} is called \emph{hyperbolic}. Precisely, we have the following definition.
\begin{dfn}[Hyperbolicity]
	Let $X$ be a complex quasi-projective variety.  
	\begin{thmlist}
		\item The variety $X$ is \emph{pseudo Picard hyperbolic} if  there is a proper Zariski closed subset $\Xi\subsetneqq X$  such that  any holomorphic map $f:\bD^*\to X$ from the punctured disk $\bD^*$ to $X$ with $f(\bD^*)\not\subseteq \Xi$ extends to  a holomorphic map from the disk $\bD$ to a projective compactification $\overline{X}$ of $X$.
		\item  We say that the variety $X$ is \emph{pseudo Brody hyperbolic} if there exists a proper Zariski closed subset $\Xi \subsetneqq X$ such that every non-constant holomorphic map $f : \mathbb{C} \to X$ (an \emph{entire curve}) has image contained in $\Xi$. 
	\end{thmlist}
\end{dfn}
Note that every pseudo Picard hyperbolic variety is pseudo Brody hyperbolic. While we conjecture the converse to hold true, as of now, we lack both a proof and any counter-example of our conjecture.

In the algebraic setting, we introduce the following definition.
\begin{dfn}[Strongly of log general type]
	Let $X$ be a complex quasi-projective variety. We say that $X$ is \emph{strongly of log general type}\footnote{This terminology originates from Demailly \cite{Dem15} in his strategy
		for the proof of Green-Griffiths-Lang conjecture. 
		Although our notion differs from Demailly’s original definition,
		we retain the same terminology for its descriptive convenience.}
	 if there exists a proper Zariski closed subset $\Xi\subsetneqq X$ such that every positive-dimensional closed subvariety of $X$ not contained in $\Xi$ is of log general type.
\end{dfn} 
Lang conjectured that a complex quasi-projective variety is strongly of log general type if it is of log general type. To my knowledge, this conjecture remains open  even for complex surfaces in general.

\subsubsection{Generalized Green-Griffiths-Lang conjecture}
To characterize algebraic varieties falling into the hyperbolic category, we will start by examining cases where $C$ is a smooth quasi-projective curve, with $\overline{C}$ as its compactification, and $D$ representing the complement of $C$ within $\overline{C}$.   
\begin{figure}\label{figure}
	\centering
	\caption{Hyperbolicity from different viewpoints}
	\begin{tabular}{| m{5cm} | m{2.5cm} | m{2cm} | m{3cm} | }
		\hline
		& $\deg (K_{\overline{C}}+D)$ & $\pi_1(C)$ & Hyperbolicity \\
		\hline\hline
		$\mathbb{P}^1, \mathbb{C}$ & $<0$ & $\{1\}$ & no \\
		\hline
		$\mathbb{C}^*$, torus & $=0$ & Infinite, abelian & no \\
		\hline
		$\mathbb{P}^1\backslash \{\text{at least three points}\}$  torus$\backslash \{\text{at least one point}\}$ $\cdots$ & $>0$ & Infinite, non-abelian & yes \\
		\hline
	\end{tabular}
\end{figure}
As illustrated in Figure~1, we can make the following observations. 
From an algebraic–geometric perspective, hyperbolic curves can be characterized as those whose 
logarithmic canonical bundle is positive. 
On the other hand, from a topological viewpoint, hyperbolic curves are precisely those with 
infinite and non-abelian fundamental groups.

It's worth noting that the generalized Green-Griffiths-Lang conjecture aligns with the algebraic geometric viewpoint, focusing on the positivity of the logarithmic canonical bundle.
\begin{conjecture}[Generalized  Green-Griffiths-Lang conjecture]\label{conj:GGL}
	Let $X$ be a smooth quasi-projective variety.  Then the following properties are equivalent:
	\begin{thmlist} 
		\item  $X$ is of log general type;
		\item  $X$ is pseudo-Picard hyperbolic;
		\item  $X$ is pseudo-Brody hyperbolic;
		\item $X$ is  {strongly of log general type}.
	\end{thmlist}
\end{conjecture}
So far \cref{conj:GGL} remains an open and challenging problem, even in situations where $X$ is a surface. We   are fascinated by this conjecture   due to its analogy with the Bombieri-Lang conjecture concerning rational points. 
\begin{conjecture}[Bombieri-Lang]
	Let $X$ be a smooth projective variety defined over a number field $k$. Then there exists a  Zariski closed subset $\Xi\subsetneqq X$ such that  for all   number field extensions  $k'$ of $k$, the set of $k'$-rational points in $X\backslash \Xi$ is finite. 
\end{conjecture}

\subsubsection{Hyperbolicity of compactifications after taking finite \'etale coverings} 
It is natural to ask why we are interested in the more general notion Picard hyperbolicity. It indeed enjoys the following algebraicity property.
\begin{proposition}[\cite{Denarxiv}] \label{extension theorem} 
	Let $X$ be a smooth quasi-projective variety that is pseudo Picard hyperbolic. Then there exists a proper Zariski closed subset $\Xi\subsetneq X$ such that  any meromorphic map $f:Y\dashrightarrow X$ from another smooth quasi-projective variety $Y$ to $X$ satisfying $f(Y)\not\subset \Xi$ is \emph{rational}.
\end{proposition} 
A direct consequence of \cref{extension theorem}  is the following uniqueness of algebraic structure of pseudo Picard hyperbolic varieties. 

\begin{cor}[\cite{Denarxiv}] 
	Let $X$ and $Y$ be smooth quasi-projective varieties such that there exists an analytic   isomorphism $\varphi:Y^{\rm an}\to X^{\rm an}$ of associated complex spaces.
	Assume that $X$ is pseudo Picard hyperbolic. 
	Then $\varphi$ is an algebraic isomorphism. \qed
\end{cor} 

Let us discuss some examples of pseudo Picard hyperbolic varieties. A classical result due to Borel \cite{Bor72} and Kobayashi-Ochiai \cite{KO71} is that quotients of bounded symmetric domains by torsion-free lattices are Picard hyperbolic. In \cite{Denarxiv} we  proved a similar result for algebraic varieties that admit a complex variation of Hodge structures.
\begin{thm}[{\cite[Theorem A]{Denarxiv}}]\label{thm:PicardVHS}
	Let $X$ be a smooth quasi-projective variety. Assume that there is a $\bC$-VHS on $X$ whose period mapping is injective at one  point. Then $X$ is pseudo Picard hyperbolic. \qed
\end{thm}   
A similar result was discussed in \cite{BB20}.

In \cite{Nad89}, Nadel proved the nonexistence of certain level structures on abelian varieties over complex function fields, which was refined by Aihara-Noguchi and  Rousseau in \cite{AN91,Rou16}. Precisely, they proved the following theorem:
\begin{thm}[\cite{Nad89,Rou16}]
	Let $X$ be a smooth quasi-projective variety such that $X=\Omega/\Gamma$ where $\Omega$ is a bounded symmetric domain and $\Gamma$ is an arithmetic torsion free lattice acting on $\Omega$. Then there exists a finite index subgroup $\Gamma'\subset \Gamma$, such that  for the quasi-projective variety $X':=\Omega/\Gamma'$, its projective compactification $\overline{X'}$ is Brody (moreover Kobayashi) hyperbolic modulo the boundary $\overline{X'}\backslash X'$. 
\end{thm}
In \cite{Denarxiv},   we obtained the following result which incorporates previous results by Nadel, Aihara-Noguchi and Rousseau. 
\begin{thm}[{\cite[Theorem B]{Denarxiv}}]\label{thm:Picard}
	Let $X$ be a smooth quasi-projective variety. Assume that there is a complex variation of  Hodge structures on $X$ whose period mapping is injective at one  point.  Then there exists a finite \'etale cover $X'$ of $X$ such that  its projective compactification $\overline{X'}$ is pseudo Picard hyperbolic and strongly of general type.  
\end{thm}
The proofs of \cref{thm:PicardVHS,thm:Picard} in \cite{Denarxiv} are rather involved and rely heavily on analytic techniques from non-abelian Hodge theory. Subsequently, in \cite{CD21}, Cadorel and the author gave a simplified proof and, moreover, established the following more general result. 
\begin{thm}[\cite{CD21}]
	Let $X$ be a smooth quasi-projective variety. Assume that there exists a harmonic bundle $(E,\theta,h)$ on $X$ such that the Higgs field $\theta\colon T_X \to \End(E)$ is injective at some point of $X$. Then there exists a finite \'etale cover $X’$ of $X$ whose projective compactification $\overline{X’}$ is pseudo-Picard hyperbolic.
\end{thm} 
Nonetheless, Nevanlinna theory remains an essential ingredient in both works.
\subsubsection{How Fundamental Groups Determine Hyperbolicity} 
It is natural to ask whether there is  a characterization of   hyperbolicity of algebraic varieties in terms of fundamental groups. 
As illustrated in Figure 1,   such characterization requires that the fundamental group $\pi_1(X)$  must be  infinite and non-abelian.  However, this topological requirement is insufficient, as demonstrated by non-hyperbolic varieties whose fundamental groups meet this criterion. 
\begin{example} \label{example1}
	Let $C$ be a projective curve of genus $g \geq 2$. The product variety $X = C \times \mathbb{P}^1$ is clearly not pseudo Brody hyperbolic due to the $\mathbb{P}^1$ factor. Its fundamental group is $\pi_1(X) \simeq \pi_1(C)$, which is both infinite and non-abelian, showing the necessity of additional constraints.
\end{example} 
The failure of hyperbolicity in \Cref{example1} can be attributed to the fact that the non-hyperbolic factor $\mathbb{P}^1$ trivializes the fundamental group in a relative sense. This suggests that a useful condition must ensure that the fundamental group remains \emph{infinite} when restricted to relevant subvarieties. So having \emph{big fundamental groups} would indeed exclude the counterexample $C \times \mathbb{P}^1$. 

However, simply having a big fundamental group is not enough, as shown by varieties whose non-hyperbolicity stems from a different group-theoretic defect: 
\begin{example} \label{example2}
	Let $C$ be a projective curve of genus $g \geq 2$ and let $E$ be an elliptic curve. The variety $X = C \times E$ is not pseudo Brody hyperbolic. Its fundamental group is $\pi_1(X) \simeq \pi_1(C) \times \pi_1(E)$. Since $\pi_1(E) \simeq \mathbb{Z}^2$, $\pi_1(X)$ is large and non-abelian.
\end{example} 
The non-hyperbolicity of $X = C \times E$ is tied to the presence of normal abelian subgroup. More precisely, for any fixed point $x \in C$, the subgroup corresponding to the elliptic curve factor,
$$
{\rm Im}\left[\pi_1(\{x\}\times E)\to \pi_1(C\times E)\right] \triangleleft \pi_1(C\times E),
$$
is an infinite normal abelian subgroup of $\pi_1(X)$. Therefore, we should require that  fundamental groups are   ``highly non-abelian". 
\begin{dfn}[Semisimple Group]\label{dfn:semisimple}
A finitely generated group $G$ is called \emph{semisimple} if every abelian normal subgroup of $G$ is finite. 
\end{dfn} 
The variety $X = C \times E$ in \Cref{example2} has  fundamental group  that is  not semisimple.  
In conclusion, to give  the fundamental group  characterization  of varieties with strong hyperbolicity, we propose the following conjecture: 
\begin{conjecture}\label{conj:funhyp}
	Let $X$ be a quasi-projective normal variety. If there exists a quotient $\varrho:\pi_1(X)\twoheadrightarrow G$ such that $G$ is {semisimple} and $\varrho$ is {big} in the sense defined below, then $X$ is {strongly of log general type} and {pseudo Picard hyperbolic}.
\end{conjecture}

We now define the precise requirement for the quotient map.

\begin{dfn}\label{dfn:big}
	A quotient map $\varrho:\pi_1(X)\twoheadrightarrow G$ is said to be {big} if, for every closed irreducible positive-dimensional subvariety $Z \subset X$ passing through a very general point, the image
	\[
	\varrho\left(\mathrm{Im} \left[\pi_1(Z^{\mathrm{norm}}) \longrightarrow \pi_1(X)\right]\right)
	\]
	is an infinite subgroup of $G$.
\end{dfn}

In \cite{CDY22,DY23b} together with Cadorel and Yamanoi, we proved \cref{conj:funhyp} for linear quotients. 

\begin{thm}\label{main:hyper}
	Let $X$ be a complex quasi-projective normal variety and let $G$ be a semisimple algebraic group defined over an algebraically closed field $K$. If there exists a big and Zariski dense representation $\varrho:\pi_1(X)\to G(K)$, then:
	\begin{thmlist}
		\item \label{item:hyper1} when ${\rm char}\, K = 0$, for any $\sigma\in {\rm Aut}(\mathbb{C}/\mathbb{Q})$, the Galois conjugate $X^\sigma := X\times_\sigma \mathbb{C}$ is pseudo Picard hyperbolic and strongly of log general type;
		\item \label{item:hyper2}when ${\rm char}\, K > 0$, $X$ is pseudo Picard hyperbolic and strongly of log general type.
	\end{thmlist}
\end{thm}

In \cite{CDY22}, we proved \cref{main:hyper} in the case ${\rm char}\, K = 0$, and in \cite{DY23b}, we treated the case of positive characteristic.

\subsubsection{On GGL conjecture}
We discuss the generalized Green-Griffiths-Lang conjecture in this subsection.   
\cref{conj:GGL} is known to hold for curves. However, when $\dim X \geq 2$, there has been little progress. We summarize below four classes of varieties for which \cref{conj:GGL} has been established: 
\begin{itemize}
	\item Complex projective surfaces with big cotangent bundles, proved by McQuillan \cite{McQ}, and later extended to quasi-projective surfaces with big logarithmic cotangent bundles by El Goul \cite{ElG}.
	\item Subvarieties of abelian varieties, by the classical theorem of Bloch, Ochiai and Kawamata, and more generally subvarieties of semi-abelian varieties, by Noguchi \cite{Nog81}.
	\item Projective varieties of maximal Albanese dimension, by Kawamata \cite{Kaw81} and Yamanoi \cite{Yam15}.
	\item General hypersurfaces in projective space $\bP^n$ ($n \geq 3$) of sufficiently high degree, proved in \cite{DMR} (based on the strategy of Siu \cite{Siu02}), with degree bounds subsequently improved in \cite{Dar16,MT,BK24,Cad24}; as well as complements of general hypersurfaces of high degree in $\bP^n$ ($n \geq 2$), proved in \cite{Dar16,BD19}, to mention only a few.
\end{itemize}
In \cite{CDY25}, Cadorel, Yamanoi, and the author first established \cref{conj:GGL} for quasi-projective varieties of maximal quasi-Albanese dimension. Our proof relies primarily on Nevanlinna theory, thereby generalizing the result of \cite{Yam15} to the non-compact setting. Building on the main result in \cite{CDY25}, we further established a \emph{non-abelian version} in \cite{CDY22,DY23b}.

%Recall that a quasi-projective manifold $X$ is said to have maximal quasi-Albanese dimension if the quasi-Albanese map $\alpha:X\to A(X)$ is generically finite onto its image $\overline{\alpha(X)}$. By the universal property of the quasi-Albanese map, this is equivalent to the existence of a morphism $a:X\to A$ to a semi-abelian variety $A$ such that $\dim X=\dim a(X)$. 

\begin{thm}[\cite{CDY22,DY23b}]\label{thm:GGL}
	Let $X$ be a quasi-projective variety. If there exists a semisimple and big representation $\pi_1(X)\to \GL_N(\mathbb{C})$, or a big representation $\pi_1(X)\to \GL_N(K)$ for some field $K$ of positive characteristic, then \cref{conj:GGL} holds for $X$.
\end{thm}
In \cite{CDY22}, we proved \cref{thm:GGL} in the case $\operatorname{char} K = 0$, and in \cite{DY23b}, we addressed the case of positive characteristic.  The case where $X$ is projective and $\operatorname{char} K = 0$ was also discussed in \cite{Bru22} using different methods.

\section{Non-Abelian Hodge Theory and the Shafarevich Conjecture}\label{sec:sha}
Let us mention that the techniques of non-abelian Hodge theory developed by Simpson~\cite{Sim92} in the archimedean setting and by Gromov--Schoen~\cite{GS92} in the non-archimedean setting were first recognized by Katzarkov~\cite{Kat97} as being applicable to the study of \cref{conj:Sha}. 
Together with Ramachandran, he proved \cref{conj:Sha} for projective surfaces with reductive fundamental group~\cite{KR98} (i.e., those whose fundamental group admits a faithful linear representation). 
These techniques were subsequently further developed by Eyssidieux~\cite{Eys04}, and they proved to be highly effective in establishing the conjecture in the linear case~\cite{EKPR12}.

On the other hand, Zuo~\cite{Zuo96} and Yamanoi~\cite{Yam10} discovered that similar methods can also be applied to questions of hyperbolicity for algebraic varieties. 
More recently, Mese, Wang, and the author~\cite{DW24,DW24b,DMW24} realized that the techniques in studying the Shafarevich conjecture in \cite{Eys04,EKPR12,DY23,DY23b} can be used to investigate the topology of algebraic varieties, as discussed in \cref{sec:topology}.

In this section, we review some recent progress on the Shafarevich conjecture and outline several of the essential techniques involved in its proof, which are also relevant to the other conjectures discussed in \cref{sec:topology,sec:hyper}. 
A particularly  comprehensive survey is provided by Eyssidieux~\cite{Eys11}, which we strongly recommend to readers interested in this conjecture.

\subsection{Conjecture and some history}
Let us recall an equivalent version of \cref{conj:Sha}. 
\begin{conjecture}[\cite{Sha77}]
	The universal cover $\widetilde{X}$ of a smooth projective variety $X$ is holomorphically convex.
\end{conjecture}
We recall the definition of holomorphic convexity. 
\begin{dfn}
	Let $Y$ be a complex space. For any compact subset $K\subset Y$, the \emph{holomorphic hull} of $K$ is defined as
	\[
	\widehat{K}=\widehat{K}_{\mathcal{O}(Y)}=\left\{z \in Y \mid |f(z)| \leqslant \sup_{w \in K}|f(w)|, \, \forall f \in \mathcal{O}(Y)\right\}.
	\]
	We say that $Y$ is \emph{holomorphically convex} if $\widehat{K}$ is compact for every compact subset $K \subset Y$.
	
	Furthermore, $Y$ is called \emph{Stein} if it is holomorphically convex and \emph{holomorphically separable} (i.e., for any distinct points $x, y \in Y$, there exists $f\in \mathcal{O}(Y)$ such that $f(x)\neq f(y)$).
\end{dfn} 

We have the following criterion for the holomorphic convexity of complex spaces.

\begin{thm}[Cartan-Remmert]\label{thm:Cartan}
	A complex space $Y$ is \emph{holomorphically convex} if and only if there exists a proper holomorphic fibration $Y\to S$ over a Stein space.
\end{thm}

If the Shafarevich conjecture holds for $X$, one can show that there exists a proper holomorphic fibration $\operatorname{sh}_X:X\to \operatorname{Sh}(X)$ with the property that for any subvariety $Z\subset X$, the image $\operatorname{Im}[\pi_1(Z)\to \pi_1(X)]$ is finite if and only if $\operatorname{sh}_X(Z)$ is a point. 
Such a holomorphic map, if it exists, is called the \emph{Shafarevich morphism} of $X$. 
More generally, we have
\begin{dfn}[Shafarevich morphism]\label{def:Shafarevich morphism}
	Let $X$ be a quasi-projective normal variety. Let $\varrho:\pi_1(X)\to \GL_{N}(K)$ be a representation where $K$ is any field. A dominant  holomorphic map ${\rm sh}_\varrho:X\to {\rm Sh}_\varrho(X)$ with  connected general fibers is called the \emph{Shafarevich morphism associated with $\varrho$} if, for any closed subvariety $Z \subset X$, the image $\varrho(\operatorname{Im}[\pi_1(Z)\to \pi_1(X)])$ is finite if and only if $\operatorname{sh}_\varrho(Z)$ is a point.
\end{dfn}

We provide a brief historical overview of the developments surrounding the Shafarevich conjecture. We note that the literature is vast, and the following list is not exhaustive.

\begin{itemize} 
	\item Gurjar \cite{Gur85} and Napier \cite{Nap90}  initiated this investigation. 
	\item In 1993, Koll\'ar \cite{Kol93} and Campana \cite{Cam94} independently constructed a birational model of the Shafarevich morphism, known as the \emph{Shafarevich map}, for any variety $X$. 
	
	\item In 1998, Katzarkov and Ramachandran \cite{KR98} established the Shafarevich conjecture for surfaces with reductive fundamental group, i.e.\ those admitting an almost faithful semisimple representation $\pi_1(X) \to \operatorname{GL}_N(\mathbb{C})$. In the same work, they proved the holomorphic convexity of suitable intermediate Galois coverings. The latter result was recently  extended by Yuan Liu~\cite{Liu23} to the compact Kähler setting.
	
	\item In 2004, Eyssidieux \cite{Eys04} proved the reductive Shafarevich conjecture. 
	A central and particularly deep aspect of his work is the discovery of 
	sufficient conditions ensuring the holomorphic convexity of certain Galois coverings,
	associated with Simpson's absolutely constructible subsets \cite{Sim93b}. 
	This result has had a substantial impact on complex geometry: beyond resolving the 
	reductive case, the techniques developed by Eyssidieux have inspired numerous 
	applications by the present author and his collaborators in recent years, as will be discussed 
	later in this paper.
	
	\item Building on the foundational deformation theory of Deligne and Goldman-Millson \cite{GM88}, Eyssidieux and Simpson \cite{ES11} constructed a canonical variation of mixed Hodge structure associated with the formal local ring $\mathcal{O}_{\hat{\rho}}$ of the representation variety $R(\pi_{1}(X), \GL_{N})$ at a point $\rho$ corresponding to a $\mathbb{C}$-VHS. This construction provides a crucial framework for the study of the linear Shafarevich conjecture.
	
	\item In 2012, building upon the results in \cite{Eys04, ES11}, Eyssidieux, Katzarkov, Pantev, and Ramachandran \cite{EKPR12}  completely proved the Linear Shafarevich conjecture (cf. \cref{thm:EKPR}), representing a significant development in the study of this conjecture.

	\item In 2015, based on \cite{EKPR12}, Campana, Claudon and Eyssidieux  \cite{CCE15} proved the Linear Shafarevich conjecture for compact K\"ahler manifolds. 
	
	\item In 2023, Yamanoi, Katzarkov, and the author \cite{DY23} extended the reductive Shafarevich conjecture 
	to the case of \emph{singular} projective varieties. 
	Our work also introduced new perspectives in the proof and constructed Shafarevich morphisms 
	for quasi-projective varieties with reductive $\pi_1$,  which were obtained independently 
	by Brunebarbe \cite{Bru23}.
	
	\item In 2023, Green, Griffiths, and Katzarkov established the holomorphic convexity of universal coverings of quasi-projective varieties whose partial Albanese map is proper. This result was subsequently reproved by Aguilar and Campana in \cite{AC25} using different methods.
	
	\item In 2025, Yamanoi and the author~\cite{DY23b} constructed the Shafarevich morphism for quasi-projective normal varieties with linear fundamental groups in positive characteristic. We also proved the Shafarevich conjecture for projective normal surfaces in this setting.
	
	\item More recently, Bakker, Brunebarbe, and Tsimerman addressed the Shafarevich conjecture for \emph{quasi-projective varieties} with linear fundamental groups in their extensive work~\cite{BBT24}.
	
	\item While this survey was awaiting publication, the author proved the linear Shafarevich conjecture for \emph{normal compact K\"ahler varieties} \cite{Den26}, thereby extending the results of \cite{EKPR12,CCE15} to the singular and K\"ahler setting.
\end{itemize}
\subsection{Non-Abelian Hodge theory: archimedean setting}
Non-abelian Hodge theory explores  the geometry of local systems on complex algebraic varieties. The subject was initiated by the work of Siu \cite{Siu80} and Sampson \cite{Sam86} on harmonic maps from Kähler manifolds to non-positively curved target spaces. A major breakthrough came with the work of Corlette \cite{Cor88} and Donaldson \cite{Don87}, who proved the existence of equivariant harmonic maps associated with complex semisimple local systems over compact Kähler manifolds. Building on this, Simpson developed the full framework of non-abelian Hodge theory via Higgs bundles—a notion that had also been studied earlier by Hitchin \cite{Hit87} in the case of Riemann surfaces. He discovered the amazing connection between the work of Donaldson \cite{Don87b} and Uhlenbeck-Yau \cite{UY86} and the work by Griffiths \cite{Gri70}, Deligne, Schmid  et. al. on the variation of Hodge structures \cite{Sim88,Sim92}. 

The extension of the theory to the quasi-projective setting was initiated by Simpson 
for curves \cite{Sim90}, and later completed by Mochizuki 
\cite{Moc06,Moc07,Moc07b} in a monumental series of works spanning more than a 
thousand pages. As a consequence, the theory in the archimedean setting (concerning complex local systems) 
is now fully established.

In this subsection, we briefly recall a small portion of Simpson's non-abelian Hodge theory 
and its subsequent development by Mochizuki in the quasi-projective setting, providing the 
minimal technical framework needed for the results discussed in this paper. 
\subsubsection{The Simpson correspondence} 
Note that when $X$ is compact K\"ahler or quasi-projective, the above harmonic map is pluriharmonic by the Siu--Sampson formula, which leads to the notion of Higgs bundles introduced by Hitchin and Simpson.
\begin{dfn}[Higgs bundle]\label{Higgs}
	A \emph{Higgs bundle} on $X$  is a pair $(E,\theta)$ where $E$ is a holomorphic vector bundle, and $\theta:E\to E\otimes \Omega^1_X$ is a holomorphic one form with value in $\End(E)$, called the \emph{Higgs field},  satisfying $\theta\wedge\theta=0$.
	%	\begin{eqnarray}\label{higgs triple}
		%	(\db +\theta)^2=0.
		%	\end{eqnarray} 
\end{dfn}
Let  $(E,\theta)$ be a Higgs bundle  over a complex manifold $X$.  
Suppose that $h$ is a smooth hermitian metric of $E$.  Denote by $\nabla_h$  the Chern connection of $(E,h)$, and by $\theta^\dagger_h$  the adjoint of $\theta$ with respect to $h$.  We write $\theta^\dagger$ for $\theta^\dagger_h$ for short if no confusion arises.   The metric $h$ is  \emph{harmonic} if the connection  $\nabla_h+\theta+\theta^\dagger$
is flat.
\begin{dfn}[Harmonic bundle] A harmonic bundle on  $X$ is
	a Higgs bundle $(E,\theta)$ endowed with a  harmonic metric $h$.
\end{dfn}
The notion of harmonic comes from the harmonic maps that are interpreted as follows. 

Let $(M,g)$ be a closed Riemannian manifold, and assume that there exists a representation $\varrho:\pi_1(M)\to \GL_N(\bC)$. This determines a flat vector bundle $(V,D)$ over $M$, where $V$ is a smooth vector bundle endowed with a flat connection $D$ whose monodromy representation is $\varrho$. Then for  smooth hermitian metric $h$ on $V$,  corresponds to a $\varrho$-equivariant smooth map 
$$
u_h:\widetilde{M}\to \GL_N(\bC)/U_N.
$$
Here note that $\GL_N(\bC)$ acts on the Riemannian symmetric space $S:=\GL_N(\bC)/U_N$ (which has non-positive sectional curvature). The main theorem by Corlette, proves that there exists a smooth metric $h$ for $V$ such that $u_h$ is energy minimizing, i.e. the energy 
$$
E(u_h)=\int_{M}|du_h|^2d\vol_g
$$
is the critical point. Here we remark that, since $u_h$ is $\varrho$-equivariant, the energy density $|du_h|^2$ on $\widetilde{M}$ is $\pi_1(M)$-invariant and hence descends to a smooth function on $M$, which we denote by the same symbol $|du_h|^2$.  Such an energy minimizing map  $
u_h:\widetilde{M}\to \GL_N(\bC)/U_N 
$ is called $\varrho$-equivariant \emph{harmonic  map},  and it is unique up 
to the following obvious ambiguity: let $G$ denote the Zariski closure of 
$\varrho(\pi_1(M))$ in $\GL_N(\bC)$, and let $Z_{\GL_N(\bC)}(G)$ denote 
the centralizer of $G$ in $\GL_N(\bC)$. Then for any other 
$\varrho$-equivariant harmonic map $u'$, there exists some 
$g\in Z_{\GL_N(\bC)}(G)$ such that 
$$
u'(\tilde{x})=g\cdot u_h(\tilde{x}) \quad \text{for all } \tilde{x}\in \widetilde{M},
$$
where $g$ acts on $S$ by left multiplication. 
\begin{rem}
	Let us verify that $g\cdot u_h$ is indeed another $\varrho$-equivariant 
	harmonic map. Since $g\in Z_{\GL_N(\bC)}(G)$ commutes with every element 
	of $\varrho(\pi_1(M))$, for any $\gamma\in \pi_1(M)$ and 
	$\tilde{x}\in \widetilde{M}$ we have
	$$
	g\cdot u_h(\gamma\cdot \tilde{x})
	=g\cdot \bigl(\varrho(\gamma)\, u_h(\tilde{x})\bigr)
	=\varrho(\gamma)\bigl(g\cdot u_h(\tilde{x})\bigr),
	$$
	so $g\cdot u_h$ is $\varrho$-equivariant. Moreover, left multiplication 
	by $g$ is an isometry of $S$, so $g\cdot u_h$ remains harmonic with the 
	same energy as $u_h$. Corlette's theorem asserts that, conversely, every 
	$\varrho$-equivariant harmonic map arises in this way. In particular, if $G$ is semisimple, then any $\varrho$-equivariant harmonic map is unique, as $Z_{\GL_N(\bC)}(G)$ is trivial. 
\end{rem} The precise equation is given by
\begin{align}
	d_\nabla^*du_h=0.
\end{align}
We denote by  $T_S$   the complexified tangent bundle of $S$. Then we have  $du_h\in \Gamma(\widetilde{M}, \Omega_{\widetilde{M}}^1\otimes u_h^*T_S)$, and 
$$
d_\nabla:\Omega_{\widetilde{M}}^k\otimes u_h^*T_S \to \Omega_{\widetilde{M}}^{k+1}\otimes u_h^*T_S 
$$
is the exterior covariant derivative induced by the connection 
$$
\nabla:u_h^*T_S\to \Omega_{\widetilde{M}}^1\otimes u_h^*T_S
$$
induced by canonical metric of $T_S$.

When $(M,g)$ happens to be a K\"ahler manifold $(X,\omega)$, the Siu-Sampson formula shows that $u_h$ is moreover pluriharmonic, i.e.
\begin{align}\label{eq:pluriharmonic}
	\db_\nabla d' u_h=0,\quad \d_\nabla d''u_h=0.
\end{align} 
here $du_h=d'u_h+d''u_h$ decomposes according to the $(1,0)$ and $(0,1)$-parts of $\Omega_{\widetilde{X}}^1=\Omega_{\widetilde{X}}^{1,0}\oplus \Omega_{\widetilde{X}}^{0,1}$, and   $d_\nabla=\d_\nabla+\db_\nabla$ also decomposes accordingly. 

Let us see how the Higgs bundle structure is defined. First note that the hermitian metric $h$ gives rise to a decomposition  for the flat connection $D$ of $V$ by
$$
D=\nabla_h+\Phi_h,
$$
where $\nabla_h$ is unitary, and $\Phi_h$ is self-adjoint.  Therefore, by the flatness of $D$, we have
\begin{align}\nonumber
	\nabla_h^2+[\Phi_h,\Phi_h]=0\\
	\nabla_h\Phi_h=0\label{eq:flat}
\end{align} 
Note that this holds for any choice of $h$!   

On the other hand, one can show that
\begin{align}\label{eq:coin}
	\Phi_h=du_h, \quad \nabla=\nabla_h
\end{align}  
The Siu-Sampson formula shows that
$$
[\Phi_h^{1,0},\Phi_h^{1,0}]=0. 
$$
Hence we have
$$
\nabla_h^2=(\nabla^2)^{(1,1)}. 
$$
Let $E:=u_h^*T_S$, that is a complex vector bundle (not yet a holomorphic one!). Since 
$$
\nabla:E\to E\otimes \Omega_{\widetilde{M}}^1
$$ 
is unitary, it follows that $\nabla^{0,1}$ gives rise to a complex structure for $E$.  
By \eqref{eq:pluriharmonic}, \eqref{eq:coin} and \eqref{eq:flat}, we have
\begin{align}
	\db_\nabla \Phi_h^{1,0}=0. 
\end{align} 
Hence $(E,\db_\nabla,\theta,h):=(E,\db_\nabla,\Phi_h^{1,0},h)$ is a harmonic bundle on $\widetilde{X}$. Since it is $\varrho$-equivariant, one can show that it descends to a harmonic bundle on $X$. In summary, one has the following:
\begin{thm}[Corlette]\label{thm:Corlette}
	Let $(X,\omega)$ be a compact K\"ahler manifold.  A representation $\varrho:\pi_1(X)\to \GL_{N}(\bC)$ is reductive, if and only if  there exists a harmonic bundle $(E,\theta,h)$ such that the monodromy representation $\nabla_h+\theta+\theta_h^*$  is $\varrho$. 
\end{thm}
I refer the reader to the survey articles \cite{Lou20,DM23,Mau15} for the results mentioned above.
\begin{rem}A representation \( \varrho \colon \pi_1(X) \to \GL_N(K) \) is called \emph{reductive} if its Zariski closure is a reductive group. Note that \( \varrho \) is reductive if it is semisimple. When \( \operatorname{char} K = 0 \), the converse holds; however, it fails when \( \operatorname{char} K > 0 \) (see \cite[Example~2.7]{DY23b}).
\end{rem}

\medspace

On the other hand, Simpson introduced the stability of Higgs bundles, and he proved the following result.
\begin{thm}[Simpson]\label{thm:Simpson}
	Let $(X,\omega)$ be a compact K\"ahler manifold and let $(E,\theta)$ be  a $\mu_\omega$-polystable Higgs bundle on $X$. If $c_1(E)\wedge \{\omega\}^{n-1}=c_2(E)\wedge\{\omega\}^{n-2}=0$, then there exists a harmonic metric $h$ for $(E,\theta)$. 
\end{thm}
In other words, he established the one-to-one correspondence between semisimple local systems on $X$, and the polystable Higgs bundles with vanishing Chern classes.  This is called the \emph{Simpson correspondence} in the literature.  
\begin{rem}
	\begin{itemize}
		\item  	While harmonic metric depends on the choice of the metric on the Riemannian manifold, for K\"ahler manifolds, the pluriharmonic metric is independent of the choice of the K\"ahler metric on $X$.
		\item The term \emph{non-abelian Hodge theory} originates from the following distinction: 
		classical Hodge theory concerns the cohomology of the constant sheaf \(\mathbb{C}\) on \(X\), 
		whereas non-abelian Hodge theory is related to the cohomology of nontrivial local systems on \(X\), 
		whose monodromy groups are generally non-abelian.
	\end{itemize} 
\end{rem}
A straightforward application of Simpson's theorem is the following.
\begin{thm}[Simpson]\label{thm:reductive}
	Let $f \colon X \to Y$ be a morphism of projective varieties, where $Y$ is smooth and $X$ is normal. For any reductive representation $\varrho \colon \pi_1(Y) \to \mathrm{GL}_{N}(\mathbb{C})$, the pullback $f^*\varrho \colon \pi_1(X) \to \mathrm{GL}_{N}(\mathbb{C})$ is also reductive.
\end{thm}

\begin{proof}
	Let $\mu \colon Z \to X$ be a desingularization. Since $X$ is normal, the induced homomorphism on fundamental groups
	\[
	\mu_*: \pi_1(Z) \longrightarrow \pi_1(X)
	\]
	is surjective. Consequently, the image of the representation $f^*\varrho$ coincides with the image of $(f \circ \mu)^*\varrho$. Thus, $f^*\varrho$ is reductive if and only if $(f \circ \mu)^*\varrho$ is reductive. Therefore, after replacing $X$ with a desingularization, we may assume without loss of generality that $X$ is smooth.
	
	We now apply the Simpson correspondence. Let $(E,\theta,h)$ be the harmonic bundle on $Y$ corresponding to the reductive representation $\varrho$. The pullback $(f^*E, f^*\theta, f^*h)$ is also a harmonic bundle on $X$ corresponding to the representation $f^*\varrho$. By \cref{thm:Corlette}, the existence of a harmonic metric implies that the associated representation $f^*\varrho$ is reductive. The theorem is proved.
\end{proof}

\subsubsection{Complex Variation of Hodge structure}
In \cite{Sim88}, Simpson discovered that $\bC$-VHS in algebraic geometry is indeed a special case of harmonic bundles. Let us recall the definition. 
\begin{dfn}
	Let $Y$ be a complex manifold. A $\bC$-VHS consists of $(V=\oplus_{p+q=r}V^{p,q},\nabla,Q) $ where $(V,\nabla)$ is a flat bundle and $Q$ is a non-definite hermitian form for $V$  such that
	\begin{itemize}
		\item $V=\oplus_{p+q=r}V^{p,q}$ is a direct sum of smooth vector bundles that is orthogonal with respect to $Q$,
		\item $Q$ is  $\nabla$-parallel such that $h:=(-1)^pQ|_{V^{p,q}}$ is positively definite
		\item we have   the Griffiths transversality:
		\[
		\begin{array}{ccc}
			\nabla: V^{p,q} \to A^{0,1}(V^{p+1,q-1})\!\!\!\! & \oplus A^1(V^{p,q}) &\!\!\!\! \oplus A^{1,0}(V^{p-1,q+1}) \\
			\bar{\theta} & D  & \theta
		\end{array}
		\]
	\end{itemize} 
\end{dfn} 
Since $D''^2=0$,  \(E^{p,q}:=(V^{p,q},D'')\) has a structure of  holomorphic vector bundle such that any  smooth local  section $s$ of $V^{p,q}$ is a holomorphic section of $E^{p,q}$ if and only if $D''s=0$.  
The pair \((E=\bigoplus_{p,q} E^{p,q},\theta)\) is then called a \emph{Hodge bundle}, where the $\cO_Y$-morphism 
\[
\theta: E^{p,q}\longrightarrow E^{p-1,q+1}\otimes \Omega_Y^1
\]
satisfying  the integrability condition  \(\theta\wedge\theta=0\), is  the 
 \emph{Higgs field}. 
The Hodge decomposition is orthogonal with respect to the Hermitian metric \(h\), which is 
referred to as the \emph{Hodge metric}. One can check that $(\oplus_{p+q=r},E^{p,q},\theta,h)$ is a harmonic bundle.

Note that a $\mathbb{Z}$-VHS originally arises geometrically as follows: let $f:X\to Y$ be a smooth projective family over a complex manifold $Y$ endowed with a relative ample line bundle $L\to X$. Then the primitive part of the higher direct image of the constant sheaf $\mathbb{C}$ on $X$, denoted by $R^kf_*(\mathbb{C})_{\mathrm{prim}}$, underlies a $\mathbb{C}$-VHS defined by$$\left( R^kf_*(\mathbb{C})_{\mathrm{prim}}\otimes_{\mathbb{C}} \mathcal{O}_Y \stackrel{C^\infty}{\simeq} \bigoplus_{p+q=k}R^qf_*(\Omega_{X/Y}^p), \nabla, Q \right)$$where $\nabla$ is the Gauss-Manin connection for $R^kf_*(\mathbb{C})$, and the sesquilinear form $Q$ is determined by the Hodge-Riemann bilinear relations:$$Q(\alpha, \beta) = i^{p-q}\int_{X_t}\alpha\wedge \overline{\beta}\wedge c_1(L_t)^{\dim X_t-k}$$ for any $\alpha, \beta\in H^{k}(X_t)_{\mathrm{prim}}$. Here $X_t$ denotes the fiber $f^{-1}(t)$ for each $t \in Y$.  Since the $\mathbb{C}$-VHS has an integral structure given by $R^kf_*(\mathbb{C})=R^kf_*(\mathbb{Z})\otimes_{\mathbb{Z}} \mathbb{C}$, this constitutes a $\mathbb{Z}$-VHS. We will not verify the axioms of $\bC$-VHS here, but refer the reader to the excellent textbook by Voisin \cite{Voi02} for more details.
\subsubsection{$\bC^*$-action on Higgs bundles}
For any Higgs bundle $(E,\theta)$ on a compact K\"ahler manifold $Y$ and any $t\in\bC^*$, we define the action by
$$
t.(E,\theta):=(E,t\theta).
$$
We shall use \cref{thm:Corlette,thm:Simpson} to define such action on representation of $\pi_1(Y)$. Let $\varrho:\pi_1(Y)\to\GL_N(\bC)$ be a semisimple representation. Then there exists a harmonic bundle $(E,\theta,h)$ corresponding to $\varrho$. By Simpson's theorem, $(E,\theta)$ is $\mu_\omega$-polystable with vanishing Chern classes. Note that $(E,t\theta)$ is also $\mu_\omega$-polystable. Therefore, we use Simpson's theorem again to get a harmonic metric $h_t$ for $(E,t\theta)$. Let $\varrho_t:\pi_1(Y)\to\GL_N(\bC)$ be the monodromy representation given by the flat connection $\nabla_t+\theta+\bar{\theta}_{h_t}$.  In general,  $\varrho_t$ is not conjugate to $\varrho$. 
However, one can show that
\begin{proposition}[\cite{Sim92}]
	The semisimple representation $\varrho$ corresponds to a $\bC$-VHS if and only if it is $\bC^*$-invariant, i.e. $\varrho_t$ is conjugate to $\varrho$ for any $t\in \bC^*$. 
\end{proposition}

Note that most of Simpson's results, except those discussed in the next subsection,  have been extended to the quasi-projective setting in a series of works by 
Mochizuki \cite{Moc06,Moc07,Moc07b}. In particular, all of the results stated above 
remain valid in the quasi-projective case.
\subsubsection{Betti and Dolbeault Moduli spaces, rigid representation and ubiquity}  
In this subsection, we discuss the Betti moduli spaces  and rigid representations. Let $X$ be a  smooth projective variety. 
We have an affine scheme $\Hom(\pi_1(X), \GL_N)$ of finite type defined over $\bZ$ that represents the functor
$$
S\mapsto \Hom(\pi_1(X), \GL_N(S))
$$
for any ring $S$.  
The \emph{Betti moduli space} $M_{\rm B}(X,\GL_N):=\Hom(\pi_1(X), \GL_N)\slash\!\!\slash \GL_{N}$ is the GIT quotient of $\Hom(\pi_1(X), \GL_N)$ by $\GL_N$, where $\GL_{N}$ acts by conjugation.  Note that it might be reducible while it is called variety. It is indeed the  \emph{$\GL_{N}$-character variety}  for the finitely generated group $\pi_1(X)$.      Thus,  such quotient $\Hom(\pi_1(X), \GL_N)\to M_{\rm B}(X,\GL_N)$ is surjective.  For any representation $\varrho:\pi_1(X)\to \GL_N(\bC)$, it is a $\bC$-point in $\Hom(\pi_1(X), \GL_N)(\bC)$. We write $[\varrho]$ to denote  its image in $M_{\rm B}(X,\GL_{N})$.  We list a  crucial property  of  $\GL_{N}$-character   varieties that will be used in this paper. For further background and more comprehensive accounts, we refer the reader to the survey of Sikora \cite{Sik12} and to the monograph of Lubotzky–Magid \cite{LM85}.

Let $K$ be any algebraically closed field of characteristic zero.  
By \cite[Theorem 30]{Sik12}, the orbit of  any representation $\varrho$ in $ \Hom(\pi_1(X), \GL_N)(K)$ is closed if and only if  $\varrho$ is   semisimple. Two semisimple representations $\varrho,\varrho'$ are conjugate under $\GL_N(K)$ if and only if $[\varrho]=[\varrho']$.

In~\cite{Sim94,Sim94b}, Simpson constructed the moduli spaces of Higgs bundles with vanishing Chern classes on a smooth projective variety $X$. Fix a positive integer $N$ and a polarization $L$ on $X$. Simpson constructed a quasi-projective scheme $M_{\mathrm{Dol}}(X,N)$ parametrizing   $\mu_L$-semistable Higgs bundles $(E,\theta)$ of rank $N$ on $X$ with vanishing Chern classes. The points of the moduli space correspond to equivalence classes defined as follows:

\begin{itemize}
	\item Let $(E,\theta)$ be a $\mu_L$-semistable Higgs bundle with vanishing Chern classes. Let $\mathrm{Gr}(E, \theta)$ denote the graded object associated to a Jordan--H\"older filtration of $(E,\theta)$. Then $\mathrm{Gr}(E,\theta)$ is a polystable Higgs bundle with vanishing Chern classes. It is \emph{locally free} by~\cite[Theorem~2]{Sim92}. This result also holds in the compact K\"ahler setting (see the note by the author~\cite{DenSim}); however, it fails if the condition on the vanishing of Chern classes is dropped.
	\item Two $\mu_L$-semistable Higgs bundles $(E_1,\theta_1)$ and $(E_2,\theta_2)$ with vanishing Chern classes are called \emph{Jordan-equivalent} if there exists an isomorphism between $\mathrm{Gr}(E_1,\theta_1)$ and $\mathrm{Gr}(E_2, \theta_2)$.
\end{itemize}
\begin{thm}[\cite{Sim94}]\label{thm:analytic iso}
	Let $X$ be a smooth projective variety.  The points of $M_{\rm Dol}(X,N)(\bC)$ correspond to Jordan
	equivalence classes of $\mu_L$-semistable  Higgs bundles of rank $N$ with vanishing Chern classes.   Moreover, there is a real analytic isomorphism between $M_{\rm Dol}(X,N)(\bC)$ and $M_{\rm B}(X,\GL_N)(\bC)$. Moreover, the $\bC^*$-action on $M_{\rm Dol}(X,N)$ given by $(E,\theta)\mapsto (E,t\theta)$ for any $t\in\bC^*$ is algebraic. 
\end{thm} 
For any Higgs bundle $(E,\theta)$, consider the characteristic polynomial  of $\theta$ defined by   $\det (\lambda I-\theta)=\lambda^N+\sum_{i=1}^{N} a_i\lambda^{N-i}$, where $a_i\in H^0(X,{\rm Sym}^i\Omega_X)$. 
\begin{dfn}
	The map $M_{\rm Dol}(X,N)\to \oplus_{k\geq 1} H^0(X,\Omega_X^k)$  given by 
	$$
	(E,\theta)\mapsto (a_1,\ldots,a_N) 
	$$
	is called the \emph{Hitchin morphism}. 
\end{dfn} 
\begin{proposition}[\cite{Sim94b}]\label{prop:proper}
	The Hitchin morphism   $M_{\rm Dol}(X,N)\to \oplus_{k\geq 1} H^0(X,\Omega_X^k)$  is algebraic, and \emph{proper}. 
\end{proposition}
Based on his theorem, he estbalished the following striking results.
\begin{lem}\label{lem:limit}
	The limit $\lim_{t\to 0}t.(E,\theta)$ exists, and it corresponds to a $\bC$-VHS. 
\end{lem} 
\begin{proof}
	Note that the inverse image of the Hitchin morphism of the origin, corresponds to Higgs bundles with nilpotent $\theta$ (including $\bC$-VHS). The image of $t.(E,\theta)$ under the Hitchin morphism converges to the origin if $t\to 0$. Since the Hitchin morphism is proper, the limit exists.  It is easy to see that the limit is $\bC^*$-invariant, and thus it corresponds to a $\bC$-VHS. 
\end{proof}
We then apply the real analytic isomorphism $M_{\rm Dol}(X,N)(\bC)\to M_{\rm B}(X,\GL_N)(\bC)$ to obtain  the $\bC^*$-action on $M_{\rm B}(X,N)$, that is a \emph{continuous} action, but in general not algebraic!  
Therefore, we have the following theorem.
\begin{proposition}[\cite{Sim91}]\label{prop:ubiquity}
	In any connected component of $M_{\rm B}(X,N)$,  there exists a representation that underlies a $\bC$-VHS. Moreover, any representation can be deformed to a $\bC$-VHS. 
\end{proposition}
This result is called \emph{Simpson's ubiquity theorem} in the literature.

\begin{dfn}
	A semisimple representation   
	$\varrho:\pi_1(X)\to \GL_N(\bC)$  is   \emph{rigid} if  the irreducible component  of $M_{\rm B}(X,\GL_{N})(\bC)$ containing $[\varrho]$  is zero dimensional,  otherwise it is called \emph{non-rigid}.   \qed
\end{dfn}
For a  rigid semisimple  representation $\varrho$, by  above arguments any continuous deformation $\varrho':\pi_1(X)\to \GL_N(\bC)$ of $\varrho$ which is semisimple is conjugate to $\varrho$  under $\GL_N(\bC)$. By \cref{prop:ubiquity}, we have the following result.
\begin{cor}\label{cor:rigid}
	Any rigid representation corresponds to a $\bC$-VHS. \qed
\end{cor} 
But the drawback is that such $\bC$-VHS in general does not have discrete monodromy. That motivates us to consider the non-archimedean representations in \cref{sec:na}.  

We note that most of the results in this subsection, except for 
\cref{prop:ubiquity,cor:rigid}, remain open in the quasi-projective setting. 
In particular, extending \cref{lem:limit} to the quasi-projective case would be a 
very interesting problem. 
\subsection{Non-abelian Hodge theory: non-archimedean cases}\label{sec:na}
While Simpson--Mochizuki's theory deals exclusively with complex local systems, 
\emph{all} of the problems discussed so far naturally lead to the study of 
$p$-adic local systems, as well as $\mathbb{F}_p((t))$-local systems, on 
complex algebraic varieties, as we will see shortly.  More precisely, given a Zariski dense representation
\[
\varrho \colon \pi_1(X) \longrightarrow G(K),
\]
where $G$ is a reductive algebraic group defined over a non-archimedean local field $K$, Gromov--Schoen \cite{GS92} proved that, when $X$ is compact Kähler, there exists a $\varrho$-equivariant harmonic map from the universal cover $\widetilde{X}$ to the \emph{Bruhat--Tits building} $\Delta(G)_K$, a CAT$(0)$ space on which $G(K)$ acts by isometries. We refer the reader to \cite{BH99} for the definition of CAT$(0)$ spaces.

More recently, Brotbek, Daskalopoulos, Mese, and the   author \cite{BDDM,DM24} extended the 
Gromov--Schoen theorem to the quasi-projective setting. 
In this subsection, we present this theorem and explain how one can extract algebraic and 
analytic structures from it.
\subsubsection{Bruhat-Tits buildings}  
Note that \(S:=\GL_N(\mathbb{C})/U_N\) is the symmetric space associated with \(\GL_N(\mathbb{C})\).
It satisfies the following properties:
\begin{itemize}
	\item it is \emph{non-positively curved} (indeed, it has non-positive sectional curvature);
	\item \(\GL_N(\mathbb{C})\) acts transitively on \(S\);
	\item the stabilizer of any point under this action is a compact subgroup.
\end{itemize}
If \(K\) is a non-archimedean local field, such a na\"ive construction is no longer available.
Indeed, a maximal compact subgroup of \(\GL_N(K)\) is conjugate to \(\GL_N(\mathcal{O}_K)\),
which is an open subgroup. Consequently, the quotient
\(\GL_N(K)/\GL_N(\mathcal{O}_K)\), viewed as an analogue of the symmetric space,  is totally discrete and lacks interesting geometric structure.
A natural and rich analogue of the symmetric space in the non-archimedean setting is provided
by the Bruhat--Tits building.

Let \(K\) be a non-archimedean local field of characteristic zero, and let \(G\) be a reductive 
algebraic group defined over \(K\). 
There exists a Euclidean building associated with \(G\), called the (enlarged) 
Bruhat--Tits building and denoted by \(\Delta(G)_K\) (or simply \(\Delta(G)\)). 
It is a non-positively curved space on which \(G(K)\) acts strongly transitively by isometries, 
and the stabilizer of any point under this action is a precompact subgroup of \(G(K)\). 
We refer the reader to \cite{KP23,Guy23} for the definition and basic properties of 
Bruhat--Tits buildings.

Associated with \(\Delta(G)_K\) is a pair \((V,W)\), where \(V\) is a real Euclidean space 
endowed with its Euclidean metric, and \(W\) is an affine Weyl group acting on \(V\). 
More precisely, \(W\) is a semidirect product
\[
W = T \rtimes W^{v},
\]
where \(W^{v}\) is the \emph{vectorial Weyl group}, a finite group generated by reflections of 
\(V\), and \(T\) is a group of translations of \(V\).

For any apartment $A$ in $ \Delta(G)$,  there exists an isomorphism $i_A:A\to V$, which is  called a chart. For two charts $i_{A_1}:A_1\to V$ and $i_{A_2}:A_2\to V$, if $A_1\cap A_2\neq\varnothing$,   it satisfies the following properties:
\begin{enumerate}[label=(\alph*)]
	\item $Y:=i_{A_2}(i_{A_1}^{-1}(V))$ is convex.
	\item \label{item:weyl}There is an element $w\in W$ such that $w\circ i_{A_1}|_{A_1\cap A_2}=i_{A_2}|_{A_1\cap A_2}$.
\end{enumerate} 
%The charts allow us to map facets into $\Delta(G)$ and their images are
%also called facets. The  axioms guarantee that these
%notions are chart independent. 
\begin{figure}[h] 
	\centering	
	\begin{tikzpicture}[
		grow cyclic,
		level distance=1cm,
		level/.style={
			level distance/.expanded=\ifnum#1>1 \tikzleveldistance/1.5\else\tikzleveldistance\fi,
			nodes/.expanded={\ifodd#1 fill\else fill=none\fi}
		},
		level 1/.style={sibling angle=120},
		level 2/.style={sibling angle=90},
		level 3/.style={sibling angle=90},
		level 4/.style={sibling angle=45},
		nodes={circle,draw,inner sep=+0pt, minimum size=2pt},
		]
		\path[rotate=30]
		node {}
		child foreach \cntI in {1,...,3} {
			node {}
			child foreach \cntII in {1,...,2} { 
				node {}
				child foreach \cntIII in {1,...,2} {
					node {}
					child foreach \cntIV in {1,...,2} {
						node {}
						child foreach \cntV in {1,...,2} {}
					}
				}
			}
		};
	\end{tikzpicture}   
	\caption{Bruhat-Tits building for ${\rm SL}_2(\bQ_p)$ with $p=2$.}
\end{figure}
\subsubsection{Harmonic mapping to Euclidean building} 
We first present our result on the extension of Gromov-Schoen's theorem. 
\begin{thm}[\cite{BDDM,DM24}] \label{thm:BDDMex} 
	Let $X$ be a smooth quasi-projective variety and let $G$ be a reductive group defined over a non-archidemean local field $K$. Let $\Delta(G)$ be  the  Bruhat-Tits building of $G$. Denote by  $\pi_X:\widetilde{X}\to X$ the universal covering map.   If $\varrho:\pi_1(X)\to G(K)$ is a Zariski dense representation, then the following statements hold:
	\begin{thmlist}
		\item  \label{item:existence}  There exists a $\varrho$-equivariant pluriharmonic map $u:\widetilde{X}\to \Delta(G)$ with  logarithmic energy growth. 
		\item \label{item:harmonic}$ {u}$ is harmonic with respect to any K\"ahler metric on $\widetilde{X}$.  	
		\item  \label{item:pullback}Let $f: Y \to X$ be a morphism from a smooth quasi-projective variety $Y$. Denote by $\tilde{f}: \widetilde{Y} \to \widetilde{X}$   the lift of $f$ between the universal covers of $Y$ and $X$.  Then the $f^*\varrho$-equivariant map $ {u }\circ \tilde{f}: \widetilde{Y} \to \Delta(G)$ is   pluriharmonic and has logarithmic energy growth.  
	\end{thmlist} 
\end{thm} 
For the definitions appearing in the theorem, 
we refer the reader to \cite{BDDM} or to \cite[Section~2]{DM24} for more details. 
The above theorem was established in \cite[Theorem~A]{BDDM} by 
Brotbek, Daskalopoulos, Mese, and the  author in the case where $G$ is semisimple, building upon previous work of Daskalopoulos and Mese \cite{DM23J,DM24A,DM23M}.
Subsequently, building on \cite{BDDM}, 
 Mese and the author extended the result to the general reductive case 
in \cite[Theorem~A]{DM24}. 
As we will see later, this extension provides a crucial building block both for studying the hyperbolicity of 
(non-compact) algebraic varieties \cite{CDY22} and for addressing the  linear Shafarevich conjecture in the 
quasi-projective case \cite{DY23,DY23b,Bru23,BBT24}.

By \cref{thm:BDDMex}, there exists a $\varrho$-equivariant \emph{harmonic mapping with logarithmic energy growth} $u:\widetilde{X}\to \Delta(G)$.  
Moreover, such a map $u$ is pluriharmonic.  

We denote by $\mathcal{R}(u)$ the regular set of the harmonic map $u$. Explicitly, this is the set of points $x\in \widetilde{X}$ for which there exists an open neighborhood $\Omega_{x}$ of $x$ such that $u(\Omega_{x})\subset A$ for some apartment $A$. 

Since $G(K)$ acts transitively on the apartments of $\Delta(G)$ and $u$ is $\varrho$-equivariant, the set $\mathcal{R}(u)$ is the pullback of an open subset of $X$. By abuse of notation, we denote this subset of $X$ also by $\mathcal{R}(u)$.

In \cite{GS92}, Gromov and Schoen prove a   deep \emph{regularity theorem}: the Hausdorff codimension of the complement $X\setminus \mathcal{R}(u)$ is at least two. In \cite{DM24},   Mese  and the author proved that $X\setminus \mathcal{R}(u)$ is contained in a proper Zariski closed subset of $X$. 

It has long been a major open problem whether such regularity theorems for harmonic maps into Euclidean buildings hold for non-locally compact Euclidean buildings (e.g., the Bruhat-Tits building of $\operatorname{GL}_{N}(K)$, where $K=\mathbb{Q}((t))$). This was only recently resolved in a celebrated work by Breiner, Dees, and Mese \cite{BDM24}, who established the Gromov-Schoen regularity theorem for harmonic maps into {non-locally compact} Euclidean buildings.  
This is  one of the most significant breakthroughs in this subject over the past two decades. 

Let me also mention that, in \cite{DM21G}, Daskalopoulos and Mese proved a  similar regularity theorem for harmonic maps into \emph{hyperbolic buildings}, which is another highly interesting development in this subject.  

\medspace

We now fix orthogonal coordinates $(x_1,\ldots,x_N)$ for $V$. Define smooth real functions on $\Omega_x$ by setting
\begin{align}\label{eq:ui}
	u_{A,i} := x_i \circ i_A \circ u,
\end{align}
where $i_A:A\to V$ is the chart defined in the previous subsection. The pluriharmonicity of $u$ implies that $\bar{\partial}\partial u_{A,i}=0$ for each $i$; hence, $\partial u_{A,i}$ is a holomorphic 1-form. However, the multiset of holomorphic 1-forms $\{\partial u_{A,1},\ldots,\partial u_{A,N}\}$ depends on the choice of the chart and the apartment containing $u(\Omega_x)$. We can resolve this ambiguity by considering the multiset of holomorphic 1-forms:
$$
\{\partial x_i \circ w \circ i_{A}\circ u\}_{i=1,\ldots,N; \, w\in W^v}.
$$
One can verify (see, e.g., \cite{BDDM,CDY22}) that this multiset of holomorphic 1-forms is invariant under the action of $W^v$ and is independent of the choice of the chart. They glue together to define a multivalued 1-form on $\mathcal{R}(u)$, denoted by $\eta_u$. Since $u$ is $\varrho$-equivariant, one can verify that $\eta_u$ descends to a \emph{splitting} multivalued 1-form on $\mathcal{R}(u)$.

Furthermore, using the Lipschitz property of the harmonic map, one can show that $\eta_u$ extends to a multivalued 1-form on all of $X$. Since $u$ has logarithmic energy growth, in \cite{BDDM} we can prove that  $\eta_u$ extends to a multivalued section for the logarithmic cotangent bundle $\Omega_{\overline{X}}(\log D)$, where $\overline{X}$ is a smooth projective compactification for $X$ and $D:=\overline{X}\backslash X$ is a simple normal crossing divisor.     We refer the reader to \cite[\S 3.1]{CDY22} for the formal definition of \emph{multivalued sections} of a holomorphic vector bundle. Less formally, we have
\begin{dfn}[Multivalued  section]\label{def:multivalued}
	Let $X$ be a complex manifold and let $E$ be a  holomorphic  vector bundle on $X$.  A \emph{multivalued section} $\eta$ on $X$ is   a formal sum $ \sum_{i=1}^{m}n_i\Gamma_i$, where each $n_i\in \mathbb{Z}_{>0}$ and each $\Gamma_i$  is an irreducible closed subvariety of $E$ such that  
	the natural projection $\Gamma_i\to X$ is a finite surjective morphism.  
\end{dfn}
In \cite{DM24},  Mese and the author established the uniqueness of harmonic maps in a suitable setting. In particular, we showed that $\eta_u$ is independent of the choice of $u$. Therefore, we denote this form by $\eta_\varrho$.   It also extends to a multivalued section of the logarithmic bundle $\Omega_{\overline{X}}(\log D)\to \overline{X}$. 
\begin{dfn}[Multivalued 1-form]\label{def:multivaluedform}
	The multivalued 1-form $\eta_\varrho$ on $X$ described above is called the \emph{multivalued 1-form associated with $\varrho$}.
\end{dfn}
On the other hand, we can construct an analytic object associated with $\varrho$ as follows.   For the open neighborhood  $\Omega_x$ of $x$ as above, we define  a smooth, real, semi-positive $(1,1)$-form by
\begin{align}\label{eq:form}
	\sqrt{-1}\sum_{i=1}^{N}\partial u_{A,i} \wedge \bar{\partial} u_{A,i}.
\end{align}
One can verify that this closed real semi-positive $(1,1)$-form is independent of the choice of $A$ and the orthogonal coordinates $(x_1,\ldots,x_N)$ for $V$ (see \cite[\S 3]{DW24b}). Moreover, it is invariant under the $\pi_1(X)$-action. Consequently, it descends to a smooth, real, closed, semi-positive $(1,1)$-form on $\mathcal{R}(u)$. 
The Lipschitz property of $u$ and elliptic regularity imply  that this form extends to a positive closed $(1,1)$-current $T_{\varrho}$ on $X$ with continuous potential. 
\begin{dfn}[Canonical current]\label{def:canonical}
	The closed positive $(1,1)$-current $T_\varrho$ on $X$ defined above is called the \emph{canonical current} of $\varrho$.
\end{dfn}

\begin{rem}
	Although the $\varrho$-equivariant harmonic map $u$ may not be unique, the uniqueness result in \cite{DM24} implies that both $\eta_\varrho$ and $T_\varrho$ are independent of the choice of $u$.
\end{rem}

\subsubsection{Spectral covering and the Katzarkov-Eyssidieux reduction map}

Directly manipulating such multivalued 1-forms in algebraic geometry is often impractical. A standard approach is to construct a finite (generally ramified) covering $\pi \colon \xsp \to X$ such that $\pi^*\eta$ becomes a set of global logarithmic 1-forms on $\xsp$, with an explicitly describable ramification locus. We recall the main result from \cite[Proposition 3.1]{CDY22}.

\begin{proposition}[Spectral covering]\label{prop:spectral}
	Let $\eta_\rho$ be the multivalued 1-form described in \cref{def:multivaluedform}. There exists a finite Galois cover $\pi \colon \overline{\xsp} \to \overline{X}$ with Galois group $H$, together with a multiset of holomorphic sections 
	\[
	\{\omega_1, \dots, \omega_m\} \subset H^0(\overline{\xsp}, \pi^*\Omega_{\overline{X}}(\log D))
	\]
	that is $H$-invariant. Furthermore, the ramification locus of $\pi$ is contained in 
	\[
	\bigcup_{\omega_i \neq \omega_j} \{ \omega_i - \omega_j = 0 \}. 
	\]
\end{proposition}
Such Galois covering $\pi$ is called the \emph{spectral covering of $X$} associated with $\varrho$. 

Building on \cref{prop:spectral} and the properties of harmonic mappings into Euclidean buildings, we derive the following theorem.

\begin{thm}[{\cite[Theorem E]{CDY22}}]\label{thm:KE}
	Let $X$ be a  quasi-projective normal variety and let $\rho \colon \pi_1(X) \to \GL_n(K)$ be a   representation, where $K$ is a non-archimedean local field. Then there exists a  dominant morphism $s_\rho \colon X \to S_\rho$ to a normal projective variety with connected general fibers, such that for any irreducible closed subvariety $Z \subset X$, the following properties are equivalent: 
	\begin{thmlist}
		\item \label{item:KZ1} The image $\rho(\operatorname{Im}[\pi_1(Z^{\mathrm{norm}}) \to \pi_1(X)])$ is bounded;
		\item \label{item:KZ3} The image $\rho(\operatorname{Im}[\pi_1(Z) \to \pi_1(X)])$ is bounded;
		\item \label{item:KZ2} The image $s_\rho(Z)$ is a point;
		\item \label{item:KZ4} The restriction of the canonical current $T_\rho|_Z$ defined in \cref{def:canonical} is trivial.
		\item  \label{item:KZ5} The restriction of the multivalued 1-form $\eta_\rho|_Z$ defined in \cref{def:multivaluedform} is trivial.
	\end{thmlist}   
\end{thm} 
We call the map $s_\rho$ the \emph{Katzarkov-Eyssidieux reduction map} for $\rho$. When $X$ is compact K\"ahler, such a map $s_\rho$ was constructed by Katzarkov and Eyssidieux in \cite{Eys04}. It plays an important role in the construction of the Shafarevich morphism in the next subsection. 
\subsection{Construction of the Reductive Shafarevich Morphism}\label{sec:sha mor}

Let $X$ be a smooth quasi-projective variety. We will construct the Shafarevich morphism ${\rm sh}_X:X\to {\rm Sh}(X)$ in this subsection when $\pi_1(X)$ is reductive, i.e.,   there exists a semisimple representation $\varrho:\pi_1(X)\to \GL_N(\bC)$ such that $\ker\varrho$ is finite (we say $\varrho$ is \emph{almost faithful}).  We begin with the following definition.

\begin{dfn}[Infinite monodromy at infinity]\label{def:monodromy}
	Let $X$ be a quasi-projective normal variety and let $\overline{X}$ be a projective compactification of $X$. We say a subset $M\subset M_{\rm B}(X,N)(\bC)$ \emph{has infinite monodromy at infinity} if for any holomorphic map $\gamma:\bD\to \overline{X}$ with $\gamma^{-1}(\overline{X}\setminus X)=\{0\}$, there exists a reductive representation $\varrho:\pi_1(X)\to \GL_N(\bC)$ such that $[\varrho]\in M$ and the restriction $\gamma^*\varrho:\pi_1(\bD^*)\to \GL_N(\bC)$ has infinite image.
\end{dfn}
The Shafarevich morphism associated with a local system underlying a 
 $\bC$-VHS with discrete monodromy can be constructed via period maps.
\begin{lem}\label{lem:period}
	If   $\varrho:\pi_1(X)\to \GL_{N}(\bC)$ underlies a $\bC$-VHS $\cL$ with discrete monodromy, then after replacing $X$ by a partial compactification, the $\bC$-VHS extends, and the associated period map $p:X\to \sD/\Gamma$ is proper, where $\sD$ is the period domain and $\Gamma$ is the monodromy group of $\cL$. In this case, the Shafarevich morphism ${\rm sh}_\varrho:X\to {\rm Sh}_\varrho(X)$ of $\varrho$ is the Stein factorization of $p$.
\end{lem}

\begin{proof}
	For the first claim, we refer the reader to \cite{Gri70}. Note that in this case, the monodromy representation of the extended $\cL$ has infinite monodromy at infinity. The proof relies on the hyperbolicity of the period domain and a standard application of the Ahlfors-Schwarz lemma.   Note that the period domain of $\cL$ has the form $\sD=G/V$, where $G$ is a real reductive group and $V\subset G$ is a compact subgroup (see \cite{CMP17}).
	
	Let $Z\subset X$ be a closed subvariety and let $\mu:Y\to Z$ be a desingularization.
	
	If $p(Z)$ is a point $P$, then for the lift of $p$ to the Galois  covering $\tilde{p}:\widetilde{X}_\varrho\to \sD$ associated with $\ker\varrho$, the image $p(\widetilde{X})$ is also a point $P\in \sD$. The image $\varrho({\rm Im}[\pi_1(Z)\to \pi_1(X)])$ is contained in the stabilizer of $P$, which is a conjugate $V'$ of $V$, and hence is compact. On the other hand, since $\cL$ has discrete monodromy,
	$$
	\varrho({\rm Im}[\pi_1(Z)\to \pi_1(X)]) \subset \Gamma\cap V'.
	$$
	This is a discrete subgroup of a compact group, and is therefore finite.  
	
	Conversely, if $\varrho({\rm Im}[\pi_1(Z)\to \pi_1(X)])$ is finite, then the $\bC$-VHS $\mu^*\cL$ on $Y$ has finite monodromy. After replacing $Y$ by a finite \'etale cover, the monodromy becomes trivial. By the uniqueness of the $\bC$-VHS structure underlying a local system (see \cite{Sch73}), $\mu^*\cL$ must be the trivial $\bC$-VHS. In particular, its period mapping $\tilde{p}_Y:\widetilde{Y}\to \sD$ is constant. Since we have the commutative diagram
	\[
	\begin{tikzcd}
		\widetilde{Y}\arrow[d] \arrow[r,"\tilde{p}_Y"] &\sD\\
		\widetilde{X} \arrow[ur,"\tilde{p}"'] &
	\end{tikzcd}
	\]
	it follows that $p(Z)$ is a point. The lemma is proved.
\end{proof}
%\begin{rem}
%It is expected that the Shafarevich morphism of a quasi-projective variety, when it exists, is algebraic. In \cite{Som78}, Sommese proved that ${\rm Sh}_\varrho(X)$ appearing in \cref{lem:period} is properly bimeromorphic to a quasi-projective variety, using H\"ormander’s $L^2$–techniques. More recently, when $\varrho$ underlies a $\bZ$-VHS,  Bakker, Brunebarbe, and Tsimerman \cite[Theorem 1.1]{BBT23} established that ${\rm Sh}_\varrho(X)$ is indeed algebraic, thereby confirming a longstanding conjecture of Griffiths.
%\end{rem}

If $\varrho$ is not a   $\bC$-VHS with discrete monodromy, the approach in \cite{DY23}, following the strategy initiated in \cite{Eys04}, attempts to find a fibration such that the restriction of $\varrho$ to each fiber is a direct sum of a $\bZ$-VHS.

We now introduce an important morphism $s_{\rm fac}:X\to S_{\rm Fac}(X)$. We first start with the following lemma.  
\begin{lem}[{\cite[Lemma 1.28]{DY23}}]\label{lem:simultaneous}
	Let $V$ be a smooth quasi-projective  variety and let $(f_{\lambda}:V\to S_{\lambda})_{\lambda\in\Lambda}$ be a family of   morphisms into  quasi-projective varieties $S_{\lambda}$.
	Then there exist a  quasi-projective normal
	variety $S_{\infty}$ and a morphism $f_{\infty}:V\to S_{\infty}$ such that 
	\begin{enumerate}
		\item
		$f_{\infty}$ is  dominant and has connected  general fibers,
		\item for every subvariety $Z\subset V$, $f_{\infty}(Z)$ is a point if and only if $f_{\lambda}(Z)$ is a point for every $\lambda\in \Lambda$, 
		and
		\item
		there exist $\lambda_1,\ldots,\lambda_n\in\Lambda$ such that $f_{\infty}:V\to S_{\infty}$ is the  Stein factorization of $$(f_1,\ldots,f_n):V\to S_{\lambda_1}\times\cdots S_{\lambda_n}.$$  
	\end{enumerate} 
Assume additionally that $V$ and each $S_\lambda$ are all projective. Then 	for any $\lambda\in \Lambda$, there exists a morphism $\nu_\lambda:S_\infty\to S_\lambda$ such that $\nu_\lambda\circ f_\infty=f_\lambda$.

	The above $f_\infty:V\to S_\infty$ will called the \emph{simultaneous Stein factorization} for $(f_{\lambda}:V\to S_{\lambda})_{\lambda\in\Lambda}$.   
\end{lem}
\begin{proof}
	All assertions except the last are proven in \cite[Lemma 1.28]{DY23}. Let us now prove the final one. Fix $\lambda \in \Lambda$. By Item (2), $f_\lambda$ contracts each fiber of $f_\infty$ to a point. Thus, by the universal property of the Stein factorization, there exists a unique morphism $\nu_\lambda \colon S_\infty \to S_\lambda$ such that $\nu_\lambda \circ f_\infty = f_\lambda$.
\end{proof}

\medspace

\begin{dfn}\label{def:fac}
	Fix some $N\in \bN$.	 Consider the set $\Upsilon:=\{\tau:\pi_1(X)\to \GL_N(K)\}$, where $\tau$ ranges over all semisimple representations with $K$ being any non-archimedean local field of characteristic zero. Consider the set of Katzarkov-Eyssidieux reductions $\{s_\tau\}_{\tau\in\Upsilon}$. We take a \emph{simultaneous Stein factorization} of all these $s_\tau$, denoted by $s_{\rm fac}:X\to S_{\rm Fac}(X)$. 
\end{dfn}
Note that  $s_{\rm fac}:X\to S_{\rm Fac}(X)$   is a dominant morphism with connected general fibers such that  for any subvariety $Z$ of $X$, the following properties are equivalent:
\begin{itemize}
	\item   $s_\tau(Z)$ is a point for each $\tau\in \Upsilon$;
	\item   $s_{\rm fac}(Z)$ is also a point.
\end{itemize}
Moreover, if $X$ is projective, then 
for any $\tau\in \Upsilon$, we have a factorization
\begin{equation}\label{eq:facto}
\begin{tikzcd}
	X\arrow[r,"s_{\rm fac}"] \arrow[dr,"s_\tau"'] & S_{\rm Fac}(X)\arrow[d]\\
	& S_\tau
\end{tikzcd}
\end{equation}
\begin{rem}
	Note that \eqref{eq:facto} was not used in the construction of the Shafarevich morphism.  
\end{rem}

This factorization possesses the following crucial property.

\begin{proposition}[{\cite[Proposition 3.10]{DY23}}]\label{lem:conjugate3}
	Let $X$ be a smooth quasi-projective variety, and let $f:Y\to X$ be a morphism from a smooth quasi-projective variety $Y$ such that $s_{\rm fac}\circ f(Y)$ is a point. Let $\{\tau_i:\pi_1(X)\to {\rm GL}_N(\bC)\}_{i=1,2}$ be reductive representations such that $[\tau_1]$ and $[\tau_2]$ are in the same geometric connected component of $M(\bC)$. Then $f^*\tau_1 $ is conjugate to $f^*\tau_2 $. In other words, $j(M_{\rm B}(X,N))$ is zero-dimensional, where $j:M_{\rm B}(X,N)\to M_{\rm B}(Y,N)$ is the natural morphism of character varieties induced by $f$.
\end{proposition}

The main idea of the proof proceeds as follows. Consider the morphism between character varieties $M_{\rm B}(X,N)\to M_{\rm B}(Y,N)$. If the image of $M_{\rm B}(X,N)$ under this morphism, denoted by $M$, is not zero-dimensional, then its closure is a Zariski closed subset of the affine scheme, and is hence non-compact. However, one can show that for any fixed non-archimedean local field $K$ of characteristic zero, the set of bounded representations $\tau:\pi_1(Y)\to \GL_N(K)$ is a compact subset of $M_{\rm B}(Y,N)(K)$. Therefore, there exists some $\tau:\pi_1(X)\to \GL_{N}(K)$ such that $f^*\tau$ is unbounded. By the construction of $s_{\rm fac}$, this is impossible. Hence, we conclude that $M$ is zero-dimensional.

In \cite{DY23}, we proved the following crucial result.
\begin{proposition}[{\cite[Proposition 3.13]{DY23}}]\label{prop:cons}
	There exists a $\bC$-VHS $\cL$ on $X$ such that for any morphism $f:Y\to X$ from a smooth quasi-projective variety $Y$ to $X$, if $s_{\rm fac}\circ f(Y)$ is a point, then for any semisimple representation $\varrho:\pi_1(X)\to \GL_{N}(\bC)$,  $f^*\varrho$ is a direct factor of the $\bC$-VHS $f^*\cL$, which has discrete monodromy.
\end{proposition}

Note that $\cL$ itself might not have discrete monodromy.

\begin{proof}
	Fix a  geometric connected component $M$ of  $M_{\rm B}(X,N)$. Since $M_{\rm B}(X,N)$ and $\Hom(\pi_1(X), \GL_N)$ are defined over $\bZ$, we can choose a semisimple representation $\tau:\pi_1(X)\to \GL_{N}(\overline{\bQ})$ such that  $[\tau]\in M(\overline{\bQ})$. As $\pi_1(X)$ is finitely generated, there exists a number field $k$ such that $\tau:\pi_1(X)\to \GL_N(k)$.
	
	Let ${\rm Ar}(k)$ be the set of all archimedean places of $k$, with $w_1$ being the identity map. By the ubiquity theorem (\cref{prop:ubiquity}),  and its extension to the quasi-projective setting by Mochizuki \cite{Moc06},  for any $w\in {\rm Ar}(k)$, there exists a reductive representation $\tau_{w}^{\text{\tiny{VHS}}}:\pi_1(X)\to {\rm GL}_N(\bC)$ that underlies a $\bC$-VHS, such that $[\tau_{w}^{\text{\tiny{VHS}}}]$ is in the same connected component of $[\tau_{w}]$. Let $\cL$ be the $\bC$-VHS defined as the direct sum $\oplus_{w\in {\rm Ar}(k)}\tau_w^\vhs$.
	
	By \cref{thm:reductive}, both \(f^{*}\tau_{w}\) and \(f^{*}\tau_{w}^{\vhs}\) are reductive. Then, by \cref{lem:conjugate3}, \(f^{*}\tau_{w}^{\vhs}\) is conjugate to \(f^{*}\tau_{w}\). Hence, each \(f^{*}\tau_{w}\) underlies a \(\mathbb{C}\)-VHS. %In particular, \(f^{*}\varrho\) is conjugate to \(f^{*}\tau_{w_{1}}^{\vhs}\), which underlies a \(\mathbb{C}\)-VHS, and therefore appears as a direct factor of \(f^{*}\mathcal{L}\).
	
	Let $v$ be any non-archimedean place of $k$ and let $k_v$ be the non-archimedean completion of $k$ with respect to $v$. Let $\tau_{v}:\pi_1(X)\to {\rm GL}_N(k_v)$ be the representation induced by $\tau$. By the definition of $s_{\rm fac}$, it follows that $f^*\tau_{v}(\pi_1(Z))$ is bounded. Therefore, we have a factorization
	$$
	f^*\tau:\pi_1(Z)\to {\rm GL}_N(\cO_k).
	$$
	Note that the embedding ${\rm GL}_N(\cO_k)\to \prod_{w\in {\rm Ar}(k)}{\rm GL}_N(\bC)$ has a discrete image by \cite[Proposition 6.1.3]{Zim}. It follows that for the product representation
	$$
	\prod_{w\in {\rm Ar}(k)}f^*\tau_{w}:\pi_1(Z)\to \prod_{w\in {\rm Ar}(k)}{\rm GL}_N(\bC),
	$$
	the image is discrete.
	
	Since \(f^{*}\tau_{w}\) is conjugate to \(f^{*}\tau_{w}^{\text{\tiny VHS}}\) for each 
	\(w\in \mathrm{Ar}(k)\), it follows that \(f^{*}\mathcal{L}\) has discrete monodromy. 
	If \([\varrho]\) and \([\tau]\) both belong to \(M\), then \(f^{*}\varrho\) is conjugate to 
	\(f^{*}\tau\), and hence \(f^{*}\varrho\) appears as a direct factor of \(f^{*}\mathcal{L}\).
	Note that the local system \(\mathcal{L}\) is chosen so as to be associated with \(M\).
	
	Repeating this construction for each geometrically connected component of 
	\(M_{\mathrm{B}}(X,N)\), and taking the direct sum of the resulting 
	\(\mathbb{C}\)-VHS, we obtain the desired \(\mathbb{C}\)-VHS.  
	The proposition follows.
\end{proof}

We now turn to the construction of the Shafarevich morphism of $X$ when $\pi_1(X)$ is reductive. In \cite[Proposition 3.19]{DY23}, it is shown that after replacing $X$ by an \'etale cover, there exists a partial compactification such that $\varrho$ extends to a representation with infinite monodromy at infinity. Let $\cL$ be the $\bC$-VHS from \cref{prop:cons} and let $p:\widetilde{X}\to\sD$ be the associated period map. Consider the holomorphic map
\begin{align*}
	\Phi:\widetilde{X} &\to \sD \times S_{\rm Fac}(X)\\
	x & \mapsto (p(x), s_{\rm fac}\circ \pi_X(x)),
\end{align*}
where $\pi_X:\widetilde{X}\to X$ is the universal covering map. Using \cref{prop:cons} together with \cref{lem:period}, and the assumption that there exists a faithful and semisimple representation \( \varrho \colon \pi_1(X) \to \GL_N(\bC) \), one deduces that each connected component of a fiber of \( \Phi \) is compact (cf.\ \cite[Lemma~3.30]{DY23}).Then, by \cref{lem:Stein}  below, $\Phi$ factors through a proper holomorphic fibration
$$
{\rm sh}_{\widetilde{X}}: \widetilde{X}\to {\rm Sh}(\widetilde{X})
$$
onto a normal complex space ${\rm Sh}(\widetilde{X})$. One can prove that ${\rm Sh}(\widetilde{X})$ does not contain any compact closed subvariety of positive dimension. Hence, the $\pi_1(X)$-action on $\widetilde{X}$ maps fibers of ${\rm sh}_{\widetilde{X}}$ to fibers, thereby inducing an action on ${\rm Sh}(\widetilde{X})$. This action is properly discontinuous, and ${\rm sh}_{\widetilde{X}}$ is equivariant with respect to it. We take the quotient of ${\rm sh}_{\widetilde{X}}$ to obtain a proper holomorphic fibration ${\rm sh}_X: X\to {\rm Sh}(X)$ such that  we have the commutative diagram
\[
\begin{tikzcd}
	\widetilde{X}\arrow[r] \arrow[d,"{\rm sh}_{\widetilde{X}}"] & X\arrow[d,"{\rm sh}_X"] \\
	{\rm Sh}(\widetilde{X}) \arrow[r]& {\rm Sh}(X).
\end{tikzcd}
\]
One can show that ${\rm sh}_X: X\to {\rm Sh}(X)$  is the Shafarevich morphism of $X$. 

We need the following generalized Stein factorization proved by Henri Cartan in \cite[Theorem 3]{Car60}.

\begin{thm}\label{lem:Stein}
	Let $X$ and $S$ be complex spaces and let $f: X \rightarrow S$ be a holomorphic map. Suppose that $X$ is normal and that every connected component $F$ of a fiber of $f$ is compact. Then the set $Y$ of connected components of the fibers of $f$ can be endowed with the structure of a normal complex space such that $f$ factors through the natural map $e: X \rightarrow Y$, which is a proper holomorphic fibration. \qed
\end{thm} 

\begin{rem}The proof presented here is technically different from that of \cite{Eys04} in the compact case, 
	although we follow the essential ideas introduced by Eyssidieux.  
	In \cite{Eys04}, Eyssidieux studies representations 
	\(\pi_1(X)\to \GL_{N}(k((t)))\) over non-archimedean fields, where \(k\) is a number field, 
	arising from closed curves \(C \subset \mathrm{Hom}(\pi_1(X),\GL_{N})\).  
	Since \(k((t))\) is not locally compact, the Gromov–Schoen regularity theorem was not available 
	at that time. %\footnote{Only recently was it shown in a deep paper by Breiner, Dees, and Mese 
	%	\cite{BDM24} that the Gromov–Schoen regularity theorem holds for general Euclidean buildings, 
	%	without assuming local compactness.}  
 	In this setting, Eyssidieux used reduction modulo \(p\) arguments to work with representations 
	\(\pi_1(X)\to \GL_{N}(\mathbb{F}_{q}((t)))\) in order to \emph{factor through non-rigidity}.  
	For integrality issues, he must further consider representations 
	\(\pi_1(X)\to \GL_{N}(K)\), where \(K\) is a finite extension of some \(\mathbb{Q}_p\).  
	This makes the proof substantially more involved.
	
	The advantage of our construction of \(s_{\mathrm{fac}}\) is that it allows us to treat 
	\emph{non-rigidity} and \emph{integrality} simultaneously, as we have seen in \cref{prop:cons}.     Here,  {integrality} roughly means that the representation is conjugate to one taking values in \( \GL_N(\mathcal{O}_K) \), where \( K \) is a number field and \( \mathcal{O}_K \) its ring of integers. Moreover, a further simplification is that we work only with local fields of characteristic zero.
\end{rem}
\subsection{Shafarevich morphism in positive characteristic}
The previous subsection constructed the Shafarevich morphism for reductive representation of the fundamental group of quasi-projective varieties. In this subsection, we outline the construction of the Shafarevich morphism by Yamanoi and the author  \cite[Theorem A]{DY23b} for any representation $\rho \colon \pi_1(X) \to \operatorname{GL}_{N}(K)$, where $K$ is a field of   characteristic $p>0$.
\subsubsection{A lemma on finite groups}
\begin{lem}[{\cite[Lemma 3.1]{DY23b}}]\label{lem:finite group}
	Let $K$ be an algebraically closed field of positive characteristic and let $\Gamma$ be a finitely generated group. Let $\varrho:\Gamma\to \mathrm{GL}_N(K)$ be a representation such that its semisimplification has finite image. Then $\varrho(\Gamma)$ is finite.
\end{lem}
\begin{proof}
	Since the semisimplification $\varrho^{ss}$ of $\varrho$ has finite image,  we can replace $\Gamma$ by a finite index subgroup such that $\varrho^{ss}(\Gamma)$ is trivial.  Therefore, some conjugate $\sigma$ of $\varrho$ has image in the subgroup ${\rm U}_N(K)$ consisting 
	of all upper-triangular matrices in $\GL_{N}(K)$
	with 1's on the main diagonal.  
	
	Note ${\rm U}_N(K)$  admits a central normal series whose successive
	quotients are isomorphic to $\bG_{a,K}$.  We remark that a finitely generated subgroup  of $\bG_{a,K}$ is a finite group, for $K$ has positive characteristic.
	By \cite[Proposition 4.17]{ST00},   any finite index subgroup   of a finitely generated  group  is also finitely generated.   Consequently,    $\sigma(\Gamma)$ admits a   central normal series whose successive
	quotients are   finitely generated subgroups of $\bG_{a,K}$, which are  all finite groups. 
	It follows that $\sigma(\Gamma)$  is finite. The lemma is proved. 
\end{proof}
\subsubsection{Character varieties in positive characteristic} \label{sec:2.2}   
Recall that the variety of $N$-dimensional linear  representations of $\pi_1(X)$ in characteristic zero is represented by  an affine $\bZ$-scheme $R(X ,N)$ of finite type. 
Namely, given a commutative ring $A$, the set of $A$-points of $R(X ,N)$ is:
\begin{align}\label{eqdef:repre}
  R(X ,N)(A)={\rm Hom}(\pi_1(X) , \GL_N(A)).
\end{align} 
Let $p $ be  a prime number.  
Consider $R(X,N)_{\bF_p}:=R(X ,N)  \times_{\spec \bZ}\spec \bF_p$ and note that the general linear group over $\bF_p$, denoted by $\GL(N,\bF_p)$, acts on $R(X ,N)_{\bF_p}$ by conjugation.  
Using Seshadri's extension of geometric invariant theory quotients for schemes of arbitrary field \cite[Theorem 3]{Ses77}, we can take the GIT quotient
of $R(X ,N)_{\bF_p}$ by $\GL(N,\bF_p)$, denoted by $M_{\rm B}(X ,N)_{\bF_p}$. 
Then $M_{\rm B}(X ,N)_{\bF_p}$  is also an affine $\bF_p$-scheme of finite type.  For any algebraically closed field $K$ of characteristic $p$,   the $K$-points  $M_{\rm B}(X ,N)_{\bF_p}(K)$ is identified with the conjugacy classes of semi-simple representations $X \to \GL_N(K)$.     

\subsubsection{Shafarevich morphism in positive characteristic}
We construct the Shafarevich morphism for quasi-projective varieties whose fundamental group admits a linear representation into \( \GL_N(K) \) with \( \operatorname{char} K > 0 \). A more general statement can be found in \cite[Theorem~3.9]{DY23b}.
\begin{thm}[{\cite[Theorem 3.9]{DY23b}}]\label{thm:Shapositive} 
	Let $X$ be a smooth quasi-projective variety. Assume that there exists an almost faithful representation $\rho \colon \pi_1(X) \to \operatorname{GL}_{N}(K)$, where $K$ is a field of   characteristic $p>0$.	The Shafarevich morphism ${\rm sh}_X:X\to {\rm Sh(X)}$ of $X$  is obtained through the simultaneous Stein factorization $s_\infty:X\to S_\infty(X)$  of the reductions   $\{s_{\tau}:X\to S_\tau\}_{\tau\in \Upsilon}$,  where
	$\Upsilon$ consists of all reductive representations with $K$  a non-archimedean local field of characteristic $p$,  and $s_\tau:X\to S_\tau$ is the Katzarkov-Eyssidieux reduction map  defined in \cref{thm:KE}.  
\end{thm} 
\begin{proof}
For simplicity, write \(R = R(X,N)_{\mathbb{F}_p}\) and \(M = M_{\mathrm{B}}(X,N)_{\mathbb{F}_p}\).
Let \(f:Y \to X\) be a morphism from another smooth quasi-projective variety. In the same vein as \cref{lem:conjugate3}, we have the following result.
\begin{claim}\label{claim:zero}
	If \(s_\infty \circ f(Y)\) is a point, then \(h(M)\) is zero-dimensional, where 
	\(h : M \to M_{\mathrm{B}}(Y,N)_{\mathbb{F}_p}\) is the natural morphism induced by \(f\).
\end{claim}
	\begin{proof} 
		Assume by contradiction that $h(M)$ has a positive dimensional component.  Since  $R\to M$ is a surjective $  \mathbb{F}_p$-morphism between affine varieties defined over $\mathbb{F}_p$, we thus can find  an irreducible affine curve $T\subset R$ defined over ${\overline{\bF}_p}$, such that its image in  $M_{\rm B}(Y ,N)_{\bF_p}$  under the composite morphism $$\Phi:R\to M\stackrel{h}{\to} M_{\rm B}(Y,N)_{\bF_p}$$
		is non a point. 	Take $\bar{C}$ as the compactification of the normalization $C$ of $T$, and let $\{P_1,\ldots,P_\ell\}= \bar{C}\setminus C$.  
		There exists $q=p^n$ for some $n\in \bZ_{>0}$ such that $\bar{C}$ is defined over $\bF_q$ and $P_i\in \bar{C}(\bF_q)$ for each $i$. 
		By the universal property of the representation scheme $R$, $C$ gives rise to a representation $\varrho_C:\pi_1(X)\to \GL_N(\bF_q[C])$, where $\bF_q[C]$ is the coordinate ring of $C$. 
		Consider the discrete valuation $v_i:\bF_q(C)\to \bZ$ defined by $P_i$, where $\bF_q(C)$ is the function field of $C$. 
		Let $\widehat{\bF_q(C)}_{v_i}$  be the completion of $F_{q}(C)$ with respect to $v_i$.  Then we have the isomorphism $\big(\widehat{\bF_q(C)}_{v_i},v_i\big)\simeq \big(\bF_q((t)),v\big)$, where $ \big(\bF_q((t)),v\big)$ is the formal Laurent field of $\bF_q$ with the valuation $v$ defined by  $v(\sum_{i=m}^{+\infty}a_it^i)=\min \{i\mid a_i\neq 0\}$.    
		Let $\varrho_i:\pi_1(X)\to \GL_N(\bF_q((t)))$  be the extension of $\varrho_C$ with respect to $\widehat{\bF_q(C)}_{v_i}$.  
		
		By our definition of $s_\infty$, and the assumption that $s_\infty\circ f(Y)$ is a point, and \cref{thm:KE},    $f^*\varrho_i(\pi_1(X))$ is bounded for each $i$.
		Thus after we replace $f^*\varrho_i$ by some conjugate, we have $f^*\varrho_i(\pi_1(X))\subset \GL_{N}(\bF_q[[t]])$, where  the $\bF_q[[t]]$ is the ring of integers of $\bF_q((t))$, i.e.
		$$
		\bF_q[[t]]:=\{\sum_{i=0}^{+\infty}a_it^i\mid a_i\in \bF_q \}.
		$$ 
		For any matrix $A\in \GL_N(K)$,  we denote by $\chi(A)=T^N+\sigma_1(A)T^{N-1}+\cdots+\sigma_N(A)$ its characteristic polynomial.     
		Since  $f^*\varrho_i(\pi_1(X))\subset \GL_{N}(\bF_q[[t]])$ for each $i$, it follows that $\sigma_{j}(f^*\varrho_i(\gamma))\in \bF_q[[t]]$ for each $i$.  
		Therefore, by the definition of $\varrho_i$, $v_i\big(\sigma_j(f^*\varrho_C(\gamma))\big)\geq 0$ for each $i$.  It follows that $\sigma_j(f^*\varrho_C(\gamma))$  extends to a regular function on $\overline{C}$, which is thus constant.  
		This implies that for any  two representations $\eta_1:\pi_1(X)\to \GL_{N}(K_1)$ and $\eta_2:\pi_1(X)\to \GL_{N}(K_2)$  such that ${\rm char}\, K_1={\rm char}\, K_2=p$ and  $\eta_i\in C(K_i)$,  we have  $\chi(f^*\eta_1(\gamma))=\chi(f^*\eta_2(\gamma))$ for each $\gamma\in \pi_1(Y)$.  
		In other words, $\eta_1$ and $\eta_2$ has the same characteristic polynomial.  
		It follows that $[f^*\eta_1]=[f^*\eta_2]$ by the Brauer-Nesbitt theorem. 
		Hence $\Phi(T)$ is a point.    We obtain a contradiction. 
	\end{proof} 
	Since $R$ is defined over $\overline{\bF}_p$, there exists some representation  $\eta:\pi_1(X)\to \GL_{N}({\overline{\bF}_p})$   such that $\eta$ are in the same geometric connected component of $\varrho$.   Since $\pi_1(X)$ is finitely generated, it follows that $\eta(\pi_1(X))$ is finite. Let $f:Y\to X$ be as in \cref{claim:zero}.  By  \cref{claim:zero}, we have 
	$$
	[f^*\eta]=[f^*\varrho]. 
	$$
	Therefore,   the semisimplification $(f^*\eta)^{ss}$  of $f^*\eta$ and the semisimplifcation $(f^*\varrho)^{ss}$ of $f^*\varrho$  are conjugate.  By virtue of \cref{lem:finite group},  we conclude that $f^*\varrho(\pi_1(Y))$ is finite.   As $\varrho$ is almost faithful, it follows that
	$$
	{\rm Im}[\pi_1(Y)\to \pi_1(X)]
	$$
	is finite. 
	
	On the other hand, if $g:Z\to X$ be a morphism such that $s_\infty\circ g(Z)$ is not a point, by \cref{thm:KE}, there exists some $\tau\in \Upsilon$ such that $g^*\tau$ is unbounded, thus $g^*\tau(\pi_1(Z))$ has infinite image.  This proves that $s_\infty$ is the Shafarevich morphism of $X$.  
\end{proof} 
\begin{rem}
	Note that we make the simplifying assumption that \( \rho \colon \pi_1(X) \to \GL_N(K) \), where \( K \) is a field of characteristic \( p > 0 \). This slightly simplifies the proof, although it is not necessary in general.
\end{rem}

\begin{rem}
	As noted in \cref{thm:Shapositive}, the Shafarevich morphism in positive characteristic relies solely on Katzarkov-Eyssidieux reductions. Furthermore, by virtue of \cref{lem:finite group}, the linear case can be handled without additional difficulty. This implies that while the problem in positive characteristic is in some sense simpler, it also exhibits a less rich structure than the characteristic zero case. Consequently, the Shafarevich conjecture for projective varieties whose fundamental groups admit a faithful representation into $\GL_{N}(K)$ with ${\rm char}\, K>0$  remains open, with only the case of surfaces having been proven in \cite{DY23b}.
\end{rem}

It is worth noting that in \cite{DY23,DY23b} we establish the existence of the Shafarevich morphism in a more general setting, without assuming that the fundamental group of  $X$ is linear. More precisely, we prove the following. 
\begin{thm}[\cite{DY23,DY23b}]\label{thm:Shafarevich}
	Let  $X$ be a normal quasi-projective variety, and let
	$\varrho\colon \pi_1(X)\to \GL_N(K)$ be a representation, where $K$ is an arbitrary field. 
	If ${\rm char}\,K=0$, we additionally assume that $\varrho$ is reductive. 
	Then the Shafarevich morphism associated with $\varrho$, denoted by 
	${\rm sh}_\varrho\colon X\to {\rm Sh}_\varrho(X)$, exists. \qed
\end{thm}
We refer the interested reader to \cite{DY23,DY23b} for further details.

\subsubsection{An application to Esnault's conjecture}  
A intriguing question by Esnault asks whether a complex projective variety with infinite fundamental group must admit non-trivial symmetric differentials. This question was first answered affirmatively by Brunebarbe, Klingler, and Totaro \cite{BKT13} in the case where there exists a representation $\varrho:\pi_1(X)\to \GL_N(K)$, over an arbitrary field $K$, such that $\varrho(\pi_1(X))$ is infinite.

More recently, Brotbek, Daskalopoulos, Mese, and the author \cite{BDDM} extended this result to the quasi-projective setting. 
\begin{thm}[{\cite[Theorem B]{BDDM}}]\label{thm:sym}
	Let $X$ be a smooth quasi-projective variety. If there exists a repersentation $\varrho:\pi_1(X)\to\GL_{N}(K)$ for any field $K$ such that $\varrho(\pi_1(X))$ is infinite. Then  for some $k\in \bN$, 
	we have
	$$
	H^0(\overline{X}, {\rm Sym}^k\Omega_{\overline{X}}(\log D))\neq 0,
	$$
	where  $\overline{X}$ is a smooth projective compactification, such that $D=\overline{X}\backslash X$ is a simple normal crossing divisor. 
\end{thm}   
Our approach relies on the theorem on the existence of $\varrho$-equivariant harmonic maps to Euclidean buildings (see \cref{thm:BDDMex}). We now briefly sketch the proof, which differs slightly from the one presented in \cite{BDDM}   (though the key ideas are exactly the same).
\begin{proof}[Proof  of \cref{thm:sym} (sketch)]
	The existence of logarithmic symmetric differentials is preserved under finite \'etale covers.
	Therefore, throughout the proof, we may freely replace $X$ by a finite \'etale cover.
	
	\medskip
	\noindent\textbf{Case~1:} ${\rm char}\, K = 0$.
	
	We may assume that $K=\bC$. Let $\sigma$ be the semisimplification of $\varrho$.
	
	\medskip
	\noindent\textbf{Case~1.1:} $\sigma$ has finite image.
	
	After replacing $X$ by a finite \'etale cover, the representation $\varrho$ becomes unipotent.
	Its image then admits a filtration whose successive quotients are abelian. Consequently,
	there exists an abelian representation of $\pi_1(X)$ with infinite image. This implies that
	$H^1(X,\bC)$ is infinite-dimensional.
	
	Recall that
	\[
	H^1(X,\bC)
	= H^0\bigl(\overline{X}, \Omega_{\overline{X}}(\log D)\bigr)
	\oplus H^{0,1}(X).
	\]
	Hence $H^0\bigl(\overline{X}, \Omega_{\overline{X}}(\log D)\bigr)\neq 0$.
	
	\medskip
	\noindent\textbf{Case~1.2:} $\sigma$ has infinite image.
	
	\medskip
	\noindent\textbf{Case~1.2.1:}  
	For every representation $\tau:\pi_1(X)\to \GL_N(K)$, where $K$ is a non-archimedean local
	field of characteristic zero, the image of $\tau$ is bounded.
	
In this case, the factorization map \(s_{\rm fac}: X \to S_{\rm Fac}(X)\) is constant. 
By \cref{prop:cons}, \(\sigma\) appears as a direct factor of a complex variation of Hodge 
structure \(\mathcal{L}\) with discrete monodromy. 
Since \(\sigma\) has infinite image, it follows that the monodromy of \(\mathcal{L}\) is infinite.
	
	By \cref{lem:period}, the associated period map $p:X\to \sD/\Gamma$ has positive-dimensional
	image. Let $Z$ be a desingularization of the closure of $p(X)$. Using curvature properties
	of period domains in the horizontal direction \cite{Gri70}, Brunebarbe and Cadorel
	\cite{BC20} proved that $Z$ has a big logarithmic cotangent bundle. In particular,
	$Z$ admits many logarithmic symmetric differentials. Pulling them back to $X$ yields
	the desired conclusion.
	
	\medskip
	\noindent\textbf{Case~1.2.2:}  
	There exists a representation $\tau:\pi_1(X)\to \GL_N(K)$, where $K$ is a non-archimedean
	local field of characteristic zero, whose image is unbounded.
	
By \cref{thm:KE}, there exists a nontrivial multivalued logarithmic \(1\)-form \(\eta_{\tau}\).
Taking local products of \(\eta_{\tau}\), one obtains logarithmic symmetric differentials on a 
Zariski open subset of \(\overline{X}\).
One then shows that these extend to logarithmic symmetric differentials on the pair 
\((\overline{X}, D)\).
	
	\medskip
	
	\noindent\textbf{Case~2:} ${\rm char}\, K = p>0$.
	
	By the proof of \cref{thm:Shapositive}, it follows that there exists an unbounded
	representation
	\[
	\tau:\pi_1(X)\to \GL_N\bigl(\bF_q((t))\bigr),
	\]
	where $q=p^\ell$ for some $\ell\in \bN$. Otherwise, the character variety
	$M_{\rm B}(X,N)_{\bF_p}$ would be zero-dimensional, and it would follow from
	\cref{lem:finite group} that $\varrho$ has finite image, contradicting the assumption.
	We then apply the same argument as in  {Case~1.2.2} to conclude that $X$
	admits a non-trivial logarithmic symmetric differential.
\end{proof}

\begin{rem}
	When $X$ is projective and $K=\bC$, the original argument in \cite{BKT13} is a
	beautiful application of non-abelian Hodge theory. We briefly recall their proof.
	 
If the semisimplification \(\varrho^{\mathrm{ss}}\) has finite image, the argument proceeds 
exactly as above. We may therefore assume that \(\varrho\) is semisimple with infinite image.
Suppose, by contradiction, that \(X\) admits no nontrivial symmetric differentials.
Then, by \cref{prop:proper}, the Dolbeault moduli space \(M_{\mathrm{Dol}}(X,N)\) is compact.
By \cref{thm:analytic iso}, there exists an analytic isomorphism between 
\(M_{\mathrm{B}}(X,N)\) and \(M_{\mathrm{Dol}}(X,N)\), and hence \(M_{\mathrm{B}}(X,N)\) is also 
compact. Since \(M_{\mathrm{B}}(X,N)\) is affine, it follows that it is zero-dimensional.
Consequently, every semisimple representation $\pi_1(X)\to \GL_{N}(\bC)$ is rigid, and in 
particular \(\varrho\) is rigid. By \cref{cor:rigid}, \(\varrho\) underlies a 
\(\mathbb{C}\)-variation of Hodge structure. However, such a \(\mathbb{C}\)-VHS may \emph{a 
	priori} have non-discrete monodromy.

Since \(M_{\mathrm{B}}(X,N)\) is defined over \(\mathbb{Q}\), after replacing \(\varrho\) by a 
conjugate we may assume that it is defined over a number field \(k\). Let \(v\) be a 
non-archimedean place of \(k\), and let \(k_v\) denote the completion of \(k\) at \(v\).
We denote by
\[
\varrho_v:\pi_1(X)\longrightarrow \GL_N(k_v)
\]
the induced representation. If \(\varrho_v\) is unbounded for some non-archimedean place \(v\), then the same 
argument as in Case~1.2.2 yields the existence of a nontrivial logarithmic symmetric 
differential on \(X\), contradicting our assumption.
	
 Hence, we may assume that for every non-archimedean place \(v\) of \(k\), the representation 
 \(\varrho_v\) is bounded; this property is usually referred to as the \emph{integrality} of 
 \(\varrho\). 
 By the same argument as in \cref{prop:cons}, it follows that \(\varrho\) is a direct sum of a 
 \(\mathbb{C}\)-variation of Hodge structure \(\mathcal{L}\) with discrete and infinite 
 monodromy. 
 Applying the argument of Case~1.2.1, we conclude that \(X\) admits nontrivial symmetric 
 differentials, which yields a contradiction.
  
  \medspace
  
 The theorem of \cite{BKT13} is in fact more general, as it also applies to the compact 
 K\"ahler case. 
 In this setting, one does not have a satisfactory theory of analytic morphisms between the 
 moduli spaces \(M_{\mathrm{B}}(X,N)\) and \(M_{\mathrm{Dol}}(X,N)\).
 Nevertheless, Klingler \cite{Kli13} observed that the analytic gauge-theoretic arguments of 
 \cite{Sim92} suffice to show that the absence of symmetric differentials still forces 
 \(M_{\mathrm{B}}(X,N)\) to be zero-dimensional. 
 This idea also underlies the proof of the first version of \cite{BDDM}. 
\end{rem}

\begin{rem}
It is worth noting that Simpson proposed in \cite{Sim92} a celebrated conjecture on the 
integrality of rigid local systems on smooth complex projective varieties, which was proved by 
himself and Corlette in \cite{CS08} for rank two local systems.
	A major breakthrough on this conjecture was achieved by Esnault and Groechenig
	\cite{EG18,EG25b}, who proved Simpson's conjecture for cohomologically rigid local
	systems.    Esnault also pointed out to me that, if one knows that $M_{\rm B}(X,N)$ is
	zero-dimensional, then their arguments still apply and imply the integrality of
	arbitrary rigid local systems.
\end{rem}

\subsection{Holomorphic convexity}
In the previous subsections, we constructed the Shafarevich morphism for smooth quasi-projective varieties under the assumption that the fundamental group is reductive or linear in positive characteristic. However, this construction offers no insight into the holomorphic convexity of the universal covering of $X$ when $X$ is compact.

While the techniques developed above suffice for applications to the linear 
Chern--Hopf--Thurston conjecture (\cref{thm:CHT}), the proof of the linear case of 
Koll\'ar’s conjecture (\cref{main:kollar}) relies crucially on auxiliary results 
established in the study of the Shafarevich conjecture, notably in 
\cite{Eys04,DY23,EKPR12}. 
In this subsection, we outline the key ideas of this approach and present the 
essential ingredients needed for the proof of \cref{main:kollar}. 
More precisely, we  outline  a proof of \cref{conj:Sha} in the case of projective surfaces 
and sketch the main ideas underlying the general higher-dimensional case.
\subsubsection{Some analytic tools from harmonic mapping} 
We shall give some analytic tools.  Based on \cite[Proposition 4.15]{DM24}, in \cite[\S 3]{DW24b}, we prove  that  
\begin{lem}\label{lem:potential}
	Let $X$ be a smooth projective variety endowed with a    K\"ahler form $\omega$.  Let $\tau:\pi_1(X)\to G(K)$ be  a Zariski dense representation, where $G$ is a reductive group over   a non-archimedean local field $K$. Let $u:\widetilde{X}\to \Delta(G)_K$ be the $\tau$-equivariant pluriharmonic map ensured by Gromov-Schoen's theorem \cite{GS92}. Fix a base point $x_0\in \widetilde{X}$. Define a 1-form 
	\begin{align}\label{eq:1-form}
		\beta_\tau:=\sn\db d_{\Delta(G)}^2(u(x),u(x_0)). 
	\end{align} 
	 on the universal cover $\widetilde{X}$ of $X$.  	Then we have
	\begin{itemize}
		\item  $\beta_\tau$ has $L^1$-local coefficients, and is smooth almost everywhere.
		\item There exists some constant $C>0$ such that we have \[|\beta_\tau|_{\pi_X^*\omega}\leq_{\rm a.e.} C(1+d_{\widetilde{X}}(x,x_0)),\] where $d_{\widetilde{X}}(x,x_0)$ is the distance function on $\widetilde{X}$ induced by $\pi_X^*\omega$. 
	\item Define $\phi_{\tau_i} := d_{\Delta(G)}^{2}\bigl(u(x), u(x_0)\bigr)$, which is a continuous semipositive function on $\widetilde{X}$. Then we have
	\[
	d\beta_\tau = \sn \partial\bar{\partial} \phi_{\tau_i},
	\]
	which is a real, closed, and positive $(1,1)$-current, and it satisfies
	\[
	d\beta_\tau \geq \pi^* T_\tau .
	\] 
	Here $T_\tau$ is the canonical current associated with $\tau$ defined in \cref{def:canonical}. 
	\end{itemize}  
\end{lem}
In the archimedean case, Eyssidieux proved an analogous result in 
\cite[Proposition~4.5.1]{Eys00}, which we recall below.
\begin{lem}\label{lem:potential3}
Let \(\mathcal{L}\) be a \(\mathbb{C}\)-variation of Hodge structure on \(X\).
There exists a smooth  real-valued   plurisubharmonic function 
\(\phi_\cL\) on \(\widetilde{X}\), bounded from below, such that 
$$
\sn \d\db \phi_\cL\ge \pi_X^{*}\!\left(\sqrt{-1}\,\mathrm{tr}\,\theta \wedge \theta^{*}\right), \qquad 
\lvert d\phi_\cL\rvert_{\pi_X^{*}\omega} \le C .
$$
Here \(\theta\) denotes the Higgs field of \(\mathcal{L}\), and \(\theta^{*}\) its adjoint with 
respect to the Hodge metric.  
 \end{lem} 
 If we  introduce a  smooth 1-form  on $\widetilde{X}$ by setting 
 $$
 \beta_\cL:=\sn \db \phi_\cL, 
 $$
 then we have  
\begin{align}\label{lem:potential2}
	  d\beta_{\mathcal{L}} \ge \pi_X^{*}\!\left(\sqrt{-1}\,\mathrm{tr}\,\theta \wedge \theta^{*}\right),
	 \qquad 
	 \lvert \beta_{\mathcal{L}} \rvert_{\pi_X^{*}\omega} \le C.
\end{align} 
 
Next, we give  the criterion for an infinite topological Galois covering of a compact complex normal space to be Stein.   
\begin{lem}[{\cite[Proposition 4.1.1]{Eys04}}]\label{prop:stein}
Let \(S\) be a compact complex normal space, and let \(\nu:\Sigma \to S\) be an infinite 
topological Galois covering.  
Let \(T\) be a closed positive \((1,1)\)-current on \(S\) with continuous local potentials, 
whose cohomology class \(\{T\}\) is K\"ahler. 
Assume that there exists a real-valued continuous plurisubharmonic function 
\(\phi\) on \(\Sigma\), bounded from below, such that
\[
\sqrt{-1}\,\partial\bar{\partial}\phi \ge \nu^{*}T .
\]
Then \(\Sigma\) is a Stein space.  \qed
\end{lem}
\subsubsection{Two lemmas on finiteness and boundedness criteria}
We will need the following results concerning finiteness and boundedness criteria for subgroups. 
Their proofs rely on the geometry of Bruhat--Tits buildings.
\begin{lem}[{\cite[Lemma 5.3]{CDY22}}]\label{lem:BT2}
	Let \(G\) be an almost simple algebraic group defined over a non-archimedean local field \(K\). 
	Assume that \(\Gamma \subset G(K)\) is an unbounded subgroup whose Zariski closure contains 
	\(G^{\circ}(K)\), where \(G^{\circ}\) denotes the identity component of \(G\). 
	If \(N \triangleleft \Gamma\) is a bounded normal subgroup, then \(N\) is finite.
	\qed
\end{lem}
We refer the readers to \cite{Bru22} for an independent proof.
\begin{lem}[{\cite[Lemma 7.15]{DY23b}}]\label{lem:BT}
	Let $G$ be a semisimple algebraic group over a  
	non-archimedean local field $K$.  Let $\Gamma\subset G(K)$ be a finitely generated subgroup   whose Zariski closure contains 
	\(G^{\circ}(K)\), where \(G^{\circ}\) denotes the identity component of \(G\). 
	If  its derived group $\cD \Gamma$ is  \emph{bounded}, then  $\Gamma$ is also  bounded. \qed
\end{lem} 
\subsubsection{Ideas of proof of holomorphic convexity}
 In this subsection, we outline the proof of the reductive Shafarevich conjecture for smooth 
projective surfaces. By \cref{thm:Shafarevich}, one can show that, after passing to a finite étale cover, an almost 
faithful representation factors through a large local system on the Shafarevich variety 
\(\mathrm{Sh}(X)\). 
Consequently, it suffices to establish the Steinness of the universal cover in the case of 
\emph{large and reductive} representations.  
\begin{thm}
	Let $X$ be a smooth projective surface. If there exists a semisimple and  large representation $\varrho:\pi_1(X)\to \GL_{N}(\bC)$, then the universal covering $\widetilde{X}$ of $X$ is Stein.
\end{thm}
\begin{proof}
	Let $\cL$ be the $\bC$-VHS defined in \cref{prop:cons}.  Let  $\theta$ be its Higgs field   and $\theta^*$ be  the  adjoint of $\theta$ with respect to the Hodge metric.  
	%Define a 1-form on $\widetilde{X}$  $$\beta_\btau:=\sum_{i=1}^{m}\beta_{\tau_i}$$ 
%	where $\beta_{\tau_i}$ is the 1-form associated with $\tau_i$ defined in \eqref{eq:1-form}.  
	
	\medspace
	
	\noindent {\it Case 1: The set $\Upsilon$ in \cref{def:fac} contains no unbounded representation.  }  
	
	This means that $s_{\rm fac}:X\to S_{\rm Fac}(X)$ is a constant map.   Then by \cref{prop:cons},  $\varrho$ is  a direct factor of a $\bC$-VHS $\cL$ with discrete monodromy.  	Since $\varrho$ is large,  the monodromy representation of $\cL$  is thus also large.  Hence, by \cref{lem:period}  its period map $p:X\to \sD/\Gamma$ is finite.  One can show that for the real \((1,1)\)-form 
	\(\sqrt{-1}\,\mathrm{tr}\,\theta \wedge \theta^{*}\), its cohomology class
	\[
	\bigl\{\sqrt{-1}\,\mathrm{tr}\,\theta \wedge \theta^{*}\bigr\}
	\]
	is a K\"ahler class.
	By \cref{lem:potential3}, there exists a smooth  real-valued   plurisubharmonic function 
	\(\phi_\cL\) on \(\widetilde{X}\), bounded from below, such that 
\[\sn\d\db \phi_{\mathcal{L}} \ge \pi_X^{*}\!\left(\sqrt{-1}\,\mathrm{tr}\,\theta \wedge \theta^{*}\right).
	\]
	Then, by \cref{prop:stein}, the universal cover \(\widetilde{X}\) is Stein.
	
	\medspace
	
	From now on, we assume that $\Upsilon$  contains at least one  unbounded representation. By \cref{lem:simultaneous}, there exists  unbounded and  reductive representations $\btau:=\{\tau_i:\pi_1(X)\to \GL_{N}(K_i)\}_{i=1,\ldots,k}$ with $K_i$ a non-archimedean local field of characteristic $0$ such that the Stein factorization of  products of   Katzarkov-Eyssidieux reduction maps $$s_{\tau_1}\times\cdots\times s_{\tau_k}:X\to S_{\tau_1}\times\cdots\times S_{\tau_k}$$ coincides with $s_{\rm fac}:X\to S_{\rm Fac}$.     
	Let $\xsp\to X$ be the spectral covering associated with $\btau$.   Let  $\{\eta_1,\ldots,\eta_k\}\subset H^0(\xsp, \pi^*\Omega_X^1)$  be the associated spectral 1-forms as in \cref{prop:spectral}.  We define the \emph{rank} as follows:     if there exist   \( i \) and \( j \) such that \( \eta_i \wedge \eta_j \neq 0 \), then we say that the rank is \( 2 \); otherwise, the rank is \( 1 \). 
	
	\medspace
	
	\noindent {\it Case 2: The spectral 1-forms have rank 2}.   By our construction of canonical currents in \cref{def:canonical}, we have
	$$
	C\pi^*(T_{\tau_1}+\cdots+T_{\tau_k})=\sn\sum_{i=1}^{m} \eta_i\wedge\overline{\eta_i}
	$$
	for some constant $C>0$. 
	Then 
	$$T_\btau:=T_{\tau_1}+\cdots+T_{\tau_k}$$ is strictly positive at general points since the spectral forms associated with $\btau$ has rank 2. Since $T_\btau$ has continuous local potentials, it follows that $T_\btau$  is a  big and nef class.   By the construction of $T_\btau$, one can show that there exists a closed positive $(1,1)$-current $T$ on $S_{\rm Fac}(X)$ such that 
	$$
	s_{\rm fac}^*T=T_\btau. 
	$$
	This in particular shows that $s_{\rm fac}$ is a birational morphism. 
	On the other hand, as a consequence of \cref{thm:KE}, we can show that  $\{T\}\cdot C>0$ for any irreducible curve $C$ on $S_{\rm Fac}(X)$. 	  
	We now apply a theorem by Demailly-P\u{a}un \cite{DP04} (or \cite{Lam99}) to  conclude that $\{T\}$ is a K\"ahler class on $S_{\rm Fac}(X)$.   
	
	On the other hand,  for each positive dimensional fiber $F$ of $s_{\rm fac}$, by \cref{prop:cons}, $\cL|_{F^{\rm norm}}$ is a $\bC$-VHS with discrete monodromy, and $\varrho|_{\pi_1(F^{\rm norm})}$ is a direct factor of  $\cL|_{F^{\rm norm}}$. Since $\varrho$ is large, then the period map $F^{\rm norm}\to \sD/\Gamma$ is finite, where $\Gamma$ is the monodromy group of $\cL|_{F^{\rm norm}}$. One can show that
	$$
	\{(	 	\sn {\rm tr}\theta\wedge\theta^*)|_{F}\}
	$$  
	is a K\"ahler class. Therefore,
	$$
	\{	\sn {\rm tr}\theta\wedge\theta^*+T_{\btau}\} 
	$$
	is a K\"ahler class on  $X$.  Let $\phi_{\tau_i}$ be the continuous semipositive function associated with $\tau_i$, as defined in \cref{lem:potential}, and let $\phi_{\cL}$ be the continuous function, bounded from below, associated with $\cL$, as defined in \cref{lem:potential3}.
	   Then we have  
	$$
\sn\d\db(\phi_\cL+\sum_{i=1}^{k}\phi_{\tau_i}) \geq 	\pi_X^*(\sn {\rm tr}\theta\wedge\theta^*+T_{\btau})
	$$
	By \cref{prop:stein}, we conclude that $\widetilde{X}$ is Stein.  
	
	\medspace
	
	%By the construction of $s_{T_{k+1}}$, there exists a   reductive representation $\tau:\pi_1(X)\to \GL_{N}(\overline{\bF_q((t))})$ such that $s_\tau(T)$ is not a point. 

	\noindent {\it Case 3: Regardless of the elements $\tau_1,\ldots,\tau_m \in \Upsilon$ chosen, their spectral 1-forms have rank 1}.

	\noindent {\it Case 3.1: $\dim S_{\rm Fac}(X)=1$}.  
	
	The proof is the same as  Case 2. We leave it as an exercise to the readers. 
	
	\noindent {\it Case 3.2: $\dim S_{\rm Fac}(X)=2$}.   
	
	\noindent {\it Case  3.2.1: The dimension of spectral 1-forms is at least two}: 
	 
	Suppose  
	$
	{\rm dim}_\bC{\rm Span}\{\eta_1,\ldots,\eta_\ell\}\geq 2.
	$

	Without loss of generality, we may assume that $\eta_1\wedge\eta_2\equiv 0$ and $\eta_1\not\in \{\bC\eta_2\}$. 
	According to the Castelnuovo-De Franchis theorem (cf. \cite[Theorem 2.7]{ABCKT}), there exists a proper fibration $h:\xsp\to C$ over a smooth projective curve $C$ such that $\{\eta_1,\eta_2\}\subset h^*H^0(C,\Omega_C^1)$.  
	Since $s_{\rm fac}$ 
	is birational,  we can choose a general fiber $F$ of $h$, which is irreducible and  such that $s_{\rm fac}\circ\pi(F)$  
	is not a point. 
	%Here $e_{T}:X_0\to S_{T}$ is the natural morphism.  %Thanks to \cref{lem:strictly nef}, 
	%	\begin{align} \label{eq:nef}
		%	\{\cT_{T}\}\cdot {\rm sh}_M\circ\pi(F)>0
		%\end{align}  
		By \cref{thm:KE}, we can show that there exists some $i$ such that $\eta_i|_{F}\neq 0$. 
		Given that  $\eta_1|_{F}\equiv 0$, this implies that $\eta_i\wedge\eta_1\neq 0$. It contradicts with our assumption that the spectral $1$-forms have rank 1.  
		Therefore, this case cannot occur. 
		
		\medspace
		
		We now turn to the final—and most difficult—case, which also relies on several deep results  from non-abelian Hodge theories by Simpson. 
		
		\medspace
		
		\noindent \textit{Case 3.2.2: Regardless of the elements $\tau_1,\ldots,\tau_m \in \Upsilon$ chosen, the dimension of    spectral 1-forms remains 1.}

		Pick any non-archimedean representation $\tau:\pi_1(X)\to \GL_{N}(K)$ that is unbounded.  
		Let  $G$ be the Zariski closure of $\tau(\pi_1(X))$, which is reductive.  Consider the isogeny
		$g:G\to G/Z\times G/\cD G$ 
		where $Z$ is the central torus of $G$ and $\cD G$ is the derived group of $G$. 
		As a result, $G':=G/Z$  is semisimple and $G'':=G/\cD G$ is a torus. 
		Let
		$
		\tau':\pi_1(X)\to G'(\overline{K}) 
		$ 
		be the composite of $\tau$ with the projection $G\to G'$, and  
		$
		\tau'':\pi_1(X)\to   G''(\overline{K}) 
		$ 
		be the composite of $\tau$ with the projection $G\to G''$.  Then $\tau'$ and $\tau''$ are both Zariski dense representations. 
		\begin{claim}
			The representation $\tau'$ is bounded. 	 
		\end{claim}
		\begin{proof}
			Let $\xsp\to X$ be the spectral covering associated with $\tau$ and let $\nu:Y\to \xsp$ be a desingularization. By the assumption, the spectral 1-forms are a subset of $\bC \eta$, where $\eta$ is a  1-form in $  H^0(Y,\Omega_Y^1)$. Consider the \emph{partial Albanese morphism} $a:Y\to A$ induced by $\eta$ (see \cite[Definition 3.25]{CDY22} for the definition).   Then there exists a one form $\eta'\in H^0(A,\Omega_A^1)$ such that $a^*\eta'=\eta$.  If $\dim a(Y)=1$, then the Stein factorization $h:Y\to C$ of $a$ is a proper holomorphic fibration   over a smooth projective curve $C$ such that $\eta_1\in h^*H^0(C,\Omega_C^1)$.  We are now in  a situation akin to Case 3.1, and we can apply the same arguments to reach a contradiction.  Hence $\dim a(Y)=2$.   Let $\pi_A:\widetilde{A}\to A$ denote the universal covering map. We denote by $Y':=Y\times_{\widetilde{A}}A$  a connected component of the fiber product and let $\pi':Y'\to Y$ be the induced   \'etale cover.   
			It is worth noting that $\pi'^*\eta$ is  exact. Consequently, we can define the following holomorphic map: 
			\begin{align*}
				h: Y'&\to \bC\\
				y&\mapsto \int_{y_0}^{y}\pi'^*\eta.
			\end{align*} We then have the following commutative diagram:
			\begin{equation*}
				\begin{tikzcd}
					\widetilde{Y}\arrow[r, "p"]\arrow[rr, bend left=30, "\pi_Y"]&	Y'\arrow[dd, bend right=30, "h"']\arrow[r, "\pi'"] \arrow[d] & Y\arrow[d, "a"]\\
					&	\widetilde{A}\arrow[d] \arrow[r, "\pi_A"]& A\\
					&		\bC
				\end{tikzcd}
			\end{equation*} 
			The   holomorphic map $\widetilde{A}\to \bC$ in the above diagram is  defined by the linear 1-form $\pi_A^*\eta'$ on $\widetilde{A}$.   
			By Simpson's Lefschetz theorem \cite{Sim93}, for any $t\in \bC$, $h^{-1}(t)$ is connected and $\pi_1(h^{-1}(t))\to \pi_1(Y')$ is surjective.  
			By  
			definition of $h$,  $\pi_Y^*\eta|_{Z}\equiv 0$ where $Z$ is any connected component of $p^{-1}(h^{-1}(t))$. 
			Here $p:\widetilde{Y}\to Y'$ is the natural covering map. 
			
			Consider the  Zariski dense  representation $\tau':\pi_1(X) \to G'\big( \overline{K}\big)$ as defined previously. 
			Let  $L$ be a finite extension of $K$ such that $G'$ is defined on $L$ and  $\tau':\pi_1(X)\to G'(L)$. 
			We denote by $\sigma:\pi_1(Y)\to G'\big(L\big)$ the pullback of $\tau'$ via the morphism $Y\to X$.   
			The existence of a $\sigma$-equivariant harmonic mapping $u:\widetilde{Y}\to \Delta(G')_L$ is guaranteed by \cite{GS92}, where $\Delta(G')_L$ is the Bruhat-Tits building of $G'$.
			
			We note that $\pi_Y^*\eta$ is the (1,0)-part of the complexified differential of the harmonic mapping $u$ at general points of $\widetilde{Y}$, with $\pi_Y:\widetilde{Y}\to Y$ denoting the universal covering. 
			For  any connected component $Z$ of $p^{-1}(h^{-1}(t))$ for a general $t\in \bC$, since $\pi_Y^*\eta|_{Z}\equiv 0$,   and  all the spectral forms are assumed to be $\bC$-linearly equivalent, it follows that $u(Z)$ is a point $P\in \Delta(G')_L$. 
			Since $u$ is $\sigma$-equivariant, it follows that $\pi'^*\sigma({\rm Im}[\pi_1(h^{-1}(t))\to \pi_1(Y')])$ is contained in the subgroup of $G'(L)$  fixing the point $P$, which is thus bounded. 
			Recall that $\pi_1(h^{-1}(t))\to \pi_1(Y')$  is surjective. 
			Hence $\pi'^*\sigma(\pi_1(Y'))$ is a bounded subgroup of $G'(L)$. 	
			Additionally, note that the derived group $$\cD \pi_1(Y)\subset {\rm Im}[\pi_1(Y')\to \pi_1(Y)],$$ and it follows that $\sigma(\cD \pi_1(Y))$ is bounded.  	
			Since $\tau'$ is Zariski dense, and ${\rm Im}[\pi_1(Y)\to \pi_1(X)]$ is a finite index subgroup of $\pi_1(X)$,    the Zariski closure  of $\sigma(\pi_1(Y))$ contains the identity component of $G'$, and it is also semisimple.     
			We apply \cref{lem:BT}  to conclude  that $\sigma( \pi_1(Y))$ is bounded.	  Hence $\tau'$ is bounded. 
		\end{proof} 
		By \cref{thm:KE}, the Katzarkov-Eyssidieux reduction map $s_{\tau'}$ is constant. This implies that the reduction map $s_{\tau}$ is identified with $s_{\tau''}$.
		
		Note that $G''$ is a torus. A key idea originally introduced by Eyssidieux in \cite{Eys04} is the use of Simpson's \emph{absolutely constructible subsets} to handle this situation. We sketch the arguments  in \cite[Theorem 4.5]{DY23}; the following discussion is intended as an outline rather than a rigorous proof.
		
		Consider the composition of natural morphisms among Betti moduli spaces:
		\[
	\Psi:	M_{\rm B}(X,G)\to M_{\rm B}(Y,G) \to M_{\rm B}(Y,G'').  
		\]
		By \cite{Sim93b}, the image of $M_{\rm B}(X,G)$ under this composition is an \emph{absolutely constructible subset}, denoted by $M_{\rm acs}$, of $M_{\rm B}(Y,G'')$. Since $G''$ is a torus, $M_{\rm B}(Y,G'')$ is essentially a product of copies of $M_{\rm B}(Y,1)$. 
		
		For simplicity, we assume that $G''$ is a one-dimensional torus. Then Simpson's theorem \cite{Sim93b} states that the closure of $M_{\rm acs}$ is a finite union of torsion translates of subtori in $M_{\rm B}(Y,1)$. 
		
	Since we assume that $\dim a(Y)=2$, a crucial step in \cite[Theorem~4.5]{DY23} shows that one can find sufficiently many representations
	\[
	\{\sigma_i \colon \pi_1(X) \to G(K_i)\}_{i=1,\ldots,\ell},
	\]
	where each $K_i$ is a non-archimedean local field of characteristic zero, such that their images
	\[
	\{\Psi(\sigma_i) \colon \pi_1(X) \to \GL_1(K_i)\}_{i=1,\ldots,\ell}
	\]
	in $M_{\rm acs}$ are unbounded representations, and the holomorphic $1$-forms on $Y$ induced by them have rank~$2$.
	This implies that the associated spectral $1$-forms have rank~$2$, which leads to a contradiction. 
	\end{proof}
	\begin{rem}\label{rem:byproduct}
		In the general case, one must extend Simpson’s Lefschetz theorem for leaves of holomorphic foliations defined by \emph{systems of holomorphic $1$-forms}. 
		This extension is a crucial step in the proof of the reductive Shafarevich conjecture and was established by Eyssidieux in \cite{Eys04}. 
		Based on this general Lefschetz theorem, together with a suitable application of absolutely constructible subsets, we prove in \cite[Proof of Theorem~4.31]{DY23} that there exist representations
		\[
		\{\tau_i \colon \pi_1(X) \to G(K_i)\}_{i=1,\ldots,\ell},
		\]
		where each $K_i$ is a non-archimedean local field of characteristic zero, such that the sum of the canonical currents associated with these representations,
$
		\sum_{i=1}^{\ell} T_{\tau_i},
$
		is the pullback via $s_{\rm fac}$ of a closed positive $(1,1)$-current $T_{\rm fac}$ on $S_{\rm Fac}(X)$, whose cohomology class $\{T_{\rm fac}\}$ is \emph{K\"ahler}.
		
		Recall that for the $\bC$-VHS $\cL$ on $X$ constructed in \cref{prop:cons}, its restriction to each fiber $F$ of $s_{\rm fac}$ has discrete monodromy. 
		If we assume that there exists a big and semisimple representation on $X$, then for a general fiber $F$ of $s_{\rm fac}$, the period map of $\cL|_{F}$ is generically finite onto its image. 
		Consequently, one can show that the cohomology class
		\[
		\left\{ \sum_{i=1}^{\ell} T_{\tau_i} + \sn\,{\rm tr}(\theta \wedge \theta^*) \right\}
		\]
		is big and nef on $X$. 
		This fact is crucial in the proof of the linear Koll\'ar conjecture (\cref{main:kollar}). 
	\end{rem}
	\begin{rem}
	In \cite[Appendix]{DY23}, Katzarkov, Yamanoi, and the   author proved the reductive Shafarevich conjecture for projective \emph{normal} varieties $X$ admitting a faithful reductive representation
	\[
	\varrho \colon \pi_1(X) \to \GL_N(\mathbb{C}),
	\]
	thereby extending the results of \cite{Eys04} to the singular setting.
	We now outline the main strategy of the proof in \cite{DY23}.

		Let \(\mu:Y\to X\) be a desingularization. 
		Since the induced morphism \(\pi_1(Y)\to \pi_1(X)\) is surjective, 
		the natural morphism of Betti moduli spaces
		\[
		\iota:M_{\mathrm{B}}(X,N)\hookrightarrow M_{\mathrm{B}}(Y,N)
		\]
		induced by~\(\mu\) is a closed immersion.
		Using a theorem of Lerer~\cite{Ler22}, we show that the image of~\(\iota\) is an 
		\emph{absolutely constructible subset} of \(M_{\mathrm{B}}(Y,N)\) in the sense of 
		Budur--Wang~\cite{BW20}.
		
		Define the subgroup
		\[
		\Gamma:=\bigcap_{\tau}\ker\tau,
		\]
		where \(\tau:\pi_1(Y)\to \GL_{N}(\mathbb{C})\) ranges over all reductive representations 
		lying in \(\iota(M_{\mathrm{B}}(X,N))\).
		We then prove that the covering \(\widetilde{Y}_{\Gamma}\to Y\)  of $Y$ is holomorphically convex.
		
		Next, define
		\[
		\Gamma':=\bigcap_{\tau}\ker\tau,
		\]
		where \(\tau:\pi_1(X)\to \GL_{N}(\mathbb{C})\) ranges over all reductive representations in 
		\(M_{\mathrm{B}}(X,N)\).
		The morphism~\(\mu\) lifts to a proper holomorphic map
		\(\tilde{\mu}:\widetilde{Y}_{\Gamma}\to \widetilde{X}_{\Gamma'}\), fitting into the 
		commutative diagram
		\[
		\begin{tikzcd}
			\widetilde{Y}_{\Gamma}\arrow[r]\arrow[d,"\tilde{\mu}"] & Y\arrow[d,"\mu"]\\
			\widetilde{X}_{\Gamma'}\arrow[r] & X .
		\end{tikzcd}
		\]
		By \cref{thm:Cartan}, it follows that \(\widetilde{X}_{\Gamma'}\) is holomorphically convex.
		Since the representation \(\varrho:\pi_1(X)\to \GL_{N}(\mathbb{C})\) is faithful and reductive, 
		we conclude that \(\widetilde{X}_{\Gamma'}=\widetilde{X}\), which completes the proof.
	\end{rem}
Let me conclude this section by presenting a recent application of the linear Shafarevich conjecture (\cref{thm:EKPR}). In \cite{EF25}, Eyssidieux and Funar established new constraints on algebro-geometric subgroups of mapping class groups. As a consequence, they proved that the universal cover of the smooth proper Deligne–Mumford stack  $\overline{\mathcal{M}}_{g,n}[{\bf{k}}]$, the universal covering space of the stack is, in most cases, a Stein manifold. Further applications of the techniques used to study the linear Shafarevich conjecture, as presented in this section, will be discussed in the subsequent sections.
\subsection{On the linear Shafarevich conjecture}
Proving the linear Shafarevich conjecture required invoking the deep results of Eyssidieux and Simpson \cite{ES11} regarding the construction of suitable tautological representations arising from variations of mixed Hodge structures. In this section, we utilize a more recent development from \cite{CDP25} concerning  the \emph{1-step phenomenon} in Hodge theory to provide a simplified proof for a special case. This approach avoids the complex reductive machinery in \cite{DY23,Eys04} discussed above, and  \cite{ES11}.

\begin{thm}[{\cite[Theorem 4.5]{CDP25}}]\label{thm:1-step}
	Let $h:Y \to X$ be a holomorphic map between compact K\"ahler manifolds.
	Let $(V,\nabla)$ be a semisimple flat bundle on $X$ with monodromy 
	representation $\varrho:\pi_1(X,x)\to \GL_N(\mathbb{C})$, and set $y=h(x)$.
	Denote by 
	\[
	h^*: R(X,N) \longrightarrow R(Y,N)
	\]
	the induced morphism of representation varieties defined in \eqref{eqdef:repre}, given by
	$\tau \mapsto h^*\tau$ for any $\tau \in R(X,N)$.
	
	Assume that:
	\begin{enumerate}[label=(\alph*)]
		\item\label{item:center} the pullback representation $h^*\varrho$ is trivial;  
		\item for every $\eta \in H^1(X,\mathrm{End}(V))$, one has $h^*\eta = 0$ in 
		$H^1(Y,h^*\mathrm{End}(V))$.
	\end{enumerate}
	
	Let $R$ be any (geometrically) irreducible component of $R(X,N)$ 
	containing $\varrho$.  
	Then the image of $R$ under $h^*$ consists solely of the trivial representation.  
\end{thm}

\begin{thm}
	Let $X$ be a smooth projective variety. Assume that there is a linear and faithful representation $\varrho:\pi_1(X)\to \GL_{N}(\mathbb{C})$. Assume additionally that $\varrho$ lies in the same irreducible component $R(\mathbb{C})$ of $R(X,N)$ containing the trivial representation. Then the universal covering $\widetilde{X}$ of $X$ is holomorphically convex. 
\end{thm}

\begin{proof}  
	Let $a:X\to A$ be the Albanese morphism of $X$, and let $f:X\to S$ be the Stein factorization of $a$. 
	
	\begin{claim}\label{claim:faithful}
		The morphism $f$ is the Shafarevich morphism of $\varrho$. 
	\end{claim}
	\begin{proof}[Proof of \cref{claim:faithful}]
		Let $h:Y\to X$ be any morphism from a smooth projective variety $Y$ such that $f(h(Y))$ is a point. By the universal property of the Albanese map, this implies $h^*H^1(X,\mathbb{C})=0$. Hence, for the trivial flat bundle $(V,\nabla)$ of rank $N$, we have $h^*H^1(X,\mathrm{End}(V))=0$. Because $\varrho$ and the trivial representation $\varrho_0:\pi_1(X)\to \GL_{N}(\mathbb{C})$ lie in the same irreducible component of $R(X,N)$, it follows from  \cref{thm:1-step}  that $h^*\varrho=h^*\varrho_0$, which is the trivial representation.  
		
		Conversely, consider any morphism $g:Z\to X$ from a smooth projective variety $Z$ such that $f(g(Z))$ is not a point. In this case, the pullback 
		$$
		g^*H^1(X,\mathbb{C})\neq \{0\}. 
		$$
		This non-triviality guarantees that $\mathrm{Im}[\pi_1(Z)\to \pi_1(X)]$ is not finite. Consequently, $f$ satisfies the exact conditions to be the Shafarevich morphism of $\varrho$. 
	\end{proof} 
	
	Let $\widetilde{X}^{\mathrm{ab}}$ be a connected component of the fiber product $X \times_A \widetilde{A}$, where $\widetilde A\to A$ denotes the universal covering of $A$, and let $\varphi: \widetilde{X}^{\mathrm{ab}} \to W$ be the Stein factorization of the induced map $\widetilde{X}^{\mathrm{ab}} \to \widetilde{A}$. 
	
	Let $\pi: \widetilde{X} \to \widetilde{X}^{\mathrm{ab}}$ be the universal cover of $\widetilde{X}^{\mathrm{ab}}$ (which is also the universal cover of $X$), and consider the composition $\Phi := \varphi \circ \pi: \widetilde{X} \to W$. 
	
	Replacing $X$ by a finite \'etale cover, we may assume that $\varrho(\pi_1(X))$ is torsion-free. By \cref{claim:faithful} and \cref{lem:Gurjar}, for each fiber $F$ of $f:X\to S$, the image 
	$$
	\varrho(\mathrm{Im}[\pi_1(F )\to \pi_1(X)] )
	$$
	is finite, and thus trivial. Since $\varrho$ is a faithful representation, it follows that
	$$
	\mathrm{Im}[\pi_1(F )\to \pi_1(X)] 
	$$
	is itself trivial. This implies that for any $y \in W$, its preimage $\Phi^{-1}(y) = \pi^{-1}(\varphi^{-1}(y))$ in $\widetilde{X}$ is a disjoint union of copies of the compact connected space $\varphi^{-1}(y)$. In particular, the connected components of the fibers of $\Phi$ are compact.
	
	By a theorem of Cartan \cite{Car60}, the set $Z$ of connected components of the fibers of $\Phi$ can be endowed with the structure of a normal complex space such that $\Phi$ factors as $h \circ \psi$, where $\psi: \widetilde{X} \to Z$ is a proper holomorphic fibration and $h: Z \to W$ is a map with discrete fibers. This yields the following commutative diagram:
	\begin{equation*}
		\begin{tikzcd}
			\widetilde{X} \arrow[d, "\psi"'] \arrow[r, "\pi"]\arrow[dr,"\Phi"]  & \widetilde{X}^{\mathrm{ab}}\arrow[d, "\varphi"] \arrow[r] & X \arrow[dd, "a"] \\
			Z \arrow[r, "h"'] & W \arrow[d] & \\
			& \widetilde{A} \arrow[r] & A
		\end{tikzcd}
	\end{equation*}
	
	Let $G := \mathrm{Aut}(\widetilde{X}/\widetilde{X}^{\mathrm{ab}})$ be the Galois group of the covering $\pi$. Note that the preimage $\pi^{-1}(\varphi^{-1}(y))$ is exactly the $G$-orbit of a single connected component in $\widetilde{X}$, and thus the group $G$ acts transitively and freely on the set of connected components of the fibers of $\Phi$. Consequently, $G$ induces a free action on the complex space $Z$. Furthermore, this action is properly discontinuous (see, e.g., \cite[Lemma 3.34]{DY23}). 
	
	By construction, the quotient $Z/G$ is canonically isomorphic to $W$, meaning $h: Z \to W$ is a topological covering map with Galois group $G$. The map $W \to \widetilde{A}$ is finite and proper, and since $\widetilde{A} \cong \mathbb{C}^{\dim A}$ is Stein, $W$ is Stein. Because $h: Z \to W$ is a topological covering of a Stein space, $Z$ must also be Stein. 
	
	Finally, since $\psi: \widetilde{X} \to Z$ is a proper holomorphic map to the Stein space $Z$, it follows from the Cartan-Remmert theorem that $\widetilde{X}$ is holomorphically convex. 
\end{proof}   

If $\varrho$ does not lie in any irreducible component of $R(X,N)(\mathbb{C})$ containing the trivial representation, one needs the reductive Shafarevich conjecture, proved in \cite{Eys04,DY23}, together with \cite{ES11}, to handle the general case for \emph{smooth projective varieties}. Beyond the smooth projective setting, building on \cite{DY23} and on ideas from \cite{EKPR12,CCE15}, the author recently proved the linear Shafarevich conjecture for \emph{normal compact K\"ahler varieties} \cite{Den26}, thereby completely answering a question raised in \cite[p. 1549]{EKPR12}.
\begin{thm}[{\cite[Corollary B]{Den26}}]\label{thm:linear}
	Let $X$ be a normal compact K\"ahler variety. If there is a
	faithful (resp. large) linear representation
	$
	\varrho:\pi_1(X)\to\mathrm{GL}_N(\bC),
	$
	then the universal cover of $X$ is holomorphically convex (resp. Stein).
\end{thm}
To establish this theorem, the author developed new methods in non-abelian Hodge theory for normal compact K\"ahler varieties $X$, including:
\begin{itemize}
	\item an extension of the Goldman--Millson theorem: the representation variety $R_{\mathrm{B}}(X,N)$ has at most quadratic singularities at semisimple representations;
\item an extension of Simpson's ubiquity theorem \cite{Sim92}, proved there for connected components in the projective setting: every \emph{irreducible} component of $M_{\mathrm{B}}(X,N)$ contains a representation underlying a $\bC$-VHS;
	\item an extension of the Eyssidieux--Simpson construction of $1$-step $\bC$-VMHS;
\item a new analytic proof of the rationality lemma of \cite{EKPR12}, which is a key ingredient in the proof of \cref{thm:linear}: our proof is valid in the singular K\"ahler setting, and relies on explicit analytic constructions in the deformation theory of local systems inspired by \cite{CDHP25} together with a cohomology non-jumping result. 
\end{itemize}Several of these results are of independent interest in non-abelian Hodge theory. We do not discuss them further here; since \cite{Den26} is written to be as self-contained as possible, we refer the interested reader to it.
		\section{Hyperbolicity and fundamental groups: ideas of proofs}\label{sec:hyperbolicity}
	In this section, we sketch the main idea behind the proofs of   \cref{main:hyper,thm:GGL}.   
	\subsection{A theorem of Nevanlinna theory on  semiabelian varieties}\label{sec:second}
	In this subsection, we recall some notions from Nevanlinna theory. Readers who are not familiar with Nevanlinna theory are referred to the excellent survey by Yamanoi \cite{Yam15b} or Demailly \cite{Dem97}.

	In this section, $A$ is a semi-abelian variety and $Y$ is a Riemann surface with a proper surjective holomorphic map $\pi:Y\to\mathbb C_{>\delta}$, where $\mathbb C_{>\delta}:=\{z\in\bC\mid \delta<|z|\}$ with some fixed positive constant $\delta>0$.
	
	For $r>2\delta$, define
	$
	Y(r)=\pi^{-1}\big( \mathbb C_{>2\delta}(r) \big)
	$ 
	where $\mathbb C_{>2\delta}(r)=\{z\in \bC\mid  2\delta<|z|<r\}$. 
	In the following, we assume that $r>2\delta$.
	The \emph{ramification counting function} of the covering $\pi:Y\to  \bC_{>\delta}$ is defined by
	$$
	N_{{\rm ram}\,  \pi}(r):=\frac{1}{{\rm deg} \pi}\int_{2\delta}^{r}\left[\sum_{y\in {Y}(t)} \ord_y \ram \pi \right]\frac{dt}{t},
	$$ 
	where $\ram \pi\subset Y$ is the ramification divisor of $\pi:Y\to\mathbb C_{>\delta}$.

	Let $L$ be a line bundle on $X$.
	Let $f:Y\to X$ be a holomorphic map.
	We define the order function $T_f(r,L)$ as follows.
	First suppose that $X$ is smooth.
	We equip with a smooth hermitian metric $h_L$, and let $c_1(L,h_L)$ be the curvature form of $(L, h_L)$.  
	$$
	T_f(r,L):=\frac{1}{\operatorname{deg} \pi} \int_{2\delta}^{r}\left[\int_{Y(t)} f^{*}c_1(L,h_L)\right] \frac{d t}{t}.
	$$
	This definition is independent of the choice of  the hermitian metric up to a function  $O(\log r)$.  
	
	\begin{thm}[{\cite[Theorem A]{CDY25}}]\label{thm2nd}
		Let $X$ be a smooth quasi-projective variety which is of log general type.
		Assume that there is a morphism $a:X\to A$ to   a semi-abelian variety $A$ such that $\dim X=\dim a(X)$.
		Then there exists a proper Zariski closed set $\Xi\subsetneqq X$ with the following property:
		let $f:Y\to X$ be a holomorphic map such that $$N_{\ram\pi}(r)=O(\log r)+o(T_f(r))||$$ and that $f(Y)\not\subset \Xi$. 
		Here, we use the symbol $||$ to denote that the stated estimate holds for $r\to\infty$  outside an exceptional set of finite Lebesgue measure.

		Then $f$ does not have essential singularity over $\infty$, i.e.,	there exists an extension $\overline{f}:\overline{Y}\to\overline{X}$ of $f$, where $\overline{Y}$ is a Riemann surface such that  $\pi:Y\to \mathbb C_{>\delta}$ extends to a proper map $\overline{\pi}:\overline{Y}\to \mathbb C_{>\delta}\cup\{\infty\}$  
		and $\overline{X}$ is a compactification of $X$.	
	\end{thm}
	Note that \cref{thm2nd} is proven by Yamanoi in \cite{Yam15} when $X$ is compact. Its proof  is based on techniques in Nevanlinna theories  in \cite{Yam15}.  Compared with the compact case treated in \cite{Yam15}, the lack of Poincar{\'e} reducibility theorem is a major difficulity to treat the non-compact case.
	We use a more general ``cover'' than {\'e}tale cover to overcome this problem.   We refer the readers to \cite[Remark 10.11]{CDY25} for the main difficulty and novelty in the non-compact cases.  Since the proof of \cref{thm2nd} is highly involved and unrelated to other aspects of the paper, we choose to omit it. Instead, we present a fundamental result from Nevanlinna theory.
	\begin{claim}
		Let $f:Y\to X$ be as above. Fix any ample line bundle $L$ on $X$.  If the order function $T_f(r,L)=O(\log r)$ as $r\to\infty$, then $f$ does not have essential singularity at infinity. 
	\end{claim}
	In a nutshell, the ultimate goal in proving  \cref{thm2nd} is to estimate  the order  function $T_f(r,L)$ utilizing Nevanlinna theory tools, such as the logarithmic derivative lemma, the Second Main Theorem, jet differentials, and other related techniques.

	In the context of Nevanlinna theory in \cite[\S 10]{CDY25}, another crucial result is obtained.
	\begin{thm}[{\cite[Corollary 10.8]{CDY25}}]\label{cor:GGL}
		Let $X$ be a smooth quasi-projective variety and let $a:X\to A\times S$ be a morphism such that $\dim X=\dim a(X)$, where $S$ is a smooth quasi-projective variety ($S$ can be a point). Write $b:X\to S$ as the composition of $a$ with the projection map $A\times S\to S$. Assume that $b$ is dominant.
		\begin{thmlist}
			\item\label{coritem1} 
			Suppose $S$ is pseudo Picard hyperbolic.
			If $X$ is of log general type, then $X$  is pseudo Picard hyperbolic.
			\item \label{coritem2} Suppose $S$ is strongly of log general type. 
			If $X$ is pseudo Brody hyperbolic, then $X$ is strongly of log general type.   \qed
		\end{thmlist} 
	\end{thm}

	\subsection{Hyperbolicity and non-archimedean local systems}
	A crucial step for the proof of  \cref{main:hyper} is the following result.
	
	\begin{thm}[{\cite[Theorem F]{CDY22}}]\label{main:unbounded hyperbolic}
		Let $X$ be a quasi-projective normal variety and let $G$ be an   almost simple algebraic group defined over a non-archimedean local field $K$. If $\varrho:\pi_1(X)\to G(K)$ is a big, Zariski dense, and unbounded representation, then $X$ is   of log general type, and  pseudo Picard hyperbolic.  
	\end{thm}  
	We would like to sketch the idea of the proof of \cref{main:unbounded hyperbolic} since the methods are new even if $X$ is projective (compared with \cite{CCE15}).
	\begin{proof}[Proof of \cref{main:unbounded hyperbolic} (sketch)]
		For simplicity, we assume that $G$ is geometrically connected. 	Let $\pi:\overline{\xsp}\to \overline{X}$ be the spectral covering  associated with $\varrho$ in \cref{prop:spectral}. By \cref{prop:spectral}, it is a finite Galois covering with the Galois group $H$ and satisfies the following properties:
		\begin{itemize}
			\item there exists forms $\{\eta_1,\ldots,\eta_\ell\}\subset H^0(\overline{\xsp},\pi^*\Omega_{\overline{X}}^1(\log D))$ such that $\{\eta_1,\ldots,\eta_\ell\}$ coincides with the mutivalued one-forms $\pi^*\{\omega_1,\ldots,\omega_\ell\}$ induced by the $\varrho$-equivariant pluriharmonic map $u$ with logarithmic energy at infinity constructed in \cref{thm:BDDMex}.
			\item 
			Let us denote by ${\rm Ram}(\pi)$  the ramification locus of $\pi:\overline{\xsp}\to \overline{X}$. Then we have
			\begin{align}\label{eq:ramification}
				{\rm Ram}(\pi)\subset \bigcup_{\eta_i\neq\eta_j}(\eta_i-\eta_j=0).
			\end{align} 
			\item  
			$\{\eta_1,\ldots,\eta_\ell\}$ is invariant under the Galois group $H$. 
		\end{itemize}
		\begin{claim}\label{claim:generically}
			The quasi-Albanese map $a:\xsp\to A$  satisfies $\dim \xsp=\dim a(\xsp)$.
		\end{claim} 
		Let us explain the proof of \cref{claim:generically}. Assume by contradiction that $\dim a(\xsp)<\dim \xsp$. Let $F$ be a connected component of a general fiber of $a$.  Then $\eta_i|_{F}\equiv 0$ for each $\eta_i$.  
		
		Let $\pi':\widetilde{\xsp}\to \xsp$ be the universal covering   and denote by $\tilde{\pi}:\widetilde{\xsp}\to \widetilde{X}$   the map between universal covering lifting $\pi:\xsp\to X$.  Denote by $\tau=\pi^*\varrho: \pi_1(\xsp)\to G(K)$.  Then $u\circ\tilde{\pi}:\widetilde{\xsp}\to \Delta(G)$ is  $\tau$-equivariant harmonic map with logarithmic energy at infinity by \cref{item:pullback}. Let $F'$ be a connected component of $\pi'^{-1}(F)$.   Since 	$\{\eta_1,\ldots,\eta_\ell\}$  is generically the $(1,0)$-part of complexified differentials of $u\circ \tilde{\pi}$, it follows that $u\circ\tilde{\pi}(F')$ is a point. This implies that $\tau({\rm Im}[\pi_1(F)\to \pi_1(\xsp)])$ fixes a point in $\Delta(G)$, hence is bounded. 
		
		Note that $\tau({\rm Im}[\pi_1(F)\to \pi_1(\xsp)])$  is a normal subgroup of $\tau(\pi_1(\xsp))$.  Since $\tau(\pi_1(\xsp))$ is unbounded as $\varrho$ is unbounded, by \cref{lem:BT2}, we conclude that $\tau({\rm Im}[\pi_1(F)\to \pi_1(\xsp)])$  is finite. 
		
		Since we assume that $\varrho$ is big, $\tau$ is also big. We obtain a contradiction. Hence $\dim a(\xsp)=\dim \xsp$. 
		
		Therefore, the logarithmic Kodaira dimension $\bar{\kappa}(\xsp)\geq 0$. Assume that it is not maximal, then  the logarithmic Iitaka fibration
		$j:\xsp\to J$ has general fibers positive dimensional. Let $F$ be a general fiber  of $j$. Then $a|_F:F\to A$ is generically finite into the image and $\bar{\kappa}(F)=0$. By the criterion of abelian variety in \cite[Lemma 1.4]{CDY25}, we conclude that $\pi_1(F)$ is abelian. 
		
		Note that the Zariski closure of $\tau(\pi_1(\xsp))$ is also almost simple. 	Since \[\tau({\rm Im}[\pi_1(F)\to \pi_1(\xsp)])\]  is an abelian and normal subgroup of $\tau(\pi_1(\xsp))$ and $\varrho$ is Zariski dense, we conclude that the  Zariski closure of $\tau({\rm Im}[\pi_1(F)\to \pi_1(\xsp)])$, denoted by $N$, is  a normal subgroup of $G$. Since $G$ is almost simple, it follows that $N$ is finite. Hence $\tau({\rm Im}[\pi_1(F)\to \pi_1(\xsp)])$  is finite, contradicting with the assumption that $\varrho$ is big. Therefore, $j$ is birational, and we conclude that $\xsp$ is of log general type. 
		
		We will spread the positivity from $\xsp$ to $X$ to show that $X$ is of log general type. This step is innovative and as far as I know, the method has never appeared before.   
		Define a section
		$$
		\sigma:=\prod_{h\in H}\prod_{\eta_i\neq\eta_j}h^*(\eta_i-\eta_j)\in H^0(\overline{\xsp}, \Sym^{M}\pi^*\Omega_{\overline{X}}(\log D)), 
		$$
		which is non-zero. By \eqref{eq:ramification}, $\sigma$  vanishes at  ${\rm Ram}(\pi)$. 
		Since it is invariant under the $H$-action, it descends to a section
		$$
		\sigma^{H}\in H^0(\overline{X}, \Sym^{M}\Omega_{\overline{X}}(\log D))
		$$
		so that $\pi^*\sigma^{H}=\sigma$.   Let $R\subset X$ be the ramification locus of $\pi:\overline{\xsp}\to \overline{X}$. By the purity we know that $R$ is a divisor on $X$. Note that $\sigma^H$ vanishes at $R$. Therefore, it induces a non-trivial morphism
		\begin{align}\label{eq:symmetric}
			\cO_{\overline{X}}(R)\to  \Sym^{M}\Omega_{\overline{X}}(\log D).
		\end{align} 
		Since $\xsp$ is of log general type, and $\pi$ is unramified over $X-R$,  it follows that $K_{\overline{X}}+D+\overline{R}$ is big.   \eqref{eq:symmetric} together with  a theorem of Campana-P\u{a}un  in  \cref{thm:CP} below, implies that $K_{\overline{X}}+D$ is big. Therefore, $X$ is of log general type.
		
		\medspace
		
		Let us prove that $X$ is pseudo Picard hyperbolic. 
		Let $g:\bD^*\to X$ be non-constant holomorphic map that is not contained in ${\rm Ram}(\pi)$. Then there exists a Riemann surface $Y$, a proper surjective holomorphic map $p:Y\to\bD^*$ and a holomorphic map $f:Y\to \xsp$ such that we have the following commutative diagram
		\begin{equation*}
			\begin{tikzcd}
				Y\arrow[r, "f"] \arrow[d, "p"]& \xsp\arrow[d, "\pi"]\\
				\bD^*\arrow[r, "g"] &X
			\end{tikzcd}
		\end{equation*}
		A crucial fact is the estimation of  the ramification counting function of $p:Y\to \bD^*$ in \cite[Lemma 4.11]{CDY22} together with \cite[Lemma 11.2]{CDY25}.
		\begin{claim}\label{claim:small ram}
			There exists a proper Zariski closed subset $\Xi_1\subsetneq X$ such that if  $g(\bD^*)\not\subset \Xi_1$, then we have
			$$
			N_{{\rm ram}\,  \pi}(r)=O(\log r)+o(T_f(r)). 
			$$
		\end{claim}
		Recall that $\xsp$ is of log general type and the quasi-Albanese map $a:\xsp\to A$ satisfies that $\dim \xsp=\dim a(\xsp)$. Therefore, we apply \cref{thm2nd} to conclude that there exists a  proper Zariski closed subset $\Xi_2\subsetneq X$  such that 
		$g$ does not have essential singularity at the origin provided that $g(\bD^*)\not\subset \Xi_1\cup\Xi_2$.  This proves that $X$ is pseudo Picard hyperbolic. 
	\end{proof}  
	
	\begin{thm}[{\cite[Corollary 8.7]{CP19}}]\label{thm:CP}
		Let $\overline{X}$ be a smooth projective variety and let $D$ be a simple normal crossing divisor on $\overline{X}$. Let $L$ be a line bundle on $\overline{X}$, which admits a morphism $L \rightarrow  \bigotimes^m \Omega^1_{\overline{X}}(\log D)$ for some $m>0$, and such that the $\mathbb{Q}$-bundle $\varepsilon\left(K_{\overline{X}}+D\right)+L$ is big for some rational number $\varepsilon \geq 0$. Then $K_{\overline{X}}+D$ is big. \qed
	\end{thm}
	
	\subsection{Proof of \cref{main:hyper} when ${\rm char}\, K=0$ (sketch)} 
	We can assume that $K=\bC$.  For simplicity, we assume that the Zariski closure $G$ of $\varrho(\pi_1(X))$  is \emph{almost simple}. There are several cases that occurs. 
	
	\noindent {\bf Case 1. $\varrho$ is rigid.}   It means that for any continuous deformation $\varrho_t$ of $\varrho$, we have $[\varrho_t]=[\varrho]$, where $[\varrho]$ denotes the image of $\varrho$ in the Betti moduli space $M_{\rm B}(X,N)(\bC):=M_{\rm B}(\pi_1(X),\GL_{N})(\bC)$.  By Mochizuki's extension of \cref{cor:rigid} to the quasi-projective setting \cite{Moc06}, $\varrho$ underlies a $\bC$-variation of Hodge structure (cf. also \cite[\S 6]{CDY22} for a more self-contained proof). Moreover, after replacing $\varrho$ by a suitable conjugate, we may assume there exists  a number field $k\subset \overline{\bQ}$ such that 
	\begin{itemize}
		\item $G$ is defined over $k$;
		\item we have the factorization $\varrho:\pi_1(X)\to G(k)$;
		\item $\varrho(\pi_1(X))$ is Zariski dense in $G$.
	\end{itemize}
	
	\noindent {\bf Case 1.1.}   Assume that for each non-archimedean place $v$ of $k$, the composite $\varrho_v:\pi_1(X)\to \GL_{N}(k_v)$ of $\varrho$ and $k\hookrightarrow k_v$, is bounded. Here $k_v$ denotes the non-archimedean completion of $k$ with respect to $v$. 
	
	If this case occurs, we have a factorization $\varrho:\pi_1(X)\to \GL_{N}(\cO_k)$. Let us denote by  ${\rm Ar}(k)$ the set of archimedean places of $k$.  Note that ${\rm GL}_N(\cO_k)\to \prod_{w\in {\rm Ar}(k)}{\rm GL}_N(\bC)$ is a discrete subgroup  by \cite[Proposition 6.1.3]{Zim}. We denote by $\varrho_w:\pi_1(X)\to \GL_{N}(\bC)$ the composite of $\varrho$ and $w:k\hookrightarrow \bC$.  Then $\varrho_w$ is also rigid and thus underlies a $\bC$-VHS.   It follows that  for the product representation
	$$
	\prod_{w\in {\rm Ar}(k)}\varrho_w:\pi_1(X)\to \prod_{w\in {\rm Ar}(k)}{\rm GL}_N(\bC),
	$$
	its image $\Gamma$ is discrete. 
	
	Let $\sD$ be the period domain associated with the $\bC$-VHS of
	$\sigma := \prod_{w\in {\rm Ar}(k)} \varrho_w$.
	Since $\Gamma$ acts discretely on $\sD$, the quotient $\sD/\Gamma$ is a complex space.
	Let $p\colon X \to \sD/\Gamma$ be the period map.
	As we assume that $\varrho$ is big, the representation $\sigma$ is also big.
	By \cref{lem:period}, we have $\dim X = \dim p(X)$.
	Applying \cref{thm:PicardVHS}, we conclude that $X$ is pseudo-Picard hyperbolic, and  is strongly of log general type by \cite{BC20,CD21}.
	
	\medspace

	\noindent\textbf{Case 1.2.}
	Assume that there exists a non-archimedean place $v$ of $k$ such that the composite
	$\varrho_v \colon \pi_1(X) \to G(k_v)$, obtained from $\varrho$ via the embedding $k \hookrightarrow k_v$, is unbounded.
	Note that $\varrho_v(\pi_1(X))$ is Zariski dense in $G$.
	Since $\varrho$ is big, the representation $\varrho_v$ is also big.
	Therefore, the assumptions of \cref{main:unbounded hyperbolic} are satisfied, and the theorem follows.
	
	\noindent {\bf Case 2: $\varrho$ is non-rigid.} In previous work such as  \cite{CS08,Eys04}, the authors constructed unbounded representations using curves in character varieties in positive characteristic (after taking reduction mod $p$).  Note that this is quite natural in positive characteristic representation as we have seen in \cref{thm:Shapositive}.   However, once we made some reduction mod $p$ arguments,  these unbounded representations might not be Zariski dense in $G$ nor big (hence we cannot apply \cref{main:unbounded hyperbolic}).    In \cite{CDY22} we introduce a completely new method to construct unbounded representations and avoid reduction mod $p$. 
	\begin{claim}[{\cite[Proposition 6.1]{CDY22}}]\label{claim:unbounded}
		If $\varrho:\pi_1(X)\to G(\bC)$ is non-rigid, then there exists a  big, Zariski dense, and unbounded representation $\varrho':\pi_1(X)\to G(K)$, where $K$ is a finite extension of some $\bQ_p$ with $p$ prime.  
	\end{claim}
	The idea of the proof of \cref{claim:unbounded} is roughly that, for the set of bounded representations $R$ in the representation variety $R_{\rm B}(\pi_1(X),G)(K)$, its image in the character variety $M_{\rm B}(\pi_1(X),G)(K)$ is compact. Since $\varrho$ is non-rigid and $M_{\rm B}(\pi_1(X),G)$ is affine, the geometric connected component of  $M_{\rm B}(\pi_1(X),G)$ containing $\varrho$ is non-compact. Hence there exists some unbounded representation. Moreover, since Zariski density of a representation into an almost simple algebraic group is a Zariski open condition, we may assume that such an unbounded representation is Zariski dense.
	To ensure that it is big, some additional work is required; we refer the reader to \cite[Proposition~6.1]{CDY22} for further details.
	
	Since $\varrho$ is non-rigid, by \cref{claim:unbounded} we can construct a  big, Zariski dense, and unbounded representation $\varrho':\pi_1(X)\to G(K)$, where $K$ is some non-archimedean local field. We then apply \cref{main:unbounded hyperbolic}  to conclude the theorem.

	\subsection{On the generalized Green-Griffiths-Lang conjecture}\label{subsec:20230427}
	
	\begin{proof}[Proof of \cref{thm:GGL} (sketch)]
		
		\noindent \textbf{Case 1: ${\rm char}\, K=0$}.  
		We may assume that $K=\bC$. Let $G$ be the Zariski closure of $\varrho$, which is a complex reductive group as we assume that $\varrho$ is reductive. We may assume that $G$ is connected after we replace $X$ by a finite \'etale cover. Let $\cD G$ be the derived group of $G$ and let $R(G)$ be the radical of $G$. Define $G_1:=G/R(G)$ which is semisimple and $G_2:=G/\cD G$ which is a torus.  Then $G\to G_1\times G_2$ is  an isogeny. 
		
		Consider the representation $\sigma:\pi_1(X)\to G_1(\bC)$ by composing $\varrho$ with the quotient $G\to G_1$. Then $\sigma$ is Zariski dense.  One can show that,  after replacing $X$ by a finite \'etale cover and a birational proper modification, there exists a dominant morphism $f:X\to Y$ with connected general fibers, and  a big and Zariski dense representation \(\tau : \pi_{1}(Y) \to G(K)\) such that  $f^*\tau=\sigma$ (cf. \cite[Proposition 2.5]{CDY22}).     Therefore, by \cref{main:hyper}, we conclude that $Y$ is pseudo Picard hyperbolic and strongly of log general type, if it is not a point.  
		
		Consider the morphism $(f,{\rm alb_X}):X\to Y\times A$, where ${\rm alb}_X:X\to A$ denotes the quasi-Albanese map of $X$.  Since $\varrho$ is big, we can show that $g:=(f,{\rm alb_X})$  is generically finite into its image. Hence we apply \cref{cor:GGL} to conclude \cref{thm:GGL}.
		
		\iffalse
		The following factorization was proved in \cite{CDY22} in the case where ${\rm char}\, K=0$ and in general in \cite[Proposition 5.8]{DY23b}. It will be used throughout this memoir. 
		\begin{proposition}[\cite{CDY22,DY23b}]\label{lem:kollar}
			Let $X$ be a quasi-projective normal variety.  Let $\varrho:\pi_1(X)\to G(K)$ be a representation, where $G$ is a linear algebraic group defined over any field $K$.  Then there is a diagram
			\[
			\begin{tikzcd}
				\widetilde{X} \arrow[r, "\mu"] \arrow[d, "f"] & \widehat{X} \arrow[r, "\nu"] & X\\
				Y                                             &  &
			\end{tikzcd}
			\]
			where \(Y\) and \(\widetilde{X}\) are quasi-projective manifolds, and
			\begin{enumerate}[label=(\alph*)]
				\item $\nu:\widehat{X}\to X$ is a finite étale cover;
				\item \(\mu : \widetilde{X} \to \widehat{X}\) is a birational proper morphism;
				\item $f : \widetilde{X} \to Y$ is a dominant morphism with connected general fibers;
			\end{enumerate}
			such that there exists a big representation \(\tau : \pi_{1}(Y) \to G(K)\) with  $f^*\tau=(\nu\circ \mu)^*\varrho$.  
		\end{proposition}
		\fi
		
		\medspace

		\noindent \textbf{Case 2: ${\rm char}\, K=p>0$}.  
		
		Let $\overline{X}$ be a smooth projective compactification of $X$ such that
		$D := \overline{X} \setminus X$ is a simple normal crossing divisor.
		By the same arguments as in \cref{thm:Shapositive}, together with the assumption
		that $\varrho$ is big, we can show that there exist unbounded representations
		\[
		\{\tau_i \colon \pi_1(X) \to \GL_N(K_i)\}_{i=1,\ldots,k},
		\]
		where each $K_i$ is a finite extension of $ {\bF_{q_i}((t))}$ with
		$q_i = p^{n_i}$ for some $n_i \in \bN$, such that, for the Katzarkov--Eyssidieux
		reduction maps
		$s_{\tau_i} \colon X \to S_{\tau_i}$ associated with $\tau_i$, the product map
		\[
		(s_{\tau_1}, \ldots, s_{\tau_k}) \colon X \longrightarrow
		S_{\tau_1} \times \cdots \times S_{\tau_k}
		\]
		is generically finite onto its image.

		By \cref{prop:spectral}, there exists a spectral covering $\pi:\xsp\to X$ of Galois group $H$ such that 
		\begin{enumerate}[label={\rm (\alph*)}]
			\item \label{itea} there exists (spectral) forms $\{\eta_1,\ldots,\eta_m\} \subset H^0(\overline{\xsp}, \pi^*\Omega_{\overline{X}}(\log D))$ associated with $\tau_1,\ldots,\tau_k$, which are invariant under $H$;
			\item  \label{iteb}$\pi$ is \'etale outside  
			\begin{align}\label{eq:rami}
				R:=\{x\in \overline{\xsp} \mid \exists \eta_{i}\neq\eta_j  \mbox{ with }(\eta_i-\eta_j)(x)=0\} 
			\end{align}
			\item \label{itec}There exists a morphism $a:\xsp\to A$ to a semi-abelian variety $A$ with $H$ acting on $A$  such that $a$ is $H$-equivariant.
			\item \label{ited} The quasi-Stein factorization of the quotient $X\to A/H$ of $a$ by $H$, coincides with the   quasi-Stein factorization of $(s_{\tau_1},\ldots, s_{\tau_k})$.
		\end{enumerate}    Therefore, we have   $\dim \xsp=\dim a(\xsp)$.  
		
		Assume that $X$ is of log general type. 
		We will use   notions of Nevanlinna theory   in \cref{sec:second}.   
		For  any holomorphic map $f:\bC_{>\delta}\to X$ whose image is not contained in $\pi(R)$, there exists  a proper surjective holomorphic map   $p:Y\to \bC_{>\delta}$  from a connected Riemann surface $Y$ to  $\bC_{>\delta}$ and  a  holomorphic map $g:Y\to \xsp$  satisfying the following diagram:
		\begin{equation}\label{figure:curve}
			\begin{tikzcd}
				Y\arrow[r, "g"] \arrow[d, "p"] & \xsp\arrow[d, "\pi"]\\
				\bC_{>\delta}\arrow[r, "f"] & X
			\end{tikzcd}
		\end{equation}By \cref{claim:small ram},  there exists a proper Zariski closed subset $\Xi\subsetneq X$ such that for  any holomorphic map $f:\bC_{>\delta}\to X$ whose image not contained in $\Xi$,  
		one has  
		\begin{align*} 
			N_{{\rm ram}\, p}(r) =o( T_{g}(r,L)) + O(\log r) ||,
		\end{align*}  
		where $L$ is an ample line bundle on $\overline{\xsp}$  and $T_g(r,L)$ is the Nevanlinna order function.  
		Note that $\xsp$ of log general type as we assume that $X$ is of log general type and $\pi:\xsp\to X$ is a Galois cover.  We apply \cref{thm2nd} to conclude that $f$ has no essential singularity at the origin,
		which implies that $X$ is pseudo Picard hyperbolic.

		\medspace
		
		Assume that $X$ is pseudo Brody hyperbolic. Then $\xsp$ is also pseudo Brody hyperbolic, and by applying \cref{cor:GGL} with $S$ being a point, we conclude that $\xsp$ is of log general type. We then use exactly the same arguments as in the proof of \cref{main:unbounded hyperbolic} to spread the positivity of $\xsp$ to $X$, relying on \cref{thm:CP} to show that $X$ is of log general type. 
	\end{proof}
	
	\subsection{Proof of \cref{main:hyper} when ${\rm char}\, K=p>0$ (sketch)}  
	We will still maintain the same notations as introduced in the proof of \cref{thm:GGL}.	Let $\pi:\xsp\to X$ be the Galois covering  defined therein.  Consider the representation $\pi^*\varrho:\pi_1(\xsp)\to G(K)$, which is Zariski dense. By the proof of \cref{thm:GGL}, there exists a morphism $a:\xsp\to A$ where $A$ is a semiabelian variety such that $\dim \xsp=\dim a(\xsp)$. Hence we have $\bar{\kappa}(\xsp)\geq 0$.   By the same arguments as in the proof of \cref{main:unbounded hyperbolic}, one can show that 
	$\xsp$ is of log general type.   
	
	We now use  \Cref{claim:small ram} together with \cref{thm2nd}     to conclude that $X$ is pseudo Picard hyperbolic.

	\section{Topology of algebraic varieties in the presence of a big local system}\label{sec:topology2}
 In this section, we outline the proofs of the theorems stated in \cref{sec:topology}, using the techniques in \cref{sec:sha}, following a beautiful strategy initiated by Arapura--Wang~\cite{AW25}.

	\subsection{Proof of linear Chern-Hopf-Thurston conjecture}
	In this subsection, we skech the idea of the proof of \cref{thm:CHT} when $\varrho$ is semisimple and large. 
	
	For any perverse sheaf $\cP$, its characteristic cycle is 
	$$
	CC(\cP)=\sum_{i=1}^{m}n_i T_{Z_i}^*X,
	$$ 
	where $n_i\in \bN$, $Z_i$ is an irreducible subvariety of $X$, and \(T^{*}_{Z_i}X\) denotes the conormal bundle of \(Z_i\) in the cotangent bundle \(T^{*}X\).
	Each \(T^{*}_{Z_i}X\) is a conic Lagrangian cycle of \(T^{*}X\).
	The crucial idea, initiated by Arapura and Wang \cite{AW25}, is the following formula:
	\begin{align}
		\chi(X,\cP)=CC(\cP)\cdot  T_X^*X,
	\end{align}
	where \(T_X^*X\) denotes the conormal bundle of \(X\), namely the zero section of the cotangent 
	bundle \(T^*X \to X\). 
	Therefore, in order to prove \cref{thm:CHT}, it suffices to show that
	\begin{thm}\label{thm:CHT2}
Let $X$ be a smooth projective variety and let $\varrho:\pi_1(X)\to \GL_{N}(K)$ be a large representation, for any field $K$. Then for any closed subvariety $Z$ of $X$, we always have  		 	\[
		 T_Z^*X \cdot T_X^*X \ge 0
		 \]
		 for any closed subvariety \(Z \subset X\).
	\end{thm} 
	Although the cotangent bundle \(T^*X\) is non-compact, the above intersection number is 
	well defined since the zero section \(T_X^*X\) is compact. 
	
	We begin with a preliminary observation regarding the strategy employed in the proof of \cref{thm:CHT}. 
\begin{lem}\label{lem:baby}
	Let $X$ be a smooth projective $n$-fold. If there exists a holomorphic 1-form $\eta$ on $X$ such that its zero locus $Z(\eta) := (\eta=0)$ is zero-dimensional, then
	$$
	(-1)^n\chi(X)   \geq 0.
	$$
\end{lem} 
\begin{proof}
	Consider the graph $\Gamma$ of the section $\eta$ of $T^*X \to X$, which is a closed subvariety of the total space $T^*X$. Let $T_X^*X$ denote the zero section of $T^*X$. The graph $\Gamma$ is homologous to $T_X^*X$. Hence, computing the self-intersection of the zero section, we have
	$$
	(-1)^n\chi(X) = T_X^*X \cdot T_X^*X = \Gamma \cdot T_X^*X.
	$$
	Since $Z(\eta)$ is zero-dimensional, the intersection of $\Gamma$ and $T_X^*X$ is proper.  Therefore, we have
	 \[
	\Gamma \cdot T_X^*X \geq 0.
	\]
\end{proof}
However, in general, we cannot guarantee the existence of holomorphic 1-forms on $X$  in \cref{thm:CHT2}. Moreover, even if such holomorphic 1-forms exist, their zero loci need not consist of isolated points. Consequently, the cycles in $T^*X$ used to compute the intersection—specifically the graph of the   1-form and the zero section—may not intersect properly.

To address the first issue, we use  a \emph{multivalued 1-form} (see \cref{def:multivaluedform}) instead and adapt the intersection theory from \cref{lem:baby}. To resolve the second issue, we construct a procedure that deforms the intersection progressively to achieve properness.

More precisely, using techniques analogous to deformation to the normal cone in intersection theory, one can associate to a $d$-valued $1$-form $\eta$ a map $\Phi_\eta$ from conic Lagrangian cycles to conic Lagrangian cycles such that, for any subvariety $Z \subset X$, one has
\begin{align}\label{eq:vanishing cycle}
	\Phi_\eta(T_Z^*X)
	= n_0\, T_Z^*X + \sum_{1 \le i \le m} n_i\, T_{Z_i}^*X,
\end{align}
where \(n_0\) is the multiplicity of the zero form in the restriction 
\(\eta|_{Z_{\mathrm{reg}}}\), and each \(Z_i\) is a proper closed subvariety of \(Z\).
In particular, if \(\eta|_{Z_{\mathrm{reg}}}\) is nontrivial, then \(n_0 < d\).
	Moreover, we have
	$$
	\Phi_\eta(T_Z^*X)\cdot T_X^*X=d T_Z^*X\cdot T_X^*X.
	$$
	Therefore,
	if $\eta|_Z\not\equiv 0$, we have
	\begin{align}\label{eq:iteration}
		(d-n_0) T_Z^*X\cdot T_X^*X=\sum_{1 \leq i \leq m} n_i T_{Z_i}^* X\cdot T_X^*X,
	\end{align}
 where $Z_i \subsetneq Z$ is a proper subvariety of $Z$. Roughly speaking, we have
	\[
	\dim Z = \dim(T_Z^*X \cap T_X^*X) > \dim Z_i = \dim(T_{Z_i}^*X \cap T_X^*X),
	\]
	implying that the intersection becomes more ``proper''. We omit the precise definition of the map $\Phi_{\eta}$ here and instead refer the interested reader to \cite[\S 3.1]{DW24} for further details.

	\begin{proof}[Proof of \cref{thm:CHT2} (sketch)]
		\noindent {\rm Case 1: }  ${\rm char}\, K=p>0. $
		
		In \cref{thm:Shapositive},  it is proved  that  for any linear $\sigma:\pi_1(X)\to \GL_N(K)$, then there are representations $\tau_i:\pi_1(X)\to \GL_N(K_i)$,  where $K_i$ is non-archimedean local field of characteristic $p$ (i.e. $K_i$ is a   finite extension of $\bF_{p}((t))$), such that the Shafarevich morphism of $\sigma$ is the Stein factorization of $$s_{\tau_1}\times\cdots\times s_{\tau_k}:X\to S_{\tau_1}\times\cdots\times S_{\tau_k}.$$ 
		
		Therefore, when $\varrho$ is large, its Shafarevich morphism is just the identity map, and thus for any $Z$, there exists some $\tau_i$ such that 
		$
		s_{\tau_i}(Z)
		$ is not a point. By the property of the Katzarkov-Eyssidieux reduction map in \cref{thm:KE}, there is multivalued 1-form $\eta_{\tau_i}|_{Z}\not\equiv 0$. 
		
		Now we perform the iteration of the above algorithm, to achieve that $\dim Z_i=0$ for each $i$ in right-hand side of \eqref{eq:iteration}. The intersection of $T_{Z_i}^*X$ and $T_X^*X$ is proper, and we have 
		$$
		T_Z^*X\cdot T_X^*X\geq 0
		$$
		for each $Z$. 
		
		\medspace
		
		\noindent {\rm Case 2: }  $K=\bC $  and $\varrho$ is semisimple. 
		
	%	\begin{claim}
		%	There exist  reductive representations $\{\tau_i:\pi_1(X)\to \GL_{N}(K_i)\}_{i=1,\ldots,m}$, where each  $K_i$ is a non-archimedean local field of characteristic zero,  such that  for any closed subvariety $Z$ of $X$, if $s_{\rm fac}(Z)$ is not a point, then  $\eta_{\tau_i}|_{Z}\not\equiv 0$ for some $i$. Here $\eta_{\tau_i}$ is the multivalued 1-form on $X$ induced by $\tau_i$. 
	%	\end{claim} 
	Consider the map $s_{\rm fac}:X\to S_{\rm Fac}(X)$ defined in \cref{def:fac}. 		By \cref{lem:simultaneous}, there exist  reductive representations $\{\tau_i:\pi_1(X)\to \GL_{N}(K_i)\}_{i=1,\ldots,m}$, where each  $K_i$ is a non-archimedean local field of characteristic zero,  such that $s_{\rm fac}:X\to S_{\rm fac}$ is the Stein factorization of 
			$$
			s_{\tau_1}\times \cdots \times s_{\tau_m}:X\to 	S_{\tau_1}\times \cdots \times S_{\tau_m}.
			$$
		Here $s_{\tau_i}:X\to S_{\tau_i}$ is the Katzarkov-Eyssidieux reduction map associated with $\tau_i$.

	Fix a closed subvariety $Z \subseteq X$. Assume that $s_{\rm fac}(Z)$ is not a point. Then $s_{\tau_i}(Z)$ is not a point for some $i$. By \cref{thm:KE}, this implies that $\eta_{\tau_i}|_{Z} \not\equiv 0$.
	
	We employ the same strategy, using the maps $\Phi_{\eta_{\tau_1}}, \ldots, \Phi_{\eta_{\tau_m}}$ defined above, to make the intersection of $T_X^*X$ with $T_Z^*X$ sufficiently proper. 
	In this situation, when the algorithm introduced above terminates, it means that each subvariety $Z_i$ appearing on the right-hand side of \eqref{eq:iteration} is contained in some fiber of $S_{\rm fac}$. 
	Unlike the case of positive characteristic, one may have $\dim Z_i > 0$.
	
	We now sketch the arguments of \cite{DW24} to treat this case; the following discussion is intended as an outline rather than a rigorous proof.

	Let $\mathcal{L}$ be the $\mathbb{C}$-VHS on $X$ constructed in \cref{prop:cons}. For simplicity, we assume that $Z_i$ is smooth. By \cref{prop:cons}, the restriction $\varrho|_{\pi_1(Z_i)}$ corresponds to a direct summand of the local system underlying $\mathcal{L}$, which has discrete monodromy. If $\varrho$ is large and semisimple, the period map of $\mathcal{L}|_{Z_i}$ is finite by the arguments in \cref{lem:period}. Consequently, certain curvature properties of the period domain imply that $T^*{Z_i}$ is ``almost'' nef. As observed in \cite{AW25}, a theorem of Demailly--Peternell--Schneider \cite{DPS94} then implies that $T_{Z_i}^*X \cdot T_X^*X \geq 0$. This concludes the proof for this case.
		
		\medspace
		
		\noindent {\rm Case 3: }  Suppose $K=\mathbb{C}$ and $\varrho$ is linear. The proof relies on the \emph{tautological} variation of mixed Hodge structures introduced in \cite{ES11}, combined with arguments used for the linear Shafarevich conjecture in \cite{EKPR12}. Due to the technical complexity of this construction, we omit the details here.
	\end{proof}	 
	\subsection{Proof of the linear Koll\'ar conjecture}
	In this subsection, we sketch the proof of \cref{item:vanishing}  assuming that $\varrho$ is semisimple and big.
	
		As we see before, other two items in \cref{main:kollar} follow from \cref{item:vanishing}.  So it suffices to prove \cref{item:vanishing}.   
		By the byproducts in the proof of the  reductive Shafarevich conjecture  \cite{Eys04,DY23}, as we remarked in \cref{rem:byproduct}, there exists 
		\begin{enumerate}
			\item a family of Zariski dense representations \( \{\tau_i: \pi_1(X) \to G_i(K_i)\}_{i=1,\ldots,\ell} \), where each \( G_i \) is a reductive group over a non-archimedean local field \( K_i \) of characteristic zero;
			\item a \( \bC \)-VHS \( \cL \) on \( X \),
		\end{enumerate}
		such that
		\[
		\Phi:=T_{\tau_1} + \cdots + T_{\tau_\ell} + \sqrt{-1} \, \mathrm{tr}(\theta \wedge \theta^*)
		\]
		is a closed positive \( (1,1) \)-current on \( X \), which is smooth and strictly positive over a non-empty analytic open subset \( X^\circ \subset X \).
		
		Here:
		\begin{itemize}
			\item \( T_{\tau_i} \) denotes the canonical current on \( X \) associated with \( \tau_i \), as defined in \cref{def:canonical};
			\item \( \theta \) is the Higgs field of the Hodge bundle associated with \( \cL \), and \( \theta^* \) is its adjoint with respect to the Hodge metric. \qed
		\end{itemize}
		Let $\beta_{\tau_i}$ be the   1-form on $\widetilde{X}$ associated with $\tau_i$   defined in \eqref{eq:1-form}, and $\beta_\cL$ be the    1-form on $\widetilde{X}$ associated with $\cL$   defined in \eqref{lem:potential2}. 
		Define 
		$\beta:=\beta_\cL+\sum_{i=1}^{\ell}\beta_{\tau_i}$.  Then $\beta$ has $L_{\rm loc}^1$-coefficients. By \cref{lem:potential,lem:potential2}, it satisfies 
		\begin{align}\label{eq:dsub}
			|\beta(x)|\leq_{\rm a.e.} C(1+d_{\widetilde{X}}(x, x_0)) 
		\end{align} 
		for some constant $C>0$, and 
		\begin{align}\label{eq:mutual}
			d\beta\geq   \pi_X^*(\sum_{i=1}^{\ell}T_{\tau_i}+\omega_{\cL} )=:\pi_X^*\Phi.
		\end{align}   
		
	Assume, for contradiction, that there exists a holomorphic \((p,0)\)-form \(\alpha\) on 
	\(\widetilde{X}\) which is \(L^{2}\) with respect to \(\omega\), for some \(0\le p\le n-1\).
	A crucial step in \cite[Theorem~2.2]{DW24b} is that the sublinear growth condition 
	\eqref{eq:dsub} implies that a suitable ``Stokes formula'' holds, namely 
		\begin{align*} 
			\int_{\widetilde{X} }  	\sqrt{-1}^{p^2}	 d\beta \wedge \alpha\wedge \bar{\alpha}\wedge \omega^{n-p-1}=0.
		\end{align*}   
	Since \((\sqrt{-1})^{p^{2}}\,\alpha\wedge \bar{\alpha}\) is a semipositive \((p,p)\)-form, it follows 
	from \eqref{eq:mutual} that
		\begin{align}\label{eq:pointwise}
			0=	\int_{\widetilde{X} }  	\sqrt{-1}^{p^2}	 d\beta \wedge \alpha\wedge \bar{\alpha}\wedge \omega^{n-p-1}	\geq \int_{\widetilde{X}}  	\sqrt{-1}^{p^2}	  \alpha\wedge \bar{\alpha} \wedge \omega^{n-p-1}\wedge  \pi_X^*\Phi\geq 0.
		\end{align}
		Since $\Phi$ is smooth and strictly positive over a non-empty analytic open subset \( X^\circ \subset X \), one can show that 
		$
		\alpha(x)=0
		$ 
		for any $x\in \pi_X^{-1}(X^\circ)$.  Since $\alpha$ is holomorphic, it follows that  $\alpha\equiv 0$.  This proves that $H^{(p,0)}_{(2)}(\widetilde{X})=0$ for $p<n$. By the $L^2$-Lefschetz theorem \cite{Gro92}, the map
		$$
		H^{(p,0)}_{(2)}(\widetilde{X})\stackrel{ \pi_X^*\omega^{n-p}\wedge }{\to} H^{(n,n-p)}_{(2)}(\widetilde{X})
		$$
		is isomorphic. Hence $H^{(n,q)}_{(2)}(\widetilde{X})=0$ for $q>0$. 	
		The theorem is proved when $\varrho$ is semisimple and big. 
		
		When the representation $\varrho$ is linear and big in the sense of \cite{DW24b}, the proof of the theorem relies on results regarding the linear  {Shafarevich conjecture} established in \cite{EKPR12}. Additionally, we must establish a more involved  {vanishing theorem} than the one presented above (cf. \cite[Theorem 2.2]{DW24b}). 
	\iffalse
	\subsection{Some histories on Hopf conjecture}
	A major breakthrough was achieved by Gromov. In \cite{Gro92}, he observed that if a compact Kähler $n$-fold \((X,\omega)\) has negative sectional curvature, then its Kähler form is \(d\)-bounded. More precisely, on the universal covering 
	\(\pi_X : \widetilde{X} \to X\) there exists a smooth \(1\)-form \(\beta\) such that
	\[
	d\beta = \pi_X^{*}\omega
	\qquad\text{and}\qquad
	|\beta(x)|_{\pi_X^{*}\omega} \le C
	\quad\text{for all } x \in \widetilde{X}.
	\]
	He then applies $L^2$-Hodge theories to show that the $L^2$-cohomolgy vanishes $H^k_{(2)}(\widetilde{X})=0$ unless $k\neq n$.  Together Atiyah's $L^2$-index theorem and his non-vanishing theorem $H^n_{(2)}(\widetilde{X})\neq 0$ , he proved that 
	$$
	(-1)^n\chi(X)= \dim_{\Gamma} H^n_{(2)}(\widetilde{X})>0. 
	$$
	
	His work was later extended by Cao-Xavier and Jost-Zuo independently, in which they introduced a notion of K\"ahler parabolicity: 
	\begin{dfn}
		A differential form $\alpha$ on a complete non-compact Riemannian manifold is called $d$-sublinear if there exist a differential form $\beta$ and a number $c>0$ such that $d \beta=\alpha$ and $|\beta(x)| \leq c\left(1+\rho\left(x, x_0\right)\right)$, where $\rho\left(x, x_0\right)$ stands for the Riemannian distance between $x$ and a base point $x_0$.
	\end{dfn}
	
	then there exists a notion of K\"ahler hyperbolicity, which implies 
	\fi
	
	\subsection{Deformation of big fundamental groups}
Let $(X,\omega)$ be a compact K\"ahler manifold, and let
$\varrho:\pi_1(X)\to \GL_N(\bC)$ be a reductive representation.
By \cref{thm:Corlette}, there exists a harmonic bundle $(E,\theta,h)$
whose associated flat connection
$
\nabla_h+\theta+\theta^*
$ 
has monodromy representation $\varrho$.
Here $\nabla_h$ denotes the Chern connection of $(E,h)$, and
$\theta^*$ is the adjoint of $\theta$ with respect to $h$.
It is a standard exercise to check that the $(1,1)$-form
$
\sqrt{-1}\,\mathrm{tr}(\theta\wedge\theta^*)
$ 
is real, semipositive, and closed.
Moreover, it does not depend on the choice of the K\"ahler metric $\omega$.
	\begin{dfn}[Canonical form]\label{def:canonical form}
		Such $(1,1)$-form  $\omega_\varrho:=\sn{\rm tr}(\theta\wedge\theta^*)$ is called the \emph{canonical form} associated with $\varrho$.   
	\end{dfn}       
	In \cite{DMW24}, we first establish the deformation smoothness of such canonical forms. Our main result is as follows. 
\begin{thm}[{\cite[Remark~7.8]{DMW24}}]\label{harmonic2}
	Let $\sX$ be a K\"ahler manifold, and let
	$f:\sX\to \bD$ be a proper holomorphic fibration over the unit disk $\bD$.
	For $t\in\bD$, write $X_t:=f^{-1}(t)$.
	Let $\varrho:\pi_1(X_0)\to \GL_N(\bC)$ be a reductive representation, and define
	\begin{equation}\label{eq:t}
		\varrho_t:\pi_1(X_t)\stackrel{\simeq}{\to}\pi_1(\sX)\stackrel{\simeq}{\to}\pi_1(X_0)
		\stackrel{\varrho}{\longrightarrow}\GL_N(\bC).
	\end{equation}
	Then the fiberwise defined canonical $(1,1)$-form $\omega_{\varrho_t}$ on $X_t$,
	associated with the representation $\varrho_t$, varies smoothly with respect to
	$t\in\bD$.
\end{thm}
	Another main technical result, and also main difficult result, is the following consequence of  the non-archimedean analogue of \cref{harmonic2}. 
	\begin{thm}[{\cite[Lemma 5.2]{DMW24}}]\label{lem:continuity}
	Let $f:\sX\to \bD$ be a smooth projective family of relative dimension $n$
	over the unit disk $\bD$, and let
	$\varrho:\pi_1(X_0)\to G(K)$ be a Zariski-dense representation,
	where $G$ is a reductive algebraic group over a non-archimedean local field $K$.
	For $t\in\bD$, define
	\begin{equation}\label{eq:t2}
	 \varrho_t:\pi_1(X_t)\stackrel{\simeq}{\to}\pi_1(\sX)\stackrel{\simeq}{\to}\pi_1(X_0)
		\stackrel{\varrho}{\longrightarrow} G(K).
	\end{equation}
	There exists a full-measure open subset $X_0^\circ\subset X_0$ such that
	for any $x_0\in X_0^\circ$ there exists a coordinate system
	\[
	(\Omega; z_1,\ldots,z_n,t;\varphi)
	\]
	on $\sX$ centered at $x_0$, together with a real $(1,1)$-form
	\[
	T(z,t)=\sqrt{-1}\sum_{i,j} a_{ij}(z,t)\,dz_i\wedge d\bar z_j
	\]
	on $\bD^n$ (with $z=(z_1,\ldots,z_n)$) satisfying
		\begin{thmlist}
			\item the map $\varphi:\bD^n\times \bD_\ep\to \Omega$ is a biholomorphism, with $f\circ \varphi(z_1,\ldots,z_n,t)=t$.  
			\item the coefficents $a_{ij}(z,t)$ are continuous function on $\bD^n\times\bD_{\ep}$.
			\item For each fixed $t\in \bD_\ep$, $T_{t}(z):=T(z,t)$ is a smooth semi-positive closed $(1,1)$-form on $\bD^n$.
			\item For each $t\in \bD_\ep$, one has  $ T_{\varrho_t}|_{\Omega\cap X_t}\geq T_t$.
			\item  $T_{\varrho_{0}}|_{\Omega\cap X_0}=T_0$.
		\end{thmlist} 
		Here $T_{\varrho_t}$ is the canonical current 
		defined in \cref{def:canonical}, associated with  $\varrho_t$. 
 	\end{thm}
 	 	The proof of \cref{lem:continuity} relies heavily on the theory of harmonic
 	maps into Euclidean buildings (and more generally, NPC spaces) as developed by Gromov--Schoen \cite{GS92}, 
 	Korevaar--Schoen \cite{KS,KS2} and later in \cite{BDDM}. The argument is quite involved; we refer
 	the interested reader to \cite{DMW24} for a more detailed exposition.
 	
 	Let us explain the idea of the proof of \cref{conj:deformation} in the case where
 	$\varrho$ is reductive and big.  Set $X:=X_0:=f^{-1}(0)$. By \cref{rem:byproduct}, there exist reductive
 	representations
 	\[
 	\{\tau_i:\pi_1(X)\to \GL_{N}(K_i)\}_{i=1,\ldots,\ell},
 	\]
 	where each $K_i$ is a non-archimedean local field of characteristic zero, together
 	with a $\bC$-VHS $\sigma$, such that the sum
 	\[
 	\sum_{i=1}^{\ell} T_{\tau_i}+\omega_\sigma
 	\]
 	satisfies the following properties:
 	\begin{itemize}
 		\item it is a closed positive $(1,1)$-current on $X$ with continuous potential;
 		\item it is smooth on a Zariski open subset $X^\circ\subset X$;
 		\item its cohomology class is big and nef.
 	\end{itemize}
 	Here $T_{\tau_i}$ denotes the canonical current associated with $\tau_i$, and
 	$\omega_\sigma$ is the canonical form defined in \cref{def:canonical form}
 	associated with the monodromy representation
 	$\sigma:\pi_1(X)\to \GL_{N'}(\bC)$ of the $\bC$-VHS $\cL$.
 	By Boucksom’s criterion \cite{Bou02}, the current
\begin{align*} 
	 \sum_{i=1}^{\ell} T_{\tau_i}+\omega_\sigma
\end{align*} 
 	is smooth and strictly positive on some analytic open subset
 	$U\subset X^\circ$.
 	
 	By \cref{harmonic2,lem:continuity}, for $t$ sufficiently small, the sum
\begin{align}\label{eq:current}
 	\sum_{i=1}^{\ell} T_{\tau_{i,t}}+\omega_{\sigma_t}
\end{align}
 	is strictly positive on some analytic open subset of $X_t$.
 	Here $\tau_{i,t}:\pi_1(X_t)\to G(K_i)$ and
 	$\sigma_t:\pi_1(X_t)\to \GL_{N'}(\bC)$ are the representations induced by
 	$\tau_i$ and $\sigma$, respectively, as defined in
 	\eqref{eq:t} and \eqref{eq:t2}.
 	
 	On the other hand, for any subvariety $Z\subset X_t$, if
 $
 	\mathrm{Im}\bigl[\pi_1(Z)\to \pi_1(X_t)\bigr]
$ 
 	is finite, then the restrictions $T_{\tau_{i,t}}|_Z$ and
 	$\omega_{\sigma_t}|_Z$ are both trivial. This follows from the functoriality
 	of harmonic maps under pullback established in \cref{thm:BDDMex}, together
 	with the construction of canonical forms and canonical currents.
 	Consequently, the reductive Shafarevich morphism of $X_t$, whose existence is
 	guaranteed by \cref{thm:Shafarevich}, must be birational for $t$ sufficiently
 	small. This shows that $X_t$ has a big fundamental group for all sufficiently
 	small $t$.
 	
 	In the general case where $\varrho$ is linear and big, we apply the techniques
 	developed in \cite{EKPR12}, together with analytic results on the variation of
 	mixed Hodge structures, to prove \cref{conj:deformation}. We omit the details
 	here and refer the interested reader to \cite{DMW24}. 
 	 
 	\subsection{Applications to hyperbolicity}
I now give an application of \cref{conj:deformation} to the hyperbolicity of
algebraic varieties under deformation, combining
\cref{main:hyper,thm:GGL}. 
First, we recall the following classical result on the \emph{openness}
of Brody hyperbolicity. 
 	\begin{thm}[{\cite[Proposition 1.10]{Dem20}}]\label{thm:Brody}
 		Let $f:\sX\to \bD$ be a  holomorphic proper submersion from a complex manifold $\sX$ to  the unit disk with connected fibers.  If $X_0$ is  Brody hyperbolic, then there exists $\ep>0$ such that $X_t$ is Brody hyperbolic   for  $|t|<\ep$. 
 	\end{thm}
 	There has long been a folklore conjecture that such openness properties hold for pseudo Brody hyperbolicity.
 	\begin{conjecture}\label{conj:hyperbolic}
 		Let $f:\sX\to \bD$ be as in \cref{thm:Brody}. If $X_0$ is pseudo Brody   hyperbolic, then  $X_t$ is also pseudo  Brody hyperbolic for   sufficiently small $t$. 
 	\end{conjecture}
 This conjecture is indeed a consequence of \cref{conj:GGL}.
 Let us explain how to deduce \cref{conj:hyperbolic} from \cref{conj:GGL}. 
 Assume that $X_0$ is pseudo Brody hyperbolic.
 Then $X_0$ is of general type by \cref{conj:GGL}.
 By Siu’s invariance of plurigenera \cite{Siu98,Pau07}, it follows that
 $X_t$ is of general type for any $t \in \bD$.
 Applying \cref{conj:GGL} again, we conclude that $X_t$ is pseudo Brody
 hyperbolic.
 This proves \cref{conj:hyperbolic}.

 	\begin{thm}[{\cite[Theorem D]{DMW24}}]\label{main4}
 		Let $f:\sX\to \bD$ be a smooth projective family. Assume that there is a big and reductive representation $\varrho:\pi_1(X_0)\to \GL_N(\bC)$. If $X_0$ is pseudo Brody   hyperbolic, then  $X_t$ is  pseudo Picard hyperbolic for sufficiently small $t$.
 	\end{thm} 
 \begin{proof}
   We apply \cref{thm:GGL} to conclude that $X_0$ is of general type. By Siu's invariance of plurigenera \cite{Siu98,Pau07}, $X_t$ is of general type for any $t\in \bD$.

 By \cite[Lemma 3.25]{DY23}, there exists another reductive  representation $\tau:\pi_1(X_0)\to \GL_N(\bC)$ such that  for any  reductive  representation $\sigma:\pi_1(X_0)\to \GL_N(\bC)$, we have 
 	\begin{align}\label{eq:inclusion}
 		\ker\tau\subset\ker\sigma.
 	\end{align} Let $\tau_t:\pi_1(X_t)\to \GL_N(\bC)$  be the composite map of $\tau$ and the natural isomorphism $\pi_1(X_t)\to \pi_1(X_0)$ induced by $f$. 
 	\begin{claim}\label{claim:big}
 		There exists $\ep>0$ such that $\tau_t$ is a big  and reductive representation for any $t\in \bD_\ep$. 
 	\end{claim} 
 	\begin{proof} 
We use the same notation as in the above subsection.
By \cref{thm:Shafarevich}, the Shafarevich morphism
\[
\mathrm{sh}_{\tau_t} \colon X_t \to \mathrm{Sh}_{\tau_t}(X_t)
\]
exists.
It suffices to prove that this morphism is birational for $t$ sufficiently
close to $0$.

Let
\[
\{\tau_i \colon \pi_1(X) \to \GL_N(K_i)\}_{i=1,\ldots,\ell}
\quad \text{and} \quad
\sigma \colon \pi_1(X) \to \GL_{N'}(\bC)
\]
be the reductive representations appearing in the above proof of
\cref{conj:deformation}.
Let $Z$ be a fiber of $\mathrm{sh}_{\tau_t}$.
By \eqref{eq:inclusion}, we have
\[
\tau_{i,t}\bigl(\operatorname{Im}[\pi_1(Z^{\mathrm{norm}})
\to \pi_1(X_t)]\bigr)
\]
finite for each $i$.

Recall that in \cref{prop:cons}, the $\bC$-VHS $\sigma$ is a direct sum of
reductive representations of $\pi_1(X)$ into $\GL_N(\bC)$.
By \eqref{eq:inclusion}, it follows that
\[
\sigma_t\bigl(\operatorname{Im}[\pi_1(Z^{\mathrm{norm}})
\to \pi_1(X_t)]\bigr)
\]
is also finite.
Consequently, for the closed positive $(1,1)$-current
\[
\sum_{i=1}^{\ell} T_{\tau_i,t} + \omega_{\sigma,t}
\]
defined in \eqref{eq:current}, its restriction to $Z$ is trivial.

By the same arguments as in the proof of \cref{conj:deformation}, we conclude
that the morphism
$\mathrm{sh}_{\tau_t} \colon X_t \to \mathrm{Sh}_{\tau_t}(X_t)$
is birational for $t$ sufficiently small.
\end{proof}

Since $X_t$ is of general type, it follows from
\cref{claim:big,thm:GGL} that $X_t$ is pseudo Picard hyperbolic
for any $t \in \bD_\varepsilon$.  
 \end{proof}   
 	\begin{cor}[{\cite[Corollary E]{DMW24}}]\label{corx}
 	Let $X$ be a smooth projective variety.  Assume that either 
 	\begin{enumerate}[label=(\alph*)]
 		\item \label{item:VHS} there is  $\bC$-VHS for short $\varrho:\pi_1(X)\to \GL_N(\bC)$ with discrete monodromy $\Gamma$ such that the period map $X\to \sD/\Gamma$ is   generically finite onto the image, or
 		\item\label{item:semisimple} there is a big representation $\tau:\pi_1(X)\to \GL_N(\bC)$ such that the Zariski closure of $\tau(\pi_1(X))$ is a semisimple algebraic group.
 	\end{enumerate} Then every small projective deformation of $X$ is  pseudo Picard hyperbolic.
 \end{cor}
 \begin{proof}
 	In Case \ref{item:VHS}, by the work of Griffiths-Schmid \cite{GS69}, the period domain $\sD$ is equipped with a natural metric that has negative holomorphic sectional curvature along the horizontal direction. Since the period map of $\varrho$ is assumed to be generically finite onto its image, it follows from the Ahlfors-Schwarz lemma that $X$ is pseudo Brody hyperbolic.
 	
 	On the other hand, by \cref{lem:period}, $\varrho$ is big. By \cref{thm:Corlette}, $\varrho$ is also reductive. Hence, the conditions in \cref{main4} are fulfilled, allowing us to conclude that a small projective deformation of $X$ is  pseudo Picard hyperbolic.  
 	
 	In Case \ref{item:semisimple},  by \cref{main:hyper}, $X$ is pseudo Brody hyperbolic.   We apply \cref{main4} to conclude  that a small deformation of $X$ is  pseudo Picard hyperbolic.    
 \end{proof}
 	
	\section{Some further applications}\label{sec:application} 
In this section, we present several applications of our hyperbolicity results
  in \cref{main:hyper}. 
	\subsection{Special and $h$-special varieties}
	We first recall the definition of special varieties by Campana \cite{Cam04,Cam11}.
	\begin{dfn}[Campana's specialness]\label{def:special}
		Let $X$ be a    quasi-projective normal variety.
		\begin{thmlist}
			\item	$X$ is \emph{weakly special} if for any finite \'etale cover $\widehat{X}\to X$ and any proper birational modification $\widehat{X}'\to \widehat{X}$, there exists no  dominant morphism  $\widehat{X}'\to Y$  with connected general fibers such that $Y$ is a positive-dimensional  quasi-projective  variety of  log  general type.  
			\item  $X$ is \emph{special} if  for  any proper birational modification $X'\to X$ there is no dominant morphism $X'\to  Y$  with connected general fibers over  a positive-dimensional  quasi-projective  variety  $Y$ such that the \emph{Campana orbifold base}  (or simply orbifold base) is  of log  general type.  
			\item $X$ is \emph{Brody special} if  it contains a Zariski dense entire curve. 
		\end{thmlist}
	\end{dfn}
	%Let us briefly recall the definitions of \emph{classical orbifold base} and  \emph{orbifold base}  of a fibration $f:X\to Y$ between smooth projective varieties.    Let $|\Delta| \subset Y$ be the union of all codimension one irreducible components of the locus over which the schemetheoretic fibre of $f$ is not smooth. For each component $\Delta_i$ of $|\Delta|$, let $D_i:=\sum_{j \in J} m_{i, j} D_{i, j}$ be the union of all components of $f^* \Delta_i$ that are mapped surjectively onto $\Delta_i$ by $f$. Then one defines the multiplicity (resp. classical multiplicity)  of $f$  along $\Delta_i$ by $m_i: =\inf \left\{m_{i, j}, j \in J\right\}$ (resp. $m'_i: ={\rm gcd} \left\{m_{i, j}, j \in J\right\}$) and the $\mathbb{Q}$-divisors
	%$$
	%\Delta(f):=\sum_{i \in I}\left(1-1 / m_i\right) \Delta_i, \quad  \Delta'(f):=\sum_{i \in I}\left(1-1 / m'_i\right) \Delta_i.
	%$$
	%The pair $(X, \Delta(f))$ (resp.  $(X,\Delta'(f))$) is called the orbifold base (resp. classical orbifold base) of the fibration $f$.  One can see that $\Delta(f)\geq \Delta'(f)$.   %The fibration is said to be of general type if its orbifold base is of general type.  
	
	Campana defined $X$ to be \emph{$H$-special} if $X$ has vanishing Kobayashi pseudo-distance.
	Motivated by \cite[11.3 (5)]{Cam11}, in \cite[Definition 2.2]{CDY25b} we introduce the following definition. 
	\begin{dfn}[$h$-special]\label{defn:20230407}
		Let $X$ be a smooth quasi-projective variety.
		We define the equivalence relation $x\sim y$ of two points $x,y\in X$ iff there exists a sequence of holomorphic maps $f_1,\ldots,f_l:\mathbb C\to X$ such that letting $Z_i\subset X$ to be the Zariski closure of $f_i(\mathbb C)$, we have 
		$$x\in Z_1, Z_1\cap Z_2\not=\emptyset, \ldots, Z_{l-1}\cap Z_l\not=\emptyset, y\in Z_l.$$
		We set $R=\{ (x,y)\in X\times X; x\sim y\}$.
		We define $X$ to be \emph{hyperbolically special} ($h$-special for short) iff $R\subset X\times X$ is Zariski dense.
	\end{dfn} 
	By definition, rationally connected projective varieties are $h$-special without referring to a theorem of Campana and Winkelmann \cite{CW16}, who proved that all rationally connected projective varieties contain Zariski dense entire curves. %  It also has the following properties. 
	%\begin{lem}[{\cite[Lemmas 10.2 \& 10.4]{CDY22}}]\label{lem:preserve}
	%	\begin{thmlist}
		%	\item 	If a smooth quasi-projective variety $X$ is Brody special, then it is $h$-special.
		%	\item Let $X$ be an $h$-special smooth quasi-projective variety, and let $p:X'\to X$ be a finite \'etale morphism or proper birational morphism from a quasi-projective variety $X'$.
		%	Then $X'$ is $h$-special. \qed
		%	\end{thmlist} 
	%\end{lem}
	%\begin{proposition}[{\cite[Proposition 4.13]{CDY22}}]
	%	\label{prop:pi1}If a quasi-projective smooth variety  $X$ is special or $h$-special, the quasi-albanese map $a:X\to A$ of $X$ is dominant with  connected general fibers. Moreover, it is $\pi_1$-exact, i.e.,   we have the following exact sequence:
	%	$$\pi_1(F)\to \pi_1(X)\to \pi_1(A)\to 1,$$ where $F$ is a general fiber of $a$. \qed
	%\end{proposition}
	
	In \cite{Cam04,Cam11}, Campana proposed the following tantalizing abelianity conjecture. 
	\begin{conjecture}[Campana]\label{conj:Campana}
		A special   smooth quasi-projective variety has virtually abelian fundamental group. 
	\end{conjecture}
	%Note that in cases where ${\rm char}\, K=0$,  \cref{conj:Campana}  was proved by Campana  \cite{Cam04}   (for $X$ special) and the second author   \cite{Yam10} (for $X$ Brody special)  

	In \cite{CDY25b} we discovered that  \cref{conj:Campana} fails for a non-proper  quasi-projective variety. 
	
	\subsection{Campana's conjecture revisited}
	In \cite[Example 4.26]{CDY25b}, we constructed a smooth quasi-projective variety  such that it is both special and Brody special, yet it has nilpotent fundamental group that is not virtually abelian.  Later, Aguilar-Campana \cite{AC25}  also gave another simpler construction of such examples.   
	We first recall the following definition. 
	\begin{dfn}[nilpotent group]
		A group $G$ is \emph{nilpotent} if it  has a  central series of finite length. That is, a series of normal subgroups
		$$\{1\} = G_0 \triangleleft G_1 \triangleleft \dots \triangleleft G_n = G$$
		such that $G_{i+1}/G_i\leq Z(G/G_i)$.  For a nilpotent group $G$, the smallest $n$ such that $G$ has a central series of length $n$ is called the nilpotency class of $ G$; and $G$ is said to be \emph{nilpotent of class $n$}. 
	\end{dfn}  
	The following example is given in \cite{AC25}. 
	\begin{example}\label{example}
		Let $L$ be a  holomorphic line bundle $L$ over an elliptic curve $B$   such that $c_1(L)\neq 0$.  Let $X$ be $L^*$, that is the complement  of the zero
		and infinity sections of  $\bP(L\oplus \cO_B)\to X$. Then it is a $\bC^*$-fibration over the elliptic curve.  
		By the Gysin sequence, we have 
		\begin{align}\label{eq:Gysin}
			0\to H^1(B,\bZ)\stackrel{\pi^*}{\to} H^1(X,\bZ)\to H^0(B,\bZ)\stackrel{\cdot c_1(L)}{\to} H^2(B,\bZ)\to H^2(X,\bZ)\to H^1(B,\bZ)\to\cdots
		\end{align}   
		where $\pi:X\to B$ is the projection map.   Therefore, if $c_1(L)\neq 0$,  
		$$
		\pi^*: H^1(B,\bZ) {\to} H^1(X,\bZ)
		$$
		is an isomorphism.   It follows that $\pi_1(X)$ is then
		a central extension of $\pi_1(B)$ by $\bZ$,  hence is  torsionfree, and  nilpotent of class $2$. 
	\end{example} 
	Consequently, in the quasi-projective setting, we revised \cref{conj:Campana} as follows.
	\begin{conjecture}\label{conj:Campana2}
		A special or $h$-special smooth quasi-projective variety has virtually nilpotent fundamental group.  
	\end{conjecture} 
	We propose the following stronger conjecture. Similar questions were also independently asked in \cite{AC25, Rogov25}. 
	\begin{conjecture}\label{conj:Campana3}
		Let $X$ be a smooth quasi-projective variety that is either  {special} or  {$h$-special}. Then its fundamental group $\pi_1(X)$ is  {virtually nilpotent of class $2$}.
	\end{conjecture} 
	Shimoji has proved some interesting results on  this conjecture, see \cite{Taito25}.  Rogov \cite{Rogov25} also proposed some strategy in proving this conjecture using higher Albanese maps by Hain and o-minimal geometry. 
	\subsection{Nilpotency conjecture in the linear case}
	In \cite{CDY25b,DY23b}, we confirm \cref{conj:Campana2} for quasi-projective varieties with linear fundamental groups. 
	\begin{thm}\label{thm:CDY}
		Let $X$ be a special or $h$-special smooth quasiprojective variety. Let  $\varrho:\pi_1(X)\to \GL_{N}(K)$ be a linear representation where $K$ is any field.  
		\begin{thmlist}
			\item  \label{VN} {\rm  \cite[Theorem A]{CDY25}\ } If ${\rm char}\, K=0$, then  the image $\varrho(\pi_1(X))$ is virtually nilpotent.  
			\item \label{VR:positive}   {\rm \cite[Theorem G]{DY23b}\ }   If ${\rm char}\, K>0$, then $\varrho(\pi_1(X))$ is virtually abelian.   
		\end{thmlist}  
	\end{thm}
	By  \Cref{example}, \cref{thm:CDY} is shown to be sharp. Surprisingly, in the context of representations in positive characteristic,   we   obtain a stronger result.     
	
	\begin{proof}[Proof of  \cref{thm:CDY}]
		\noindent {\it Step 1. We prove that $\varrho(\pi_1(X))$ is  solvable.} 
	 	We may assume that $K$ is algebraically closed.  Let $G$ be the Zariski closure of $\varrho(\pi_1(X))$.  By \cite{Cam11}, any finite \'etale cover of a   special (resp. $h$-special) variety is still  special (resp. $h$-special).   After replacing $X$ by a  finite étale cover,  we may assume that $G$ is connected. Let $R(G)$ be the radical of $G$. Let $H:=G/R(G)$, which is semisimple.  If $\dim H>0$,  then $\varrho$ induces a Zariski dense representation $\sigma: \pi_1(X)\to H(K)$.  %By \cref{thm:Shafarevich}, the Shafarevich morphism associated with $\sigma$
	 %	exists and is denoted by
%$ 	\mathrm{sh}_\sigma \colon X \to \mathrm{Sh}_\sigma(X).$\footnote{Alternatively, one may use the Shafarevich map of \cite{Kol93}.}
	 	 %   Therefore,  a general fiber $F$ of ${\rm sh}_\sigma$ is connected and  $ \sigma({\rm Im}[\pi_1(F)\to \pi_1(X)])$ is finite. 
	 	  We can prove that, after replacing $X$ by a composition of birational modifications and  finite étale Galois covers,  there exists   a dominant morphism $f:X\to Y$ over  a smooth quasi-projective variety $Y$  with connected general fibers, and  a big and Zariski dense representation $\tau:\pi_1(Y)\to H(K)$ such that $\sigma=f^*\tau$.  By \cref{main:hyper}, $Y$ is of log general type and pseudo Picard hyperbolic.  
		This leads to a contradiction since $X$ is special (thus weakly special by \cite{Cam11}) or $h$-special. Hence $G=R(G)$. 
		
		\medspace
		
		\noindent {\it Step 2. We prove that $\varrho(\pi_1(X))$ is virtually abelian if ${\rm char}\, K>0$.}   Note that any finite \'etale cover of a special (resp. $h$-special) variety is still special (resp. $h$-special) by \cite{Cam04} and \cite[Lemma 3.2]{CDY25b}.  Replacing  $X$ by a finite \'etale cover, we may assume that  $\pi_1(X)^{ab}\to \pi_1(A)$ is an isomorphism, where   $\pi_1(X)^{ab}:=\pi_1(X)/[\pi_1(X),\pi_1(X)]$.  Since  $X$ is special or $h$-special, by \cite[Proposition 4.13]{CDY25b}, the quasi-albanese map $a:X\to A$ of $X$ is $\pi_1$-exact, i.e.,   we have the following exact sequence:
		$$\pi_1(F)\to \pi_1(X)\to \pi_1(A)\to 1,$$ where $F$ is a general fiber of $a$. 
		Hence $[\pi_1(X),\pi_1(X)]$ is the image of $\pi_1(F)\to \pi_1(X)$, which is thus finitely generated.  
		It implies that  $[\varrho(\pi_1(X)),\varrho(\pi_1(X))]=\varrho([\pi_1(X),\pi_1(X)])$ is also finitely generated.
		By Step 1,  $G$ is solvable. Hence we have
		$
		\cD G\subset R_u(G),
		$ 
		where $R_u(G)$ is the unipotent radical of $G$ and $\cD G$ is the  the derived group of $G$. Consequently, 
		we have	%\footnote{
			%	Shall we add the following line?:
			%Note that every subgroup of finite index in $[\pi_1(X),\pi_1(X)]$ is finitely generated (cf. \cite[Proposition 4.17]{ST00}
			%	).} 
		$$ [\varrho(\pi_1(X)),\varrho(\pi_1(X))]
		\subset [G(K),G(K)]\subset R_u(G)(K).$$  
		Note that every subgroup of finite index in $[\pi_1(X),\pi_1(X)]$ is also finitely generated (cf. \cite[Proposition 4.17]{ST00}. 
		By the same arguments in \cref{lem:finite group}, we conclude that $[\varrho(\pi_1(X)),\varrho(\pi_1(X))]$ is finite. 
		Hence $\varrho(\pi_1(X))$ is virtually abelian.  
		
		\medspace
		
		\noindent {\it Step 3. We prove that $\varrho(\pi_1(X))$ is virtually nilpotent if ${\rm char}\, K=0$ .}  
		The proof is non-trivial and based on \ref{thm:202210123} below.
	\end{proof}  
	\begin{thm}[{\cite[Theorem 11.3]{CDY22}}] \label{thm:202210123}
		Let $X$ be a special or $h$-special quasi-projective manifold.
		Let $G$ be a connected, solvable algebraic group defined over $\mathbb C$.
		Assume that there exists a Zariski dense representation $\varphi:\pi_1(X)\to G$.
		Then $G$ is nilpotent.
		In particular, $\varphi(\pi_1(X))$ is nilpotent.\qed
	\end{thm} 
The proof of \cref{thm:202210123} is involved.
It is inspired by \cite{Cam01} and is based on the $\pi_1$-exactness of the
quasi-Albanese morphism mentioned above, together with Deligne’s theorem
asserting that the radical of the algebraic monodromy group of an admissible
variation of mixed Hodge structures is \emph{unipotent}.
We refer the reader to \cite[\S~4]{CDY25b} for details of the proof.

More recently, in \cite{CDP25}, Cao, Hacon, Păun, and the  author  developed Hodge theory for local systems over quasi-projective varieties, extending our previous techniques on Hodge theory for rank-one local systems over quasi-compact Kähler manifolds \cite{CDHP25}. As a consequence, we  refine \cref{VN} by proving that $\varrho(\pi_1(X))$ is nilpotent of class~$2$. This result establishes \cref{conj:Campana3} in the case where the fundamental group is linear.
	%\section{A characterization of semi-abelian variety}
	%\begin{propx}[\cite{CDY22,DY23b}]\label{thm:char}
	%	Let $Y$ be a  smooth quasi-projective   variety, and let  $\varrho:\pi_1(Y)\to \GL_{N}(K)$ be a big representation where $K$ is a  field of  positive characteristic. If ${\rm char}\, K=0$, we assume additionally that $\varrho$ is reductive. 
	%	\begin{thmlist}
		%	\item  If $Y$ is special or $h$-special, then there exists a finite \'etale cover $X$ of $Y$, such that its Albanese map $\alpha:X\to A$ is birational and $\alpha_*:\pi_1(X)\to \pi_1(A)$ is an isomorphism.
		%\item If the logarithmic Kodaira dimension $\bar{\kappa}(Y)=0$,  then there exists a finite \'etale cover $X$ of $Y$, such that  its  Albanese map $\alpha:X\to A$ is birational and   proper in codimension one,  i.e. there exists a Zariski closed subset $Z\subset A$ of codimension at least two such that $\alpha$ is proper over $A\backslash Z$.
		%\end{thmlist}
		%\end{propx}
		\subsection{Algebraic varieties with compactifiable universal coverings}
		In the work \cite{CHK13,CH13}, Claudon, H\"oring and Koll\'ar proposed the following intriguing conjecture:
		\begin{conjecture}\label{conj}
			Let $X$ be a complex projective manifold with infinite fundamental group $\pi_1(X)$. Suppose that the universal cover $\widetilde{X}$ is quasi-projective. Then   after replacing $X$ by a  finite étale cover,   there exists a locally trivial fibration $X \rightarrow A$ with simply connected fiber  $F$ onto a complex torus $A$. In particular we have $\widetilde{X}\simeq F \times \mathbb{C}^{\operatorname{dim} A}$.
		\end{conjecture}
		It's worth noting that assuming abundance conjecture, Claudon, H\"oring and Koll\'ar proved this conjecture in \cite{CHK13}. In \cite{CH13}, Claudon-H\"oring proved \cref{conj} in the case where $\pi_1(X)$ is virtually abelian. 
		
		In this subsection we establish a linear version of \Cref{conj} without relying on the abundance conjecture.   
		\begin{thm}[{\cite[Corollary G]{DY23b}}]\label{thm:univ}
		\cref{conj} holds if  there exists a faithful  representation $\varrho:\pi_1(X)\to \GL_{N}(K),$ where $K$ is any field.
		\end{thm} 
		\begin{proof} 
By \cite{CH13}, it suffices to prove that $\pi_1(X)$ is virtually abelian. 			Consider the core map $$c_X: X \dashrightarrow Y := C(X)$$  defined by Campana, such that  the orbifold base of $c_X$ is of orbifold general type (see \cite{Cam04}). By the orbifold version of the Kobayashi-Ochiai theorem \cite[Theorem 8.2]{Cam04}, the composed meromorphic map
			\[
			\widetilde{X} \xrightarrow{\pi} X \overset{c_X}{\dashrightarrow} Y
			\]
			extends to a meromorphic map from some projective compactification $\overline{\widetilde{X}}$ of $\widetilde{X}$. This implies that the general fiber $F$ of $c_X$ satisfies that $\pi^{-1}(F)$ has only finitely many connected components. In particular, the induced homomorphism $\pi_1(F) \rightarrow \pi_1(X)$ has image of finite index.
			
			According to \cite{Cam04}, $F$ is a special manifold. Since we assume the existence of a faithful representation $\varrho: \pi_1(X) \to \GL_N(K)$,   it follows from \cref{VR:positive} when $\mathrm{char}\, K>0$ and  \cite[Theorem 7.8]{Cam04} when $\mathrm{char}\, K=0$ that the image $\varrho([\pi_1(F)\to \pi_1(X)])$ is virtually abelian. As the image of $\pi_1(F) \to \pi_1(X)$ has finite index, $\pi_1(X)$ itself is virtually abelian. The conclusion then follows from \cite[Theorem 1.5]{CH13}, completing the proof of \cref{thm:univ}.
		\end{proof} 
\begin{rem}
	 Our original proof of \cref{thm:univ} was involved and relied essentially on the
	 pseudo Picard hyperbolicity established in \cref{main:hyper}.
	 The arguments in the above proof were pointed out to us by a referee
	 of \cite{DY23b}.
	 Similar arguments were already used in \cite{CH13}.
\end{rem}

	%	\begin{lem}\label{lem:VR}
		%	Let $\Gamma\subset \GL_{N}(K)$ be a finitely generated  
		%	subgroup where $K$ is an algebraically closed field  of positive characteristic.  If $\Gamma$ is virtually nilpotent, then it is virtually abelian. %\qed
	%	\end{lem}
		
		\subsection{A structure theorem: on a conjecture by Koll\'ar}
		In \cite[Conjecture 4.18]{Kol95}, Koll\'ar raised the following conjecture on the structure of varieties with big fundamental group. 
		\begin{conjecture}\label{conj:Kollar}
			Let $X$ be a smooth projective variety with big fundamental group such that $0<\kappa(X)<\dim X$. Then $X$ has a finite \'etale cover $p:X'\to X$ such that $X'$ is birational to a smooth family of abelian varieties over a projective variety of general type $Z$ which has big fundamental group. 
		\end{conjecture} 
		In this section we address \cref{conj:Kollar}.  Our theorem is the following:
		\begin{thm}[{\cite[Theorem 5.1]{CDY25b},\cite[Corollary H]{DY23b}}]
			\label{thm:20230510}
			Let $X$ be a   smooth quasi-projective variety and let  $\varrho:\pi_1(X)\to \GL_{N}(K)$ be a big representation where $K$ is any field.  When ${\rm char}\, K=0$, we assume additionally that $\varrho$ is reductive. 
			\begin{enumerate}[label={\rm (\alph*)}]
				\item \label{item:LKD} The logarithmic Kodaira dimension $\bar{\kappa}(X)\geq 0$. 
				%\item \label{item:K0}	If the logarithmic Kodaira dimension $\bar{\kappa}(X)=0$,  then up to a finite \'etale cover,  the  Albanese map $\alpha:X\to A$ of $X$ is birational and   proper in codimension one,  i.e. there exists a Zariski closed subset $Z\subset A$ of codimension at least two such that $\alpha$ is proper over $A\backslash Z$. In particular,   $\pi_1(X)$ is virtually abelian. 
				\item \label{item:Iitaka}After replacing $X$ by a suitable finite \'etale cover and a birational modification, there are a semiabelian variety $A$, a quasi-projective manifold $V$, and a birational morphism $a:X\to V$ such that we have the following commutative diagram
				\begin{equation*}
					\begin{tikzcd}
						X \arrow[rr,    "a"] \arrow[dr,  "j"] & & V \arrow[ld, "h"]\\
						& J(X)&
					\end{tikzcd}
				\end{equation*}
				where $j$ is the logarithmic Iitaka fibration and $h:V\to J(X)$ is a  locally trivial fibration with fibers isomorphic to $A$.   Moreover, for a   general fiber $F$ of $j$, $a|_{F}:F\to  A$ is proper in codimension one. 
				%	\item \label{char abelian}  If $Y$ is special or $h$-special, then there exists a finite \'etale cover $X$ of $Y$, such that its Albanese map $\alpha:X\to A$ is birational and $\alpha_*:\pi_1(X)\to \pi_1(A)$ is an isomorphism.
			\end{enumerate} 
		\end{thm}  
		\begin{proof}
			We may assume that $K$ is algebraically closed.  To prove the theorem we are free to replace $X$ by a birational modification  and by a finite \'etale cover since the logarithmic Kodaira dimension will remain unchanged. If ${\rm char}\, K > 0$, we replace $\varrho$ by its semisimplification, which
			remains big by \cref{lem:finite group}.
			 Hence we might assume that $\varrho$ is big and semisimple. Consequently, after replacing $X$ by a finite \'etale cover,  the Zariski closure $G$ of $\varrho$ is reductive and connected.  Let $\cD G$ be the derived group of $G$, which is semisimple. Let $Z\subset G$ be the maximal central torus of $G$. Then  $T:=G/\cD G$  is a torus and  $S:=G/Z$ is semisimple.  The natural morphism $G\to  S\times T$  is a central isogeny. 
			The induced representation $\varrho': \pi_1(X)\to  S(K)\times T(K)$ by $\varrho$  is also big.   Consider the  representation $\sigma:\pi_1(X)\to S(K)$, obtained by composing $\varrho$ with the morphism $G \to S$. Then $\sigma(\pi_1(X))$ is Zariski dense.  
			
			We then prove the following factorization theorem
			(cf.~\cite[Proposition~5.9]{DY23b})%
			\footnote{The main difficulty arises in the case ${\rm char}\, K>0$, since
				Selberg's lemma—asserting that a finitely generated linear group is virtually
				torsion-free—is no longer available.}:
			there exist
			\begin{enumerate}[label={\rm (\roman*)}]
				\item a generically finite, proper, surjective morphism
				$\mu \colon X_1 \to X$ from a smooth quasi-projective variety,
				obtained as a composition of birational modifications and finite
				\'etale Galois covers;
				\item a dominant morphism $f_1 \colon X_1 \to Y_1$, where $Y_1$ is a smooth
				quasi-projective variety with connected general fibers;
				\item a big and Zariski dense representation
				$\tau \colon \pi_1(Y_1) \to S(K)$;
			\end{enumerate}
			such that $\mu^*\sigma = f_1^*\tau$.
			Note that $Y_1$ may be a point.
			If $\dim Y_1 > 0$, then by \cref{main:hyper}, $Y_1$ is strongly of log general type.

		%	Consider the morphism  
		%	\begin{align*}
			%	g:X_1&\to A \times Y_1\\
			%	x&\mapsto (\alpha(x), f_1(x)).
		%	\end{align*}
		%	where $\alpha:X_1\to A$ is the quasi Albanese map  
		%	of $X_1$. 
		 	\begin{claim}\label{claim:same}
		 		For a general smooth fiber $F$ of $f_1$,  	we have	  $\dim F=\dim \alpha(F)$
	 	\end{claim}
		 \begin{proof}
			Note that $\mu^*\sigma({\rm Im}[\pi_1(F)\to \pi_1(X)])$ is trivial. Since $\mu^*\varrho':\pi_1(X_1)\to S(K)\times T(K)$ is big, by the construction of $\sigma$, we conclude that the representation $\eta:\pi_1(F)\to T(K)$ obtained by 
				$$
				\pi_1(F)\to \pi_1(X_1)\stackrel{\mu^*\varrho'}{\to}  S(K)\times T(K)\to T(K)
				$$ 
				is big.  Since $T(K)$ is commutative,  $\eta$ factors through $\pi_1(F)\to \pi_1(A)\to T(K)$. This implies that $\dim F=\dim \alpha(F)$. %Hence $\dim X_1=\dim g(X_1)$. 
		 	\end{proof}

			Let us  prove \cref{item:LKD}.  
		By \cref{claim:same}, for a    general smooth fiber $F$ of $f_1$, $\bar{\kappa}(F)\geq 0$. Since $Y_1$ is of log general type, by the  subadditivity of the logarithmic Kodaira dimension  proven in \cite[Theorem 1.9]{Fuj17}, we obtain $$\bar{\kappa}(X_1)\geq \bar{\kappa}(Y_1)+\bar{\kappa}(F)\geq \bar{\kappa}(Y_1)=\dim Y_1\geq 0.$$  
			Hence $\bar{\kappa}(X)=\bar{\kappa}(X_1)\geq 0$. 
			The first claim is proved.

			\medspace

			We proceed to prove the second assertion. For simplicity, we assume that the logarithmic Iitaka fibration $j:X_1\to J(X_1)$ is regular. Let $X_t:=j^{-1}(t)$ for any $t\in Y_1$. 
			\begin{claim}\label{claim:20230518}
				$f_1(X_t)$ is a point for very generic $t\in J(X_1)$.
			\end{claim} 
			\begin{proof}
				Since $f_1$ is dominant and $Y_1$ is strongly of log general type, $\overline{f_1(X_t)}$ is of log general type for generic $t\in J(X)$.
				Now, let us take a very generic $t\in J(X)$.
				To show that $f_1(X_t)$ is a point, suppose for the sake of contradiction that $\dim f_1(X_t)>0$.
				Then $\bar{\kappa}(\overline{f_1(X_t)})=\dim \overline{f_1(X_t)}$.
				Since $\bar{\kappa}(X_t)=0$, the general fibers of the restriction $f_1|_{X_t}:X_t\to \overline{f_1(X_t)}$ have non-negative logarithmic Kodaira dimension. By \cite[Theorem 1.9]{Fuj17} again, it follows that $\bar{\kappa}(X_t)\geq \bar{\kappa}(\overline{f_1(X_t)})>0$. This yields a contradiction.
				Thus, $f_1(X_t)$ is a point.
			\end{proof}
			
			By \cref{claim:same,claim:20230518}, for very generic $t\in J(X_1)$, we have $\dim X_t=\dim \alpha(X_t)$. By the birational criterion for semi-abelian varieties in \cite[Proof of Lemma 1.4]{CDY25}, the closure of $\alpha(X_t)$ is a semi-abelian subvariety of $A$. As there are at most countably many semi-abelian subvarieties of $A$, we conclude that the images of the very general fibers of $f_1$ are isomorphic.  
			  This roughly shows the second assertion.    
		\end{proof}
		
		\begin{rem}
			If we assume the logarithmic abundance conjecture: a smooth quasi-projective variety is $\bA^1$-uniruled if and only if $\bar{\kappa}(X)=-\infty$, then it predicts that   $\bar{\kappa}(X)\geq 0$ if there is a big representation $\varrho:\pi_1(X)\to {\rm GL}_N(\bC)$, which is slightly stronger than the first claim in \cref{thm:20230510}. Indeed, since $\varrho$ is big, $X$ is not $\bA^1$-uniruled and thus $\bar{\kappa}(X)\geq 0$ by this conjecture.  
		\end{rem}
		%\section{Some open questions}
		
		\medspace

		\section*{Conventions and notation} In this paper, we use the following conventions and notations:
		\begin{itemize}[noitemsep]
			\item Quasi-projective varieties and their closed subvarieties are assumed to be positive-dimensional and irreducible unless specifically mentioned otherwise. Zariski closed subsets, however, may be reducible.  
			\item Fundamental groups are always referred to as topological fundamental groups.
			\item If \(X\) is a complex space, we denote by \(X^{\mathrm{norm}}\) its normalization, and by 
			\(X_{\mathrm{reg}}\) its smooth locus.
			\item $\bD$ denotes the unit disk in $\bC$, and $\bD^*$ denotes the punctured unit disk.  
			\item A finitely generated group $\Gamma$ is called \emph{linear} if it admits an almost faithful representation $\varrho:\Gamma \to \GL_N(\bC)$, i.e.\ such that $\ker \varrho$ is finite. 
			It is called \emph{reductive} if, moreover, $\varrho$ is semisimple. 
			\item For a complex algebraic variety $X$, unless otherwise specified, we denote by $\pi_X:\widetilde{X}\to X$ its universal covering map. More generally, for any normal subgroup $\Gamma\subset \pi_1(X)$, we denote by $\widetilde{X}_\Gamma\to X$ the Galois covering of $X$ with Galois group $\pi_1(X)/\Gamma$.   For any representation $\varrho:\Gamma \to \GL_N(K)$,  we denote by $\widetilde{X}_\varrho\to X$ the  Galois covering of $X$ with Galois group $\pi_1(X)/\ker\varrho$.  
			\item A proper holomorphic fibration between complex spaces, is a proper holomorphic map such that each fiber is connected. 
			\item A $\mathbb{C}$-VHS (resp. $\mathbb{Z}$-VHS) denotes a complex (resp. integral) polarized 
			variation of Hodge structure.
			\item The reductive (resp.\ linear) Shafarevich conjecture refers to the Shafarevich conjecture 
			for projective varieties with reductive (resp.\ linear) fundamental groups. 
			\item A group \(G\) is said to be \emph{virtually} \(P\) if it contains a subgroup of finite index 
			that has property \(P\).
			\item For a smooth complex quasi-projective variety $X$, unless otherwise specified, we assume that $\overline{X}$ is a smooth compactification of $X$ such that the boundary $D := \overline{X} \setminus X$ is a simple normal crossing divisor.
		\item In this paper, for any quasi-projective variety $X$,  we denote by \( M_{\mathrm{B}}(X,N) \) the Betti moduli space of representations of \( \pi_1(X) \) into \( \GL_N \). 
			 
			%	\item For an algebraic group $G$, we denote by $\cD G$ its derived group. 
			%\item For any prime number $p$, we denote by $\GL(N,\bF_p)$ the general linear group over $\bF_p$. If $K$ is a field with ${\rm char}\, K=p$, we denote by $\GL_N(K)$ its $K$-points.
			%\item For a finitely generated group $\Gamma$, any field $K$ and any representation $\varrho:\Gamma\to \GL_{N}(K)$, we denote by $\varrho^{ss}:\Gamma\to \GL_{N}(\bar{K})$ the semisimplification of $\varrho$, where $\bar{K}$ denotes the algebraic closure of $K$.  
			%	The representation 
			%	$\varrho$ is \emph{reductive} if the Zariski closure of $\varrho(\Gamma)$ is a reductive group. 
			%	We note that $\varrho^{ss}$ is a semisimple representation, thus reductive (cf. \cite[Corollary 19.18]{Mil17}).  
		\end{itemize} 
		
	\medspace
	
		\noindent\textbf{Acknowledgements.} \quad 
		
	I would like to thank all of my collaborators on the results presented in this
	paper, for many fruitful discussions and exchanges of ideas:  Damian Brotbek,
		Benoît Cadorel, Junyan Cao, Georgios Daskalopoulos, Christopher Hacon, Ludmil Katzarkov,
		Chikako Mese, Mihai Păun, Botong Wang, and Katsutoshi Yamanoi.  
		
		I also thank Daniel Barlet,  Patrick Brosnan, Yohan Brunebarbe, Frédéric Campana,
		Benoît Claudon, Hélène Esnault, Philippe Eyssidieux, Henri Guenancia, Auguste Hébert,   Bruno Klingler, János Kollár, 
		Takuro Mochizuki, Pierre Py, Guy Rousseau, Christian Schnell, Carlos Simpson,    Matei Toma and Claire Voisin for many stimulating discussions on
		non-abelian Hodge theories  and related topics over the years. Finally, I would like to thank the referee and Christopher Hacon for their very helpful remarks and suggestions, which greatly improved the paper.
		
	%	 \bibliography{biblio}
	 %	\bibliographystyle{amsalpha}
	  
	\newcommand{\etalchar}[1]{$^{#1}$}
	\providecommand{\bysame}{\leavevmode\hbox to3em{\hrulefill}\thinspace}
	\providecommand{\MR}{\relax\ifhmode\unskip\space\fi MR }
	\providecommand{\MRhref}[2]{%
		\href{http://www.ams.org/mathscinet-getitem?mr=#1}{#2}
	}
	\providecommand{\href}[2]{#2}

\end{document}